\documentclass[11pt]{amsart}

\usepackage[OT2, T1]{fontenc}
\usepackage{url}
\usepackage{amsmath}
\usepackage{stackengine}
\usepackage{array}
\usepackage{graphicx}
\usepackage{amsfonts}
\usepackage{amssymb}
\usepackage{amstext}
\usepackage{amsthm}
\usepackage{enumitem}
\usepackage{bm}
\usepackage{hyperref}
\usepackage{colonequals}
\usepackage{enumitem}
\usepackage[alphabetic,lite]{amsrefs}
\usepackage{cleveref}
\usepackage[all,cmtip]{xy}
\usepackage{fullpage}
\usepackage{appendix}
\usepackage{scalerel}
\usepackage{mathrsfs}
\usepackage{comment}
\usepackage{tikz-cd}
\usepackage{tikz}
\usepackage{galois}
\usepackage{dynkin-diagrams}
\usepackage{amssymb} 
\usepackage{subcaption}
\def\acts{\curvearrowright}
\newcommand{\subscript}[2]{$#1 _ #2$}
\numberwithin{equation}{subsection}

\newtheorem{theorem}{Theorem}

\newtheorem{manualtheoreminnerb}{Theorem}
\newenvironment{manualtheoremb}[1]{%
  \IfBlankTF{#1}
    {\renewcommand{\themanualtheoreminnerb}{\unskip}}
    {\renewcommand\themanualtheoreminnerb{#1}}%
  \manualtheoreminnerb
}{\endmanualtheoreminnerb}

\newtheorem{lemma}[theorem]{Lemma}
\newtheorem{manualtheoreminnerc}{Lemma}
\newenvironment{manualtheoremc}[1]{%
  \IfBlankTF{#1}
    {\renewcommand{\themanualtheoreminnerc}{\unskip}}
    {\renewcommand\themanualtheoreminnerc{#1}}%
  \manualtheoreminnerc
}{\endmanualtheoreminnerc}

\newtheorem{proposition}[theorem]{Proposition}
\newtheorem{manualtheoreminnerd}{Proposition}
\newenvironment{manualtheoremd}[1]{%
  \IfBlankTF{#1}
    {\renewcommand{\themanualtheoreminnerd}{\unskip}}
    {\renewcommand\themanualtheoreminnerd{#1}}%
  \manualtheoreminnerd
}{\endmanualtheoreminnerd}

\newtheorem{corollary}[theorem]{Corollary}

\newtheorem{manualtheoreminnera}{Corollary}
\newenvironment{manualtheorema}[1]{%
  \IfBlankTF{#1}
    {\renewcommand{\themanualtheoreminnera}{\unskip}}
    {\renewcommand\themanualtheoreminnera{#1}}%
  \manualtheoreminnera
}{\endmanualtheoreminnera}

\newtheorem{conj}[theorem]{Conjecture}
\newtheorem{definition}{Definition}
\theoremstyle{definition}
\newtheorem{construction}{Construction}
\newtheorem{defn}[theorem]{Definition}
\newtheorem{notation}[theorem]{Notation}
\newtheorem{ass}[theorem]{Assumption}

\theoremstyle{remark}
\newtheorem{remark}[theorem]{Remark}

\newtheorem*{claim}{Claim}

\newtheorem*{question}{Question}
\newtheorem{example}[theorem]{Example}

\numberwithin{theorem}{section}

\newcommand{\bA}{\mathbb{A}}

\newcommand{\bC}{\mathbb{C}}
\newcommand{\bD}{\mathbb{D}}

\newcommand{\bF}{\mathbb{F}}
\newcommand{\bG}{\mathbb{G}}
\newcommand{\bH}{\mathbb{H}}

\newcommand{\bK}{\mathbb{K}}
\newcommand{\bL}{\mathbb{L}}
\newcommand{\bM}{\mathbb{M}}
\newcommand{\bN}{\mathbb{N}}
\newcommand{\bP}{\mathbb{P}}
\newcommand{\bQ}{\mathbb{Q}}
\newcommand{\bR}{\mathbb{R}}
\newcommand{\bS}{\mathbb{S}}

\newcommand{\bZ}{\mathbb{Z}}

\newcommand{\bbA}{\mathbf{A}}
\newcommand{\bbB}{\mathbf{B}}

\newcommand{\bbI}{\mathbf{I}}
\newcommand{\bbJ}{\mathbf{J}}

\newcommand{\bbH}{\mathbf{H}}
\newcommand{\bbK}{\mathbf{K}}

\newcommand{\bbZ}{\mathbf{Z}}

\newcommand{\cA}{\mathscr{A}}

\newcommand{\cC}{\mathcal{C}}
\newcommand{\cD}{\mathcal{D}}

\newcommand{\cG}{\mathcal{G}}

\newcommand{\cO}{\mathcal{O}}
\newcommand{\cP}{\mathcal{P}}

\newcommand{\cX}{\mathcal{X}}

\newcommand{\Sh}{\mathcal{S}h}

\newcommand{\fm}{\mathfrak{m}}
\newcommand{\fp}{\mathfrak{p}}

\newcommand{\et}{{\text{\'et}}}
\newcommand{\dR}{{\mathrm{dR}}}
\newcommand{\cris}{{\mathrm{cris}}}

\newcommand{\ord}{{\mathrm{ord}}}
\newcommand{\can}{{\mathrm{can}}}
\newcommand{\an}{{\mathrm{an}}}

\newcommand{\la}{\langle}
\newcommand{\ra}{\rangle}

\DeclareMathOperator{\FIsoc}{\textbf{F-Isoc}}

\newcommand{\Fpbar}{\bF}
\DeclareMathOperator{\GL}{GL}
\DeclareMathOperator{\SL}{SL}
\DeclareMathOperator{\GSp}{GSp}

\DeclareMathOperator{\Spf}{Spf}

\DeclareMathOperator{\Lv}{Lv}
\DeclareMathOperator{\Gal}{Gal}
\DeclareMathOperator{\End}{End}
\DeclareMathOperator{\Hom}{Hom}
\DeclareMathOperator{\Aut}{Aut}
\DeclareMathOperator{\Lie}{Lie}
\DeclareMathOperator{\Res}{Res}
\DeclareMathOperator{\Spec}{Spec}

\DeclareMathOperator{\ad}{ad}
\DeclareMathOperator{\Span}{Span}
\DeclareMathOperator{\der}{der}
\DeclareMathOperator{\Fil}{Fil}

\DeclareMathOperator{\rk}{rk}

\DeclareMathOperator{\Frob}{Frob}
\DeclareMathOperator{\gr}{gr}
\DeclareMathOperator{\im}{im}
\DeclareMathOperator{\Mod}{Mod}

\DeclareMathOperator{\Vect}{Vect}

\DeclareMathOperator{\ch}{ch}

\DeclareMathOperator{\loc}{loc}

\DeclareMathOperator{\iso}{iso}

\DeclareMathOperator{\Def}{Def}

\DeclareMathOperator{\MTT}{MT}

\begin{document}
\title{$p$-adic monodromy and mod $p$ unlikely intersections, II}
\author{Ruofan Jiang}
\address{Dept.\ of Mathematics, University of California, Berkeley}
\email{ruofanjiang@berkeley.edu}
\begin{abstract} We study ordinary abelian schemes in characteristic $p$ and their moduli spaces from the perspective of 
char $p$ Mumford--Tate, log Ax--Lindemann, and geometric André--Oort conjectures (abbreviated as $\MTT_p$, $\mathrm{logAL}_p$ and geoAO$_p$). In this paper, we achieve multiple goals: (\textbf{A}) establish the implication $\mathrm{MT}_p\Leftrightarrow  \mathrm{logAL}_p \Rightarrow \mathrm{geoAO_p}$, and show that they all follow from the Tate conjecture for abelian varieties. The equivalence $\mathrm{MT}_p\Leftrightarrow \mathrm{logAL}_p$ is exploited from both sides, which enables us to \noindent(\textbf{B})
develop a representation theory approach to  $\mathrm{logAL}_p$ and $\mathrm{geoAO_p}$ by first establishing many cases of MT$_p$ via classical techniques, and (\textbf{C}) develop an algebraization approach to $\MTT_p$ that transcends the limitation of classical methods. In particular, we introduce ``crystalline Hodge loci'', a rigid analytic space that encodes the essential information needed for proving $\mathrm{logAL}_p$, while being very approachable via integral and relative $p$-adic Hodge theory. This enables us to prove Chai's Tate-linear conjecture (a special case of $\mathrm{logAL}_p$) for proper Tate-linear subvarieties with unramified $p$-adic monodromy. As an application, we establish $\MTT_p$ for ordinary abelian fourfolds of $p$-adic Mumford type with unramified $p$-adic monodromy. %-- a result whose characteristic 0 counterpart is unknown. 
\end{abstract}
\maketitle
\tableofcontents
\section{Introduction}
Abelian varieties are one of the central objects in arithmetic geometry. As compared to abelian varieties over number fields, abelian varieties over char $p$ function fields are less understood. For example, their char $p$ behavior, groups of $p$-torsions, structures of endomorphism algebra, and the geometry of the moduli spaces, are very different from their char $0$ counterparts. Some very powerful tools in char 0, like transcendental methods, are no longer effective. On the other hand, phenomena exclusive to char $p$, e.g., the existence of Frobenius, open up a gate for studying them in a new framework. 

This paper, being the second in a series, is devoted to the study of ordinary abelian varieties in characteristic $p$ and their moduli spaces, from the following two perspectives: (1) $l$-adic and $p$-adic monodromy and (2) ``Hodge theoretic behavior'' in the context of unlikely intersections on mod $p$ Shimura varieties, extending our previous work \cite{J23}. Problems of similar flavors (at least in their baby forms) are studies in the early works of Chai and Oort; cf. \cite{Chai-Hecke1,Chai-Heckeorbit2,Chaioort1,Oortfoliation,Ch03,Chai06} and so on. In recent years, there have also been works trying to understand the reduction behavior of abelian varieties over char $p$ global fields (cf. \cite{MST,MAT,Tay22,Jiang23,JSY}), which can be put in the framework of unlikely (or just likely) intersections. In a forthcoming work of Shankar and the author, we will explore the application of $p$-adic monodromy in Picard rank jumping and $S$-integrality problems.      

%One of the main tasks of this paper is to study the We shall soon see that, as a consequence of one of our main results, the two perspectives can be bridged, unified, and become complementary to each other.  

The concrete problems that we care about in this paper can be put into the following three categories:
\begin{enumerate}
    \item(Monodromy conjectures) char $p$ Mumford--Tate conjecture. 
    \item(Unlikely intersection conjectures) char $p$ log Ax--Lindemann conjecture, char $p$ geometric André--Oort conjecture, and Chai's Tate-linear conjecture.
    \item (Cycle conjectures) The Tate conjecture and the Hodge conjecture.  
 \end{enumerate}

Let's briefly recall the formulation of the conjectures in the first and second category.

We will use $\mathcal{A}_{g}$, $\mathscr{A}_{g}$, and $\mathscr{A}_{g,\Fpbar_q}^{\ord}$ to denote a Siegel modular variety (with sufficiently small hyperspecial level), its integral canonical model, and the ordinary locus of its mod $p$ reduction. Suppose that $X_0$ is a geometric connected smooth variety over $\mathbb{F}_q$, equipped with a morphism $f_0:X_0\rightarrow \mathscr{A}_{g,\mathbb{F}_q}^{\ord}$. Let $f$ be the base change of $f_0$ to $\Fpbar:=\overline{\Fpbar_p}$. Let $A$ be the pullback abelian scheme over $X_0$, and let $A_{\eta}$ be the pullback abelian variety over the generic point of $X_0$. 
\begin{defn}
The Mumford--Tate group of $A$, denoted by $G_{\mathrm{B}}(A)$ (or $G_{\mathrm{B}}(f)$), is the generic Mumford--Tate group of the smallest Shimura subvariety of $\mathcal{A}_{g}$ whose Zariski closure in $\mathscr{A}_{g}$ contains $\im(f)$.  The definition also makes sense for $A_\eta$ in an obvious manner: . 
\end{defn}

\begin{conj}[Char $p$ Mumford--Tate conjecture]\label{conj:MTforAg}
Let $f_0:X_0\rightarrow \mathscr{A}_{g,\mathbb{F}_q}^{\ord}$ be as above. Let $G_l(A)$ ($l\neq p$) and $G_p(A)$ be the $l$-adic étale and $p$-adic (overconvergent crystalline) monodromy of $A$. Then $G_l(A)^\circ=G_{\mathrm{B}}(A)_{\bQ_l}$, and $G_p(A)^\circ=G_{\mathrm{B}}(A)_{\bQ_p}$.\end{conj}
The conjecture also makes sense for $A_\eta$, and is equivalent to the version for $A$.  
 
The unlikely intersection conjectures, on the other hand, are assertions pertaining to the arithmetic and geometry of the moduli space $\mathscr{A}_{g,\Fpbar}$. In characteristic $p$, the space $\mathscr{A}_{g,\Fpbar}$ no longer admits the interpretation as a ``ball quotient''. The key structure we need here is the \textit{formal linearity} on ordinary mod $p$ Shimura varieties that arise from the theory of canonical coordinates: let $x\in \mathscr{A}_{g,\Fpbar}^{\ord}(\Fpbar)$, then $\mathscr{A}_{g,\Fpbar}^{/x}$ is equipped with a canonical structure of \textit{formal torus}.
\begin{defn}
  Suppose that $X\subseteq \mathscr{A}_{g,\Fpbar}$ is a generically ordinary (locally closed) subvariety.
    \begin{enumerate}
    \item We say that $X$ is \textbf{special}, if it is Zariski dense in an irreducible component of the mod $p$ reduction of a classical special subvariety in char $0$.
     \item We say that $X$ is \textbf{quasi-weakly special} (cf. \cite{J23}) if, roughly speaking, its ordinary locus $X^{\ord}$ is contained in a special subvariety $Y\subseteq \mathscr{A}_{g,{\Fpbar}}^{\ord}$ which splits as an almost product of two special subvarieties $Y_1$ and $Y_2$, such that $\dim Y_1>0$, and the Zariski closure of $X^{\ord}$ in $Y$ splits as an almost product of $Y_1$ with a subvariety of $Y_2$. 
     \item Let $x\in X^{\ord}(\Fpbar)$. We say that $X$ is \textbf{Tate-linear at $x$}, if $X^{/x}$ is a formal subtorus of $\mathscr{A}_{g,\Fpbar}^{/x}$. Chai proved that if $X$ is Tate-linear at $x$, then it is Tate-linear at any smooth ordinary point; cf. \cite[Proposition 5.3, Remark 5.3.1]{Ch03}. So we can just call $X$ \textbf{Tate-linear}\footnote{Our terminology is slightly different from \cite{Ch03}. Chai requires Tate-linear subvarieties to be closed and smooth (the closed and non-smooth ones are called ``weakly Tate-linear''), while our definition has no requirement on the closedness or smoothness of the variety.}. 
\end{enumerate} 
\end{defn}
Now let
$f:X\rightarrow \mathscr{A}_{g,\Fpbar}^{\ord}$ be a map from a smooth connected variety $X$ to the ordinary locus of $\mathscr{A}_{g,\Fpbar}$. Pick a point $x\in X(\Fpbar)$. By a slight abuse of notation, we use $\mathscr{A}_{g,\Fpbar}^{/x}$ to denote the formal neighborhood of $\mathscr{A}_{g,\Fpbar}$ at $f(x)$, and use $f^{/x}$ to denote the induced map $X^{/x}\rightarrow \mathscr{A}_{g,\Fpbar}^{/x}$.
\begin{conj}[Tate-linear conjecture]\label{conj:TTl} Suppose that $f$ is a locally closed immersion. If $X$ is Tate-linear at $x$, then $X$ is special.
\end{conj}
\begin{conj}[Char $p$ log Ax--Lindemann conjecture]\label{conj:AxSchanuel} Let $\mathscr{T}_{f,x}\subseteq \mathscr{A}_{g,\Fpbar}^{/x}$ be the smallest  formal torus through which $f^{/x}$ factors. Then there exists a special subvariety of $\mathscr{A}_{g,\Fpbar}$ whose formal germ at $f(x)$ admits $\mathscr{T}_{f,x}$ as an irreducible component. 
\end{conj}
It is clear that Conjecture~\ref{conj:AxSchanuel} recovers Conjecture~\ref{conj:TTl} as a special case. 
\begin{conj}[Char $p$ geometric André--Oort conjecture]\label{Conj:AOforSV}
Suppose that $f$ is a locally closed immersion. Then exactly one of the following happens: \begin{enumerate}
    \item $X$ is quasi-weakly special,
    \item $X$ does not contain a Zariski dense collection of positive dimensional special subvarieties. 
\end{enumerate} 
\end{conj}
To compactify the notation, the conjectures~\ref{conj:MTforAg}$\sim$\ref{Conj:AOforSV} will be denoted by $\MTT_p$, $\mathrm{Tl}_p$, $\mathrm{logAL}_p$ and $\mathrm{geoAO}_p$, respectively. We say that $\MTT_p$ holds for $f=f_{0,\Fpbar}$, if it holds for $f_0$. 

Recall that in \cite{J23}, we proved that the conjectures $\MTT_p$, $\mathrm{Tl}_p$, $\mathrm{logAL}_p$ and $\mathrm{geoAO}_p$ all hold, if the image of $f$ lies in the special fiber of a finite product of GSpin Shimua varieties. In this paper, we will consider the conjectures in more general settings. 
\subsection{Main results} All main results are put into three categories:\begin{enumerate}[label=(\alph*)]
    \item Establish $\mathrm{MT}_p\Leftrightarrow  \mathrm{logAL}_p \Rightarrow \mathrm{geoAO_p}$: Theorem~\ref{thm:connectionweb} and Corollary~\ref{cor:quickconseuq}. 
    \item Solve cases of the conjectures by first establishing MT$_p$ via classical techniques:  Theorem~\ref{thm:B}, Theorem~\ref{thm:XAndemo}, Corollary~\ref{thm:XAndemo}. This also enables us to solve the Tate conjecture for ordinary abelian fourfolds over char $p$  function fields.  
    \item Solve cases of the conjectures by first establishing logAL$_p$ via a new algebraization approach: Theorem~\ref{demothm:Tconjfourfolds} and Theorem~\ref{demothm:TTLpintegraldcs}. In particular, we prove $\mathrm{Tl}_p$ in the proper unramified setting. 
\end{enumerate}
%The results that the author likes the most are the ones in category (c): the proof is (somewhat surprisingly) connected to recent development in (integral and relative) $p$-adic Hodge theory.  

\subsubsection{Main results in category $\mathrm{(a)}$}Let $X$ be a smooth connected variety over $\Fpbar$ with a morphism $f: X\rightarrow \mathscr{A}_{g,\Fpbar}^{\ord}$. Let $f_0: X_0\rightarrow \mathscr{A}_{g,\Fpbar_q}^{\ord}$ be a finite field model of $f$, and let $A_\eta$
 be the pullback abelian variety over the generic point of $X_0$.
 
 \begin{theorem}[$\subseteq$ Theorem~\ref{thm:logALimpliesMT}+Theorem~\ref{thm:MTimpliesAO}+Proposition~\ref{prop:relationbetweencycleconjectures}]\label{thm:connectionweb}
The following are true:
 \begin{enumerate}
     \item $\MTT_p$ holds for $f$ $\Leftrightarrow$ $\mathrm{logAL}_p$ holds for $f$. 
     \item Let $f$ be a locally closed immersion. Then $\MTT_p$ holds for $f$ $\Rightarrow$ $\mathrm{geoAO}_p$ holds for $f$. 
     \item Suppose that the Tate conjecture holds for powers of $A_\eta$ as well as CM abelian varieties over number fields, then $\MTT_p$ holds. 
     \item Suppose that the Hodge conjecture holds for abelian varieties of dimension $g$, and that $\MTT_p$ holds for $f$, then the Tate conjecture holds for $A_\eta$.  
 \end{enumerate}
\end{theorem} 

\begin{center}

\tikzset{every picture/.style={line width=0.75pt}} %set default line width to 0.75pt        

\begin{tikzpicture}[x=0.75pt,y=0.75pt,yscale=-1,xscale=1]

%Shape: Arc [id:dp4112342843905239] 
\draw  [draw opacity=0] (258.53,211.75) .. controls (241.65,206.04) and (229.5,190.06) .. (229.5,171.25) .. controls (229.5,153.29) and (240.57,137.92) .. (256.26,131.59) -- (272.25,171.25) -- cycle ; \draw   (258.53,211.75) .. controls (241.65,206.04) and (229.5,190.06) .. (229.5,171.25) .. controls (229.5,153.29) and (240.57,137.92) .. (256.26,131.59) ;  
%Shape: Arc [id:dp885441581040364] 
\draw  [draw opacity=0] (285.97,130.75) .. controls (302.85,136.46) and (315,152.44) .. (315,171.25) .. controls (315,189.21) and (303.93,204.58) .. (288.24,210.91) -- (272.25,171.25) -- cycle ; \draw   (285.97,130.75) .. controls (302.85,136.46) and (315,152.44) .. (315,171.25) .. controls (315,189.21) and (303.93,204.58) .. (288.24,210.91) ;  
%Shape: Arc [id:dp8910560966322769] 
\draw  [draw opacity=0] (260.35,220.32) .. controls (217.63,217.6) and (183.25,184.22) .. (178.95,141.95) -- (266,133) -- cycle ; \draw   (260.35,220.32) .. controls (217.63,217.6) and (183.25,184.22) .. (178.95,141.95) ;  
%Shape: Arc [id:dp9985247840862141] 
\draw  [draw opacity=0] (178.68,118.35) .. controls (181.4,75.63) and (214.78,41.25) .. (257.05,36.95) -- (266,124) -- cycle ; \draw   (178.68,118.35) .. controls (181.4,75.63) and (214.78,41.25) .. (257.05,36.95) ;  
%Shape: Arc [id:dp7159727046939242] 
\draw  [draw opacity=0] (365.29,139.39) .. controls (362.57,182.11) and (329.19,216.5) .. (286.92,220.8) -- (277.97,133.75) -- cycle ; \draw   (365.29,139.39) .. controls (362.57,182.11) and (329.19,216.5) .. (286.92,220.8) ;  
%Shape: Arc [id:dp5433860596651028] 
\draw  [draw opacity=0] (255.26,217.33) .. controls (216.45,211.1) and (186.2,180.67) .. (182,142.8) -- (270,133.75) -- cycle ; \draw   (255.26,217.33) .. controls (216.45,211.1) and (186.2,180.67) .. (182,142.8) ;  
\draw   (251.32,212.48) .. controls (253.65,214.95) and (256.05,216.51) .. (258.52,217.16) .. controls (255.98,217.42) and (253.38,218.61) .. (250.7,220.71) ;
\draw   (175.49,149.4) .. controls (177.22,146.48) and (178.07,143.74) .. (178.01,141.19) .. controls (178.96,143.56) and (180.82,145.74) .. (183.57,147.73) ;
\draw   (250.24,33.6) .. controls (253.16,35.33) and (255.9,36.18) .. (258.45,36.11) .. controls (256.09,37.07) and (253.91,38.93) .. (251.91,41.68) ;
\draw   (360.31,143.24) .. controls (362.82,140.96) and (364.43,138.59) .. (365.12,136.13) .. controls (365.34,138.67) and (366.47,141.3) .. (368.52,144.01) ;
\draw   (289.45,137.2) .. controls (288.14,134.07) and (286.44,131.76) .. (284.35,130.3) .. controls (286.82,130.94) and (289.68,130.74) .. (292.92,129.71) ;
\draw   (255.47,205.79) .. controls (257.14,208.74) and (259.1,210.83) .. (261.35,212.05) .. controls (258.82,211.7) and (256,212.23) .. (252.91,213.63) ;

% Text Node
\draw (286.24,214.31) node [anchor=north east] [inner sep=0.75pt]  [font=\scriptsize]  {$MT_{p}$};
% Text Node
\draw (161,122.4) node [anchor=north west][inner sep=0.75pt]  [font=\scriptsize]  {$\mathrm{logAL}_{p}$};
% Text Node
\draw (260,30.4) node [anchor=north west][inner sep=0.75pt]  [font=\scriptsize]  {$\mathrm{Tl}_{p}$};
% Text Node
\draw (348,125.4) node [anchor=north west][inner sep=0.75pt]  [font=\scriptsize]  {$\mathrm{geoAO}_{p}$};
% Text Node
\draw (259,119.4) node [anchor=north west][inner sep=0.75pt]  [font=\scriptsize]  {$\mathrm{Tate}$};
% Text Node
\draw (274,164.4) node [anchor=north west][inner sep=0.75pt]  [font=\scriptsize]  {$+\mathrm{Hodge}$};

\end{tikzpicture}


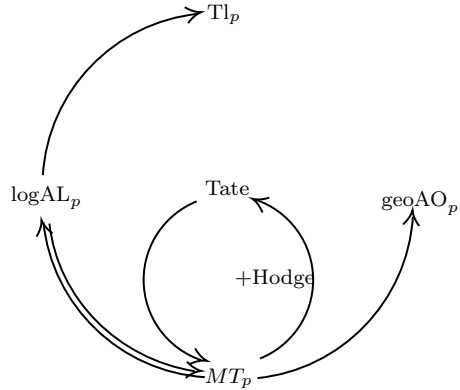
\captionof{figure}{A web of connections}\label{figure2}
\end{center}
We state quick consequences of Theorem~\ref{thm:connectionweb}: 
\begin{corollary}\label{cor:quickconseuq}
    An irreducible subvariety $X\subseteq\mathscr{A}_{g,\Fpbar}^{\ord}$ is called \textbf{formally generic}, if there exists a smooth point $x\in X(\Fpbar)$, such that  $X^{/x}$ is not contained in any proper formal subtorus of $\mathscr{A}_{g,\Fpbar}^{/x}$. Then $\MTT_p$, $\mathrm{logAL}_p$, and $\mathrm{geoAO}_p$ hold for all formally generic $X$. 
\end{corollary}
This is simply because $\mathrm{logAL}_p$ holds trivially for formally generic ones. As a consequence, $\MTT_p$, $\mathrm{logAL}_p$, and $\mathrm{geoAO}_p$ hold for ``almost all'' subvarieties of $\mathscr{A}_{g,\Fpbar}^{\ord}$.\footnote{Note that how restrictive it is for being non-formally generic: let $N=g(g+1)/2$, the tangent space $T_x\mathscr{A}_{g,\Fpbar}$ is canonically the tangent space of a formal torus, so the tangent space of any proper formal subtorus is contained in a union of $\# \bP^{N}(\Fpbar_p)=\frac{p^{N+1}-1}{p-1}$ fixed codimension 1 subspaces of $T_x\mathscr{A}_{g,\Fpbar}$. If $X$ is smooth at $x$ and is not formally generic, then $T_xX$ must be contained in one of the finitely many fixed subspaces. Using a Bertini type argument, one can easily construct an abundance of formally generic subvarieties, and show that almost all subvarieties of $\mathscr{A}_{g,\Fpbar}^{\ord}$ are formally generic.}

Combining Markman's result on the Hodge conjecture for abelian fourfolds (cf. \cite{Markman}) and Theorem~\ref{thm:B} which we will discuss below, we deduce that
\begin{corollary}[=Corollary~\ref{cor:Ttconj}]
    Let $A$ be an abelian fourfold over a finitely generated field $L$. Then the Tate conjecture holds for $A$ in the following two cases: 
    \begin{enumerate}
        \item $\mathrm{char}\, L=0$,
        \item $\mathrm{char}\, L>0$ and $A$ is ordinary. 
    \end{enumerate}
\end{corollary}

The proof of Theorem~\ref{thm:connectionweb} is a combination of the techniques from \cite{J23} and Mumford--Tate pairs \cite{P98}. Roughly speaking, Mumford--Tate pairs restrict the possible types of the monodromy representation, while \cite{J23} tells that a certain unipotent of the $p$-adic monodromy group encodes ``the same'' amount of information coming from deformation theory.  

The equivalence $\mathrm{logAL}_p\Leftrightarrow \MTT_p$ is a bridge between  problems of different nature: on one side, the problem is representation/Tannakian theoretic; on the other side, the problem lies in the realm of algebraization/unlikely intersection. The equivalence between two sides makes it possible to attack $\mathrm{logAL}_p$ and $\MTT_p$ using the techniques from both sides. To the author's best knowledge, no such connection exists in characteristic 0. 

The results in category (b) are obtained by exploiting $\mathrm{MT}_p\Rightarrow \mathrm{logAL}_p$.

The results in category (c) are obtained by exploiting $\mathrm{logAL}_p\Rightarrow\mathrm{MT}_p$.

%If we view Figure~\ref{figure2}as a ``macadamia nut'', the equivalence $\mathrm{logAL}_p\Leftrightarrow \MTT_p$ provides us with two ways to crack it open: the results in category (b) are obtained by cracking the nut from the $\mathrm{MT}_p$ side, while the results in category (c) are obtained by cracking the nut from the $\mathrm{logAL}_p$ side. 
 
\subsubsection{Main results in category $\mathrm{(b)}$}\label{subsub:exploitMT} Let's call the classical Mumford--Tate conjecture over number field $\MTT_0$. With some effort, many classical techniques for $\MTT_0$ can be adapted to char $p$. Our guideline is the following \\[5pt]
\noindent\textit{Expectation: whenever a case of $\MTT_0$ is known, a char $p$ variant can be proved.} \\
%\begin{theorem}[$=$ {Theorem~\ref{modpCommelin}}]\label{thm:modpCommelindemo}Let $A_1,A_2$ be ordinary abelian varieties over a function field $k/\Fpbar_q$. Suppose that $\MTT_p$ holds for $A_1$ and $A_2$, then it holds for $A_1\times A_2$. \end{theorem}This is a mod $p$ analogue of a theorem by Commelin; cf. \cite{JMC}. We've also made successful adaption of multiple results from  \cite{Tank83,MZ95,P98}:
\begin{theorem}[$=${Theorem~\ref{modpPink}}+{Theorem~\ref{modpTank}}+{Theorem~\ref{modpMoonen-Zarhin}+Theorem~\ref{modpCommelin}}]\label{thm:B}%\label{demo:MTcasesclassical}
    Let $A$ be an ordinary abelian variety over a function field $k/\Fpbar_q$. Then $\MTT_p$ holds for $A$, if \begin{enumerate}
        \item $A$ is of Pink type, see  Definition~\ref{def:pinktype},
        \item $A$ is geometrically simple of prime dimension,
        \item\label{it:demo:MTcasesclassical3} $A$ is an abelian fourfold \textbf{not} of $p$-adic Mumford type, see Definition~\ref{def:cohomologicalMType}.
        \item If $A$ is isogenous to a product $A_1\times A_2$ and that $\MTT_p$ holds for $A_1$ and $A_2$. 
    \end{enumerate}
\end{theorem}
These are mod $p$ analogues of the classical results from  \cite{Tank83,MZ95,P98,JMC}. 

The mod $p$ adaption truly shines when we apply it, via  Theorem~\ref{thm:connectionweb}, to the realm of unlikely intersections: 
\begin{corollary}[$=$ {Corollary~\ref{cor:Commelinconsequence}+Corollary~\ref{cor:geoAOdim5}}]\label{thm:XAndemo}~\begin{enumerate}
    \item Suppose $f:X\hookrightarrow \mathscr{A}_{g_1,\Fpbar}\times \mathscr{A}_{g_2,\Fpbar}\times ...\times \mathscr{A}_{g_n,\Fpbar}$ is a locally closed immersion into the ordinary locus. 
   If $\mathrm{logAL}_p$ holds for each factor $f_i:X\rightarrow \mathscr{A}_{g_i,\Fpbar}$, then $\mathrm{logAL}_p$ and $\mathrm{geoAO}_p$ hold for $X$.  
    \item  Suppose $f:X\hookrightarrow \mathscr{A}_{g,\Fpbar}$ is a locally closed immersion into the ordinary locus. If $g\leq 5$, then $\mathrm{geoAO}_p$ holds for $X$.
\end{enumerate}
\end{corollary}
We remark that in (1) of the corollary,  $\mathrm{logAL}_p$ holds trivially if $f_i$ dominates a special subvariety of $\mathscr{A}_{g_i,\Fpbar}$ and in (2) of the corollary, it is also possible to generalize the theorem to $g>5$, if one requires that $\dim X$ is not too small. 

Finally, we give the definition of abelian varieties of Pink type and Mumford type. Let $A$ be an abelian variety over a field $k$ finite generated over $\bQ$ or $\Fpbar_p$. If $\mathrm{char}\, k=p$, we further assume that $A$ is ordinary. 
\begin{defn}[Pink type]\label{def:pinktype}
We say that $A$ is of \textbf{Pink type}, if  $\End(A_{\overline{k}})=\bZ$, and one of the following two conditions hold: 
\begin{enumerate}[label=(\subscript{P}{{\arabic*}})]
\item\label{it:P1} 
$2\dim A$ is not a $r$-th power of an integer for $r>1$, nor of the form $\binom{2r}{r}$ for some odd $r>1$.
\item\label{it:P2} $G_\mathrm{B}(A)$ has Lie type $A_{2s-1}$ for $s\geq 1$ or $B_r$ for $r\geq 1$.
\end{enumerate}
\end{defn}
\begin{defn}[Mumford type]\label{def:cohomologicalMType}
     We say that an abelian fourfold $A$ is of \textbf{Mumford type}, if $(G_\mathrm{B}(A)_{{\bC}}^{\circ,\der},\rho_\mathrm{B})$ is the tensor product of three copies of standard representation of $\SL_{2,{\bC}}$. Similarly, it is of \textbf{$p$-adic Mumford type}, if 
$(G_p(A)_{\overline{\bQ}_p}^{\circ,\der},\rho_p)$ is the tensor product of three copies of standard representation of $\SL_{2,\overline{\bQ}_p}$. 
\end{defn}

\subsubsection{Main results in category $\mathrm{(c)}$}\label{subsub:exploitAL} The existing techniques for $\MTT_0$ have obvious limitation, and the conjecture is still open even for abelian fourfolds of $p$-adic Mumford type. If we only attack $\MTT_p$ by adapting char $0$ methods, we can go no further than our ancestors. The results we present in this section are built on new ideas. We will see that, in contrast to its char $0$ counterpart, $\MTT_p$ for $p$-adic Mumford type abelian fourfolds is indeed within reach. 

Let $A$ be an abelian fourfold over a function field $k/\Fpbar_q$. Without loss of generality, we can assume that $k$ is a global field. 
\begin{theorem}[$=$ Corollary~\ref{cor:coomologyMT}]\label{demothm:Tconjfourfolds} Suppose that $A$ is of $p$-adic Mumford type and has unramified $p$-adic monodromy $G_p(A)$ \footnote{Note that abelian fourfolds satisfying these conditions do exist: one can find an example by taking the reduction of a suitable Mumford curve.}, then $\MTT_p$ holds for $A$.
\end{theorem}
To the author's best knowledge, no result analogues to this is known for $\MTT_0$.

For the proof strategy: Theorem~\ref{demothm:Tconjfourfolds} reduces to $\mathrm{logAL}_p$ and, (with some more effort) further reduces to the Tate-linear conjecture  for compact Tate-linear curves in $\mathscr{A}_{4,\Fpbar}$. This enables us to conclude Theorem~\ref{demothm:Tconjfourfolds} as the $g=4$ case of the following result on the Tate-linear conjecture: 
\begin{theorem}[$=$Theorem~\ref{thm:TTLpintegraldcs}]\label{demothm:TTLpintegraldcs}
  Suppose that $f:T\hookrightarrow \mathscr{A}_{g,\Fpbar}$ is a proper Tate-linear subvariety with unramified $p$-adic monodromy $G_p(f)$. Then $\mathrm{Tl}_{p}$ holds for $T$.
\end{theorem}
 We will discuss the strategy for Theorem~\ref{demothm:TTLpintegraldcs} in detail in the next section. %The technical assumption ``everywhere potentially good reduction'' is likely not a serious one: it is possible to have it removed via log geometry. The condition ``unramified'' is crucial to our treatment. %Roughly speaking, we convert it to a study of the geometry of a certain rigid analytic subspace in the $p$-adic Shimura variety $\mathcal{A}_{g,K}^{\mathrm{an}}$. 
\begin{remark}
   It seems that $\MTT_p$ is more approachable than $\MTT_0$. It is reasonable to expect that the solution of $\MTT_0$ secretly relies on $\MTT_p$ -- similar to other well-known stories. % Our approach to  $\mathrm{logAL}_p\Leftrightarrow \MTT_p$ is a char $p$ phenomenon. It is not clear if a similar strategy works for $\MTT_0$. One may also wonder if the perfectoid tilting techniques are able to transfer $\MTT_p$ to $\MTT_0$. The answer is \textbf{unlikely}, because $\MTT_0$ is of a global nature. However, it is not completely unreasonable to expect that the solution of $\MTT_0$ relies on $\MTT_p$ -- similar to other well-known stories. 
\end{remark}

\subsection{Tate-linear subvarieties and the crystalline Hodge loci}
Our treatment of the Tate-linear conjecture (Theorem~\ref{demothm:TTLpintegraldcs}) involves the study of a new type of geometric objects, called the \textbf{Tate-linear loci}, which is a crystalline analogue of the Hodge loci (there is a relevant and more general construction that we termed as \textbf{crystalline Hodge loci}, see Remark~\ref{rmk:crystallineHodgeloci}). 

\subsubsection{Tate-linear loci} 
Let $\mathscr{A}^{\Sigma}_{g}$ be a fixed integral toroidal compactification. Suppose that $T\subseteq\mathscr{A}^{\Sigma}_{g,\Fpbar}$ is a closed, generically ordinary Tate-linear subvariety. Let $T^{\circ}\subseteq T$ be its smooth ordinary locus. Using global canonical coordinates, Chai proved that one can canonically lift $T^{\circ}$ to a formally linear formal subscheme $\widetilde{T^{\circ}}\subseteq (\mathscr{A}_{g,W}^{\Sigma})^{\wedge p}$ (\cite[Proposition 5.5]{Ch03}). Chai then asked if $\widetilde{T^{\circ}}$ can be extended to a closed formal subscheme of $(\mathscr{A}_{g,W}^{\Sigma})^{\wedge p}$ (\textit{loc.cit} Remark 7.2.1, we slightly generalize the setting here). If this is true, then $\widetilde{T^{\circ}}$ is algebraic by Grothendieck's GFGA, and $\mathrm{Tl}_p$ for $T$ follows from Moonen's result \cite[Theorem 4.5]{M98}.

Unfortunately, Chai's canonical lifting method does not extend to whole $T$: the variety $T$ itself may contain non-smooth, non-ordinary, and bad reduction points, where the theory of canonical coordinates may not be effective. This is a key obstruction to the approach of Conjecture~\ref{conj:TTl} via lifting techniques. Our construction of the Tate-linear loci is a partial remedy of this. In the following we give the construction in a special case. We will write $\bH_{\mathrm{dR}}$ for the first relative de Rham cohomology on $\mathscr{A}_{g}$ and write $\bH_{\cris}$ for the first relative crystalline cohomology on $\mathscr{A}_{g,\Fpbar_p}$.

%We first introduce the notion of Tate-linear loci, which serves as a $p$-adic analogue of the Hodge loci. The construction given below is not the most general one, but already encodes the key ideas: 

\begin{construction}[Tate-linear loci]\label{const:ttloci}  
Let $T\subseteq\mathscr{A}_{g,\Fpbar}$ be a smooth but not necessarily proper  Tate-linear subvariety and let  ${\mathscr{A}}_{g,W}^{/T}$ be the formal completion of ${\mathscr{A}}_{g,W}$ along $T$. The tubular neighborhood of $T$ in ${\mathscr{A}}_{g,W}^{\wedge p}$, denoted by 
$]T[$, is the rigid generic fiber $({\mathscr{A}}_{g,W}^{/T})^{\mathrm{rig}}$. From the work of Berthelot and Ogus; cf. \cite{Berthlot86} and \cite{OgusII}, we can canonically identify the convergent $F$-isocrystal $\bH_{\cris,\mathscr{A}_{g,\Fpbar}}$ with the de Rham bundle $\bH_{\dR,(\mathscr{A}_{g,W}^{\wedge p})^{\mathrm{rig}}}$, together with a Frobenius structure. Note that $\bH_{\dR,(\mathscr{A}_{g,W}^{\wedge p})^{\mathrm{rig}}}$
is further equipped with a two-step descending Hodge filtration $\mathrm{Fil}^\bullet$. Via pullback to $]T[$, we can canonically identify the convergent $F$-isocrystal $\bH_{\cris,T}$ as the de Rham bundle $\bH_{\dR,]T[}$, together with a Frobenius structure. The bundle  $\bH_{\dR,]T[}$ is again equipped with the pullback Hodge filtration $\mathrm{Fil}^\bullet$. Now consider a finite collection of crystalline tensors $\{s_\alpha\}$ cutting out the overconvergent monodromy group $G(\bH_{\cris,T})^{\circ}$. Each $s_\alpha$ gives rise to a ``multi-valued horizontal section'' $s_\alpha^{\mathrm{an}}$ of an element $\bM_\alpha\in \bH_{\dR,]T[}^{\otimes}$ (which is a filtered analytic vector bundle with connection on $]T[$, with analytic subbundle $\Fil^0 \bM_\alpha$), which becomes an actual horizontal section of $\bM_\alpha$ over a finite étale cover $\pi:\mathfrak{V}\rightarrow ]T[$ \footnote{An explicit construction of the étale cover is as follows: Let $f:X\rightarrow T$ be a connected finite étale cover so that $G(\bH_{\cris,X})=G(\bH_{\cris,X})^{\circ}$. Since taking formal thickening has no effect on the étale site, there is a unique formal scheme étale over ${\mathscr{A}}_{g,W}^{/T}$, with special fiber $X$. We will denote this formal scheme by ${\mathscr{A}}_{g,W}^{/X}$. Let $]X[$ be the rigid generic fiber of $\mathscr{Z}$. Take $\pi$ to be $]X[\rightarrow ]T[$.}. Consider the locus 
$\mathfrak{I}_\alpha\subseteq ]T[$ over which $s_\alpha^{\mathrm{an}}\in  \Fil^0 \bM_\alpha$, and take $\mathfrak{T}=\bigcap_{\alpha} \mathfrak{I}_\alpha$. This is a locally Noetherian Zariski closed analytic subspace of $]T[$. The locus $\mathfrak{T}$ is independent of the choice made, and is called the \textbf{Tate-linear locus} associated to $T$. Any irreducible component $\mathfrak{T}_0\subseteq \mathfrak{T}$ in the sense of \cite{Conrad} is called an irreducible Tate-linear locus associated to $T$. 
%Let $\mathfrak{T}$ be its image under the étale map $]X[\rightarrow ]T[$, which is independent of the choice of $\{s_\alpha\}$ or $X$. 
\end{construction}
 The construction is the most interesting when $T$ contains non-ordinary points (e.g., $T$ is proper). However, even when $T$ lies completely in the ordinary stratum, the locus $\mathfrak{T}$ contains information about the non-ordinary and boundary loci of the closure $T^{\mathrm{cl}}\subseteq \mathscr{A}_{g,\Fpbar}^\Sigma$: indeed, $\mathfrak{T}$ is ``overconvergent'' in nature, and extends to a strict neighborhood of $]T[$ in $]T^{\mathrm{cl}}[$. One key fact (Corollary~\ref{cor:containTcan}) is that there is a unique irreducible Tate-linear locus $\mathfrak{T}_{\mathrm{can}}$ with the property that \begin{equation}\label{eq:Tcan}
\mathfrak{T}_{\mathrm{can}}\cap ]T^\circ[ \text{ contains }(\widetilde{T^\circ})^{\mathrm{rig}} \text{ as an irreducible component}.
\end{equation} 
 The proof of this fact essentially relies on Crew's parabolicity conjecture (now a theorem of D'Addezio), which is not available until recently. Thanks to this, we can extend Chai's Tate-linear canonical lifting towards non-ordinary loci, and we have\\

 \noindent\textit{Observation: If $\mathfrak{T}_{\mathrm{can}}$ is algebraizable (i.e., is a component of the restriction of the analytification of an algebraic subvariety of $\mathcal{A}_{g,K}$ to $]T[$), then $\mathrm{Tl}_p$ holds for $T$; cf. Lemma~\ref{lm:moonenlemma}.}

\begin{remark}\label{rmk:crystallineHodgeloci}
Construction~\ref{const:ttloci} immediately generalizes to the setting where $T$ is smooth but not Tate-linear. And we again call $\mathfrak{T}\subseteq ]T[$ that arises from the construction the \textbf{crystalline Hodge loci} associated to $T$. As compared to the classical Hodge loci, crystalline Hodge loci are less global: they only live in the tubular neighborhood of $T$. Nevertheless, we (conjecturally) expect them to be algebraizable: more precisely, an irreducible crystalline Hodge locus is (a component of) the intersection of $]T[$ with the analytification of a special subvariety.  Using an argument similar to Lemma~\ref{lm:moonenlemma}, it is not hard to see that the algebraicity conjecture for crystalline Hodge loci associated to $T$ is equivalent to $\mathrm{logAL}_p$ for $T$.
\end{remark}
\subsubsection{Geometry of the Tate-linear loci} The algebraization problem of a Tate-linear locus is closely related to its geometry.

How can one understand the geometry of $\mathfrak{T}_{\mathrm{can}}$?  Let's take a pointwise scrutiny: in the residue disk of an ordinary $\Fpbar$-point $x$, it is possible to show that $\mathfrak{T}_{\mathrm{can}}$ is a union of the rigid generic fibers of torsion translates of subtori of the formal torus $\mathscr{A}_{g,W}^{/x}$; cf. Theorem~\ref{thm:localstructureofT}; however, in the residue disk of a non-ordinary point, it can be much wilder. Besides all that has been said, it is a natural idea to impose \textit{integrality conditions} on the $p$-adic monodromy of $T$ (e.g., being ``unramified'' as in Theorem~\ref{demothm:TTLpintegraldcs}) in order to tame the behavior of $\mathfrak{T}_{\mathrm{can}}$, especially in the residue disk of a non-ordinary point. 

Indeed, the connection between the integrality of the $p$-adic monodromy and the geometry of $\mathfrak{T}_{\mathrm{can}}$ is analogues to the connection between the hyperspecial condition of the level group and the geometry of an integral Shimura variety. In \cite{KM09}, Kisin studies the local geometry of a certain integral Shimura subvariety $\mathscr{S}_W\subseteq\mathscr{A}_{g,W}$ with hyperspecial level structure. In the formal neighborhood of an arbitrary $\Fpbar$-point $x\in \mathscr{A}_{g,\Fpbar}$, there is a type of \textbf{Lie theoretic coordinates} (in contrast to the canonical coordinates) that comes from the explicit deformation of $p$-divisible groups with Tate classes (cf. \cite[\S 7]{Fal99}, \cite[\S 4]{Moo98} and \cite[\S1.4-\S 1.5]{KM09}). Using Breuil--Kisin modules, Kisin is able to show that $\mathscr{S}^{/x}_W$ is a union of \textit{linear subspaces} in the Lie theoretic coordinates. In particular, $\mathscr{S}_W$ is smooth up to a normalization.

We briefly explain how this inspires a strategy to tackle the algebraization problem for $\mathfrak{T}_{\mathrm{can}}$, with a view towards the proof of Theorem~\ref{demothm:TTLpintegraldcs}. 

First, we show that if the 
$p$-adic monodromy of $T$ is unramified, then (up to replacing $T$ by a Hecke translate) the canonical Tate-linear locus  $\mathfrak{T}_{\mathrm{can}}$ is \textit{linear} in the Lie theoretic coordinates in the residue disk of every point of $x\in T(\Fpbar)$. Needless to say, Breuil--Kisin theory plays a key role in establishing the linearity. However, a major difference is that Kisin starts with a collection of $p$-adic étale cycles (arising from Hodge cycles) with good integral behavior, which he can feed in the Breuil--Kisin theory, while in our case we only have a collection of crystalline cycles  with good integral behavior. The particular definition of $\mathfrak{T}_{\mathrm{can}}$, however, equips us with a collection of de Rham cycles compatible with the crystalline ones. These enable us to construct $p$-adic étale cycles with good integral behavior via recent development on (relative) integral $p$-adic Hodge theory, e.g., \cite{du2024logprismaticfcrystalspurity}. Of course, \cite{du2024logprismaticfcrystalspurity} only applies to semistable formal schemes, while we do not \textit{a priori} know that $\mathfrak{T}_{\mathrm{can}}$ admits a nice formal model. This is where lots of technicalities from $p$-adic geometry kick in:  working piece by piece, normalizing and throwing out singularity, using Raynaud's theory to find a formal model, formally alterating the model (cf. \cite{TM17}), etc.

Why $p$-adic étale cycles? They are important for two reasons: globally, they propagate nice integral structures from one point to other points; locally at each $\overline{K}$-point, they feed into Breuil--Kisin machinery and translates nice integral structures into nice geometry. 

After one shows that $\mathfrak{T}_{\mathrm{can}}$ is {linear} in the residue disk of every point of $x\in T(\Fpbar)$, the strategy is to show that $\mathfrak{T}_{\mathrm{can}}$ is contained in a union of finitely many tubes of smaller radius (i.e., tubes of the form $]U[_{<1-\epsilon}$, where $U\subseteq T$ is Zariski open). When $T$ is proper, this condition implies that $\mathfrak{T}_{\mathrm{can}}$ is Zariski closed in $(\mathcal{A}_{g,K}^{\Sigma})^{\an}$: a situation that enables us to apply rigid GAGA. To avoid too much technicality, we narrow ourselves down to the proper case, where the proof is relatively clean. In future works, we will extend the results to full generality.

\subsection{Organization of the paper} Section~\ref{sec:prelim} is devoted for background knowledge and convention setup. We prove Theorem~\ref{thm:connectionweb} in section~\ref{sec:ProofMT} and section~\ref{sec:ProofAO}. The main results in category (b) will be proved in section~\ref{Sec:classical}, and the  main results in category (c) will be proved in section~\ref{sec:Newmethod}.   
\subsection*{AI Usage} No AI was used in the first version of this paper, which was submitted to arXiv in Nov. 2025. GPT6 Astra is used in the second version (current version) for:
\begin{enumerate}
    \item Proof-reading. No fatal mistake is found. The main technical issues that GPT6 Astra finds in the first version are (1) Lemma 4.8 is wrong as stated hence the proof of Theorem 4.9 is incomplete and (2) the $W$-lattices in the proof of Theorem 6.1 are not self-dual under polarization. Both issues admit easy fix and are fixed locally by the author. %The fix of the second issue uses something similar to ``Zarhin's trick'' which he learns from GPT6 Astra. 
    \item Assisting technical lemmas on rigid geometry. The assisted parts are Lemma~\ref{lm:specialization_lemma} and Construction~\ref{const:graph_ttloci}, where the author uses \textbf{étale-envelope} (\cite{etale_evelope}) that GPT6 Astra brings to his attention.  
\end{enumerate}

\subsection*{Acknowledgements} The author thanks 
Brian Conrad, Haoyang Guo, Jiaqi Hou, Zeyu Liu, Martin Olsson, Abhishek Oswal, Ananth Shankar, Sug Woo Shin, Yunqing Tang, Ziquan Yang and Alex Youcis for valuable discussions. The author thanks Yunqing Tang for feedback on the first version of this paper that in particular leads to a strengthening of Theorem~\ref{demothm:Tconjfourfolds}. The author also thanks Hanson Hao and Alan Peng for bringing him happiness when writing up the paper. 
\section{Preliminaries}\label{sec:prelim}
We review basic notions of Shimura varieties, local coordinates on Shimura varieties, coefficient objects, the mod $p$ Mumford--Tate conjecture and cycle conjectures. 
\subsection{Shimura varieties of Hodge type and special subvarieties}\label{Sec:GSpinDef}
Let $H$ be a self-dual symplectic $\bZ_{(p)}$-lattice. One can form a Siegel Shimura datum $(\GSp(H_\bQ),\mathfrak{H}^{\pm})$ with reflex field $\bQ$. Let $\bK=\bK_p{\bK}^p$ be a hyperspecial level structure such that $\bK_p=\GSp(H)(\bZ_p)$, ${\bK}^p\subseteq \GSp(H_{\bQ})(\bA_f^p)$ compact open, and $\bK$ leaves $H_{\widehat{\bZ}}$ stable. For $\bK'$ sufficiently small, we get a smooth variety $\mathcal{A}_{g,\bK}:=\Sh_{\bK}(\GSp(H_\bQ),\mathfrak{H}^{\pm})$ over $\bQ$, called the Siegel modular variety, which admits an integral canonical model $\mathscr{A}_{g,\bK}$ over $\bZ_{(p)}$ (\cite[Theorem 2.10]{Mil92}). 

When $\bK$ is sufficiently small, $H$ gives rise to a universal abelian scheme over $\cA_{g,\bK}$ (up to prime-to-$p$ isogeny), whose first $\bZ_{(p)}$-Betti cohomology is the local system induced by $H$. The abelian scheme can be extended to $\mathscr{A}_{g,\bK}$, and we denote by $A\rightarrow\mathscr{A}_{g,\bK}$ the universal abelian scheme over  $\mathscr{A}_{g,\bK}$. Let $\bbH_\text{B},\bbH_{\text{dR}}, \bbH_{l,\text{ét}}$ be the integral Betti, de Rham, $l$-adic étale ($l\neq p$) relative first cohomology of $A\rightarrow \mathscr{A}_{g,\bK}$, and let $\bbH_{\text{cris}}$ be the first integral crystalline cohomology of $A_{\mathbb{F}_p}\rightarrow \mathscr{A}_{g,\bK,\Fpbar_p}$. We also denote by $\bH_\text{B},\bH_{\text{dR}}, \bH_{l,\text{ét}}$ and $\bH_{\text{cris}}$ the rational cohomologies. 
\subsubsection{Special subvarieties and naïve integral models}\label{subsub:naivespecial} 
A Shimura datum $(G,\mathcal{D})$ is of \textbf{Hodge type}, if there is a Siegel Shimura datum $(\GSp(H_\bQ),\mathfrak{H}^{\pm})$, and an embedding of Shimura data $\iota:(G,\mathcal{D})\hookrightarrow (\GSp(H_\bQ),\mathfrak{H}^{\pm})$. Let $E$ be the reflex field of $(G,\mathcal{D})$. For a suitable level structure $\bbK\subseteq G(\bA_f)$ such that $\iota(\bbK)\subseteq \bK$, we get a morphism of Shimura varieties. 
\begin{equation}\label{eq:shimuraHodgemorphi}
\Sh_{\iota,\bbK,\bK}:\Sh_{\bbK}(G,\mathcal{D})\rightarrow \mathcal{A}_{g,\bK,E}
\end{equation}
which is an embedding for suitable choice of $\bbK$ and $\bK$.

% Suppose that there is a reductive group $G_{\bZ_{(p)}}$ over $\bZ_{(p)}$ with generic fiber $G$, and let $\bbK_p=G_{\bZ_{(p)}}(\bZ_p)$ be the associated hyperspecial subgroup. Following \cite[\S1.3.3]{MKmodp}, up to embedding in a larger Siegel Shimura datum (using Zarhin's trick), one can assume that $\iota$ is induced by an embedding $G_{\bZ_{(p)}}\hookrightarrow \GL(H)$, where $H$ is a self-dual symplectic $\bZ_{(p)}$-lattice. Let $\bK_p=\GSp(H)(\bZ_p)$, and let $\mathfrak{p}$ be a prime of $O_E$ lying above $p$. From \cite[\S 2.3.2]{KM09}, for sufficiently small $\bbK^p\subseteq G(\bA_f^p)$, there is a sufficiently small $\bK^p\subseteq \GSp(H_\bQ)(\bA_f^p)$, so that there is an embedding $\mathcal{S}_\bbK:=\Sh_\bbK(G,\mathcal{D})\hookrightarrow \cA_{g,\bK,E}$, such that the normalization of the Zariski closure of $\mathcal{S}_\bbK$ in $\mathscr{A}_{g,\bK,O_{E,(\mathfrak{p})}}$ is the integral canonical model of $\mathcal{S}_\bbK$. We will denote it by $\mathscr{S}_\bbK$. In \cite{XU20}, it is shown that the normalization is redundant. We can pullback the universal abelian scheme and local systems from $\mathscr{A}_{g,\bK,O_{E,(\mathfrak{p})}}$ to $\mathscr{S}_\bbK$. These will again be called the universal family and local systems over $\mathscr{S}_\bbK$ with respect to the embedding. 

A \textbf{special subvariety of $\mathcal{A}_{g,\bK}$ of Hodge type} is (a component of) the Hecke translate of the image of a map of the form  (\ref{eq:shimuraHodgemorphi}). There is also a notion of \textbf{weakly special subvariety} which is, roughly speaking, the product of a special subvariety with a possibly non-CM point. This notion will not be used elsewhere in the paper. We refer the readers to \cite[Definition 2.5]{M97}, \cite[\S1]{M98} for more details.  %A special subvariety may not be smooth, if the level structures are not well chosen. 

Now let's assume that $\bK$ is a sufficiently small hyperspecial level structure, so we have an ambient integral canonical model $\mathscr{A}_{g,\bK}$. Let $k$ be either a number field or $\overline{\bQ}$, and let $\mathfrak{p}$ be a prime of $O_k$ that lies above $p$. Let $\mathcal{X}\subseteq \cA_{g,\bK,k}$ be a special subvariety. The \textbf{naïve integral model} of $\mathcal{X}$ is the Zariski closure of $\mathcal{X}$ inside $\mathscr{A}_{g,\bK,O_{k,(\mathfrak{p})}}$, denoted by $\mathscr{X}$.  There is a notion of universal abelian scheme and various coefficient sheaves $\bbH_{\bullet}$ over $\mathscr{X}$ with respect to the embedding. They are just the pullback of the corresponding objects from $\mathscr{A}_{g,\bK}$. 

\subsubsection{Special subvarieties in the mod $p$ fiber} Consider an ambient integral canonical model $\mathscr{A}_{g,\bK}$. Let $\mathcal{X}\subseteq \cA_{g,\bK,k}$ be a special subvariety. The \textbf{ordinary naïve integral model} of $\mathcal{X}$ is the ordinary locus of the mod $p$ reduction of the naïve integral model $\mathscr{X}$, denoted by $\mathscr{X}^{\ord}$.

A generically ordinary (irreducible locally closed) subvariety $X\subseteq \mathscr{A}_{g,\bK,\Fpbar}$ is called \textbf{special}, if the Zariski closure of $X$ is an irreducible component of the mod $p$ fiber of the naïve integral model of a special subvareity. 

We also have the notion of \textbf{quasi-weakly special subvariety}; cf. \cite{J23}. Roughly speaking, a generically ordinary (irreducible locally closed) subvariety   $X\subseteq\mathscr{A}_{g,\bK,\Fpbar}$ is called {quasi-weakly special}, if there is a special subvariety $Y\subseteq \mathscr{A}_{g,\bK,{\Fpbar}}^{\ord}$ which splits as an almost product of two special subvarieties $Y_1$ and $Y_2$ with $\dim Y_1>0$, such that $X^{\ord}$ is Zariski dense in the almost product of $Y_1$ with a possibly non-special subvariety $X'\subseteq Y_2$. 
The precise definition is more technical, and won't be used in a serious manner, so we refer the readers to \cite{J23}.   
\subsection{Local coordinates on mixed characteristic Shimura varieties} Fix an ambient integral Siegel modular variety $\mathscr{A}_{g,\bK}$ with hyperspecial level structure. We will work over a finite field $\Fpbar_q$ or its algebraic closure $\Fpbar$. Let $W=W(\Fpbar_q)$ or $W(\Fpbar)$, depending on the context. 
\subsubsection{Canonical coordinates}\label{sub:ST} Let $x$ be an ordinary $\mathbb{F}$-point of $\mathscr{A}_{g,\bK,\mathbb{F}}$. We briefly review the theory of canonical coordinates; cf. \cite{K79}. Let $\mathscr{G}_0=A_x[p^{\infty}]$ and let $\mathscr{G}^{\et}_0$, $\mathscr{G}^{\loc}_0$ be the étale and local part of $\mathscr{G}_0$. Let  $M_0=\bD(\mathscr{G}_0)(W)$. Then  $M_0$ canonically admits a $\bZ_p$-sublattice $$\Gamma_0=\bigoplus_{i=0,1}\{v\in M_0\,|\,(F-p^i)v=0\} \subseteq M_0$$
such that $\Gamma_0\otimes W = M_0$. We have an canonical isomorphism $\Gamma_0\simeq X^*(\mathscr{G}^{\loc}_0)\oplus T_p(\mathscr{G}^{\et}_0)^{\vee}$, where the symbols $X^*$ and $T_p$ stand for the character lattice and the $p$-adic Tate module. One can define a canonical cocharacter $\mu: \bG_m\rightarrow \GL(\Gamma_0)$ inducing the splitting of $\Gamma_0$ and is of weight 1 on $X^*(\mathscr{G}^{\loc}_0)$ and weight 0 on $T_p(\mathscr{G}^{\et}_0)^{\vee}$. By a slight abuse of notation, we also write $\mu:\bG_m\rightarrow \GL(M_0)$ for the obvious scalar extension of $\mu$ to $W$. Via Grothendieck--Messing theory, $\mu$ is the Hodge cocharacter of the canonical lift of $\mathscr{G}_0$, and will be called  the \textbf{canonical Hodge cocharacter}.  

Let $U_{\GL(\Gamma_0),\mu^{-1}}$ be the opposite unipotent of $\mu$ in $\GL(\Gamma_0)$. The theory of canonical coordinates asserts that the formal deformation space $\Def(\mathscr{G}_0/W)$ canonically admits the structure of a formal torus:
$$ \Def(\mathscr{G}_0/W)\simeq \mathrm{Lie} U_{\GL(\Gamma_0),\mu^{-1}}\otimes_{\mathbb{Z}_p}\mathbb{G}_{m,W}^\wedge, $$
where $\mathbb{G}_{m,W}^\wedge$ is the formal torus representing the constant multiplicative $p$-divisible group $\mu_{p^\infty,W}$. We also have a canonical identification $\mathrm{Lie} U_{\GL(\Gamma_0),\mu^{-1}}=X_*(\mathscr{G}^{\loc}_0)\otimes_{\mathbb{Z}_p}  T_p(\mathscr{G}^{\et}_0)^{\vee}$. 
\begin{defn}
  A formal subscheme $\mathscr{T}\subseteq \Def(\mathscr{G}_0/W)$ is called \textbf{linear} (resp. \textbf{quasi-linear}) in the canonical coordinates, if it is a formal subtorus (resp. if $\mathscr{T}_{\overline{W}}\subseteq \Def(\mathscr{G}_0/W)_{\overline{W}}$ is a finite union of translates of formal subtori by torsion points).  
\end{defn}
For example, the formal deformation space $\mathscr{A}_{g,\bK,W}^{/x}\subseteq \Def(\mathscr{G}_0/W)$ is linear in the canonical coordinates: it can be identified with $\Lie U_{\bK_p,\mu^{-1}}\otimes_{\mathbb{Z}_p}\mathbb{G}_{m,W}^\wedge$, where $U_{\bK_p,\mu^{-1}}\subseteq U_{\GL(\Gamma_0),\mu^{-1}}$ is the opoosite unipotent of $\mu$ in the level group $\bK_p$. In general, formal completions of special varieties in $\mathscr{A}_{g,\bK,W}$ at an ordinary $\Fpbar_q$-point are quasi-linear:  
\begin{theorem}[{\cite[Theorem 3.7]{N96}}]\label{thm:noottheorem} Let $\mathscr{X}\subseteq \mathscr{A}_{g,\bK,W}$ be the naive integral model of a special subvariety. Then $\mathscr{X}_{{W}}^{/x}\subseteq \Def(\mathscr{G}_0/W)$ is quasi-linear. 
\end{theorem}

\subsubsection{Lie theoretic coordinates}\label{subsub:lethcoor} Let $x$ be an $\mathbb{F}$-point of $\mathscr{A}_{g,\bK,\mathbb{F}}$. When $x$ is not ordinary, the theory of canonical coordinates is no longer valid. However, we have another type of coorindates, which is less canonical but serves our purpose well in most of situations. We briefly review the constructions in \cite[\S 1.5]{KM09}. First, we think of $x$ as defined over a sufficiently large finite field $\Fpbar_q$ and let $W=W(\Fpbar_q)$. Let $\mathscr{G}_0=A_x[p^{\infty}]$ and let $M_0=\bD(\mathscr{G}_0)(W)$. Let $\mu$ be a cocharacter of $\bG_m\rightarrow \GL(M_0)$ whose mod $p$ reduction gives rise to the Hodge filtration of $M_{0,\Fpbar_q}=\bD(\mathscr{G}_0)(\Fpbar_q)$. By Grothendieck--Messing theory, any such $\mu$ is the Hodge cocharacter of a lift of $\mathscr{G}_0$ to $W$. When $x$ is not ordinary, there is generally no canonical choice of $\mu$. 

Let $U_{\GL(M_0),\mu^{-1}}$ be the opposite unipotent of $\GL(M_0)$ defined by $\mu$. Then $\Def(\mathscr{G}_0/W)$ can be identified with the completion of $U_{\GL(M_0),\mu^{-1}}$ in its (mod $p$) identity section. This gives a system of coordinates on $\Def(\mathscr{G}_0/W)$. We will call them the \textbf{Lie theoretic coordinates}. When $x$ is ordinary, one can explicitly write down a coordinate change from the Lie theoretic coordinates to the canonical coordinates using the Artin--Hasse exponential; cf. \cite[Theorem 1.4.2]{DI81}.

% Choosing coordinates, we have \begin{equation}\label{eq:Liecorabscase}R\simeq W[[t_1,...,t_n]], \,\,n=g^2,\end{equation}with formal group law respecting the group structure on $U^{\mathrm{op}}$. 

\begin{defn}
  A formal subscheme $\mathscr{T}\subseteq \Def(\mathscr{G}_0/W)$ is called \textbf{linear} in the Lie theoretic coordinates, if it is the formal completion of a $W$-subgroup scheme of $U_{\GL(M_0),\mu^{-1}}$ at the identity section. 
\end{defn}

Let $G\subseteq \GL(M_0)$ be a connected reductive subgroup which is the pointwise stabilizer of a family of $\varphi$-invariant tensors $\{s_\alpha\}\subseteq M_0^{\otimes}$, such that the filtration on $M_{0,\Fpbar_q}$ is $G\otimes_W \Fpbar_q$-split; cf, \cite[\S 1.5.4]{KM09}. There exists a cocharacter $\mu:\bG_m\rightarrow G$ whose mod $p$ reduction is the cocharacter inducing the filtration on 
$M_{0,\Fpbar_q}$. Let $U_{G,\mu^{-1}}\subseteq G$ be the opposite unipotent defined by $\mu$, and let $R_G$ be the complete local ring at the (mod $p$) identity section of $U_{G,\mu^{-1}}$. Then $\Spf R_G$ is a linear formal subscheme of $\Def(\mathscr{G}_0/W)$ in the Lie theoretic coordinates. We will leave the readers to check \cite[\S 1.5.4]{KM09} on the construction of a crystal $M_{R_G}:=M_0\otimes_W R_G$ over $R_G$.

 We will also need the following 
\begin{defn}
    [Group scheme cut out by rational tensors] Let $R$ be a $p$-adic DVR and let $N$ be an $R$-lattice.  
   Let $\{s_\alpha\}\subseteq (N[\frac{1}{p}])^{\otimes}$ be a finite family of \textit{rational} tensors that cut out a reductive subgroup $G_0\subseteq\GL(N[\frac{1}{p}])$. Let $G$ be the Zariski closure of 
   $G_0$ in $\GL(N)$, which is a flat group scheme over $R$.  The group scheme $G$ is called the group scheme cut out by (or defined by) $\{s_\alpha\}$. 
\end{defn}
Note that when $\{s_\alpha\}\subseteq N^{\otimes}$ is a family of tensors whose pointwise stabilizer is a connected reductive subgroup scheme $G\subseteq \GL(N)$, then $G$ is also the group scheme cut out by $\{s_\alpha\}$.

%\footnote{In \cite[\S 1.5.4]{KM09}, Kisin requires that the tensors $\{s_\alpha\}$ lie in $M_0^{\otimes}$. The construction remains unchanged even if we start by $\{s_\alpha\}\subseteq (M_0[\frac{1}{p}])^{\otimes}$, since we can scale each $s_\alpha$ so that it lies in $M_0^{\otimes}$, and scaling does not affect the reductive group $G$ cut out by the tensors.};

Let $K'$ be a finite extension of $K=W[\frac{1}{p}]$ and let $\varpi:R\rightarrow \cO_{K'}$ be a continuous map of $W$-algebras and denote by $\mathscr{G}_{\varpi}$ the induced $p$-divisible group over $K'$. Using Breuil--Kisin modules, Kisin is able to characterize when the map $\varpi$ factors through $R_G$; cf. \cite[Proposition 1.5.8]{KM09}. For our purpose, we will need the following 
\begin{theorem}[Kisin]\label{thm:BKcriterion} Suppose that $p>2$. Let $\varpi:R\rightarrow \cO_{K'}$ and $\mathscr{G}_{\varpi}$ be as above. Let $L=T_p\mathscr{G}^*_{\varpi}$ and suppose that there exists a family of $\Gal(K')$-invariant rational tensors $\{s_{\alpha,\et}\}\subseteq (L[\frac{1}{p}])^\otimes$ cutting out a connected reductive subgroup of $\GL(L)$, such that under the $p$-adic comparison isomorphism $$L\otimes_{\bZ_p} B_{\cris}\simeq M_0\otimes_{W}B_{\cris}$$$\{s_{\alpha,\et}\}$ maps to $\{s_\alpha\}\subseteq (M_0[\frac{1}{p}])^{\otimes}$.\footnote{Let $k'$ be the residue field of $K'$, which is a finite extension of $\Fpbar_q$. Strictly speaking, $M_0$ here is $\bD(\mathscr{G}_{0,k'})(W(k'))$. But let's follow Kisin's work and abuse the notation.} Then $\{s_\alpha\}$ cuts out a connected reductive subgroup $G\subseteq \GL(M_0)$, and $\varpi$ factors through $R_G$. 
\end{theorem} 
\begin{proof}
Let $G'\subseteq \GL(L)$ be the reductive subgroup cut out by $\{s_{\alpha,\et}\}$. Then by \cite[Proposition 1.3.2]{KM09} there exists tensors $\{s'_{\alpha,\et}\}\subseteq L^\otimes$ whose pointwise stablizer is $G'$. These tensors are again $\Gal(K')$-invariant: it suffices to show that they are $\Gal(K')$-invariant as rational tensors, which follows since the $\Gal(K')$-invariance of $\{s_{\alpha,\et}\}$ implies that the map $\Gal(K')\rightarrow \GL(L)(\bQ_p)$ factors through $G'(\bQ_p)$. Let $\{s_\alpha'\}\subseteq (M_0[\frac{1}{p}])^\otimes$ be the $\varphi$-invariant tensors arising from $\{s_{\alpha,\et}'\}$ under the $p$-adic comparison isomorphism. Then $\{s_\alpha'\}$ furthermore lies in $M_0^\otimes$ by \cite[Corollary 1.3.6(1), Corollary 1.4.3(1)]{KM09}. The fact that $\{s'_\alpha\}$ cut out a connected reductive subgroup $G\subseteq \GL(M_0)$ follows from \cite[Corollary 1.4.3 (3)]{KM09}, and so $\{s_\alpha\}$ cut out the same reductive subgroup. The assertion that $\varpi$ factors through $R_G$ is a consequence of \cite[Corollary 1.5.11]{KM09}.
\end{proof}
%\subsubsection{Group theory nuance} 

\subsection{Isocrystals and their monodromy}
We will use the following convention: $X_0$ is a geometric connected smooth variety over a finite field $\mathbb{F}_q$, $X=(X_{0})_{\Fpbar}$, and $x$ is an $\Fpbar$-point of $X_0$. Let $W_0=W(\Fpbar_q)$, $W=W(\Fpbar)$ and $K=W[\frac{1}{p}]$. 

\subsubsection{Convergent isocrystals} Convergent isocrystals on $X_0$ can be defined as sheaves on the convergent site by Ogus \cite{OgusII}, or as modules with convergent connection on rigid tubular neighborhoods by Berthelot \cite{Berthlot86}. The two definitions are from the two perspectives of formal and rigid analytic geometry, respectively, but are equivalent. We will very briefly review the constructions. 

First we review the site theoretic definition. Mimicking the construction of the crystalline site $(X_0/W_0)_{\mathrm{cris}}$, Ogus constructed the so called convergent site $(X_0/W_0)_{\mathrm{conv}}$. An object of $(X_0/W_0)_{\mathrm{conv}}$ is called an \textit{enlargement}, which is a certain thickening datum, analogous to the notion of PD thickening in $(X_0/W_0)_{\mathrm{cris}}$. The key difference is that one does not require the existence of a PD structure in the definition of an enlargement. One then defines a \textbf{convergent crystal} as a certain pullback-invariant locally free sheave on $(X_0/W_0)_{\mathrm{conv}}$. One can also make sense of a \textbf{convergent $F$-crystal} as a convergent crystal with a Frobenius structure compatible with the absolute Frobenius on $X_0$. A convergent crystal gives rise to a crystal in $(X_0/W_0)_{\mathrm{cris}}$ by restricting to the subcategory of \textit{$p$-adic enlargements}. Conversely, an $F$-crystal in $(X_0/W_0)_{\mathrm{cris}}$ upgrades to a convergent $F$-crystal, essentially by Dwork's trick. The isogeny category of convergent crystals on $X_0$ will be denoted by $\mathbf{Isoc}(X_0)$. An element in $\mathbf{Isoc}(X_0)$ is called an \textbf{isocrystal}. Similarly, the isogeny category of convergent $F$-crystals on $X_0$ will be denoted by $\FIsoc(X_0)$. An element in $\FIsoc(X_0)$ is called an \textbf{$F$-isocrystal}.

Now we review the rigid analytic definition. Suppose that $X_0$ admits an embedding $X_0$ into a smooth $p$-adic formal scheme $\mathscr{P}_0$ over $W_0$. The rigid tubular neighborhood of $X_0$ in $\mathscr{P}_0$ is denoted by $]X_0[_{\mathscr{P}_0}$ (when the context is clear, we often omit $\mathscr{P}_0$ from the subscript). Berthelot defines the category $\mathbf{Isoc}(X_0)$ of convergent isocrystals as the category of bundles with convergent connections over $]X_0[$. The definition does not depend on the choice of $\mathscr{P}_0$. In general, even if $X_0$ does not admit an embedding $X_0\hookrightarrow\mathscr{P}_0$, one can still cover it by Zariski open where such embedding exists, and glue the category $\mathbf{Isoc}(X_0)$ from local. It can be shown that $\mathbf{Isoc}(X_0)$ is independent of the choices made (i.e., the embedding and the covering), and is equivalent to the site theoretical definition by Ogus. One can also define the category $\FIsoc(X_0)$ as isocrystals with a certain Frobenius structure. 

\subsubsection{Overconvergent isocrystals}
There is another notion, called the \textbf{overconvergent isocrystals}, which can only be described using 
analytic geometry. Suppose that $X_0$ admits a compactification $Y_0$, and $Y_0$ admits a embedding into a $p$-adic formal scheme $\mathscr{P}_0$ over $W_0$, which is smooth along $X_0$. Then an isocrystal  $M$ over $]X_0[$ is said to be \textbf{overconvergent} if, roughly speaking, it extends to a strict neighborhood of $]X_0[$ in $]Y_0[$. In general, the category of overconvergent isocrystals is independent of $Y_0$, $\mathscr{P}_0$, and can be glued from local. The category of overconvergent isocrystals over $X_0$ is denoted by 
$\mathbf{Isoc}^{\dagger}(X_0)$. Finally, the category of \textbf{overconvergent $F$-isocrystals} consists of overconvergent isocrystals with a certain Frobenius structure, and is denoted by $\FIsoc^{\dagger}(X_0)$. There are obvious forgetful maps $\mathrm{Fgt}:\mathbf{Isoc}^{\dagger}(X_0)\rightarrow \mathbf{Isoc}(X_0)$ and $\mathrm{Fgt}:\FIsoc^{\dagger}(X_0)\rightarrow \FIsoc(X_0)$, the second one being fully faithful; cf. \cite{Ked01}. In practice, relative crystalline cohomology of a smooth proper scheme $\pi:\mathcal{X}\rightarrow X_0$ gives rise to an object in $\FIsoc^{\dagger}(X_0)$.

%\subsubsection{Compatible systems}
%\subsubsection{Lefschetz theorems} We will need the following result to reduce questions on monodromy of $l$-adic lisse sheaves over a variety to a curve:\begin{lemma}[{\cite[Lemma 4.4.5]{MD20}}]Let $X$ be a smooth connected variety over a perfect field of characteristic $p$, and $\mathscr{V}$ is a lisse $\mathbb{Q}_l$-sheafover $X$ ($l$ can be equal to $p$). There exists a connected finite Galois étale cover $\widetilde{X} \rightarrow X$, depending on $\mathscr{V}$, with the property that for every smooth connected curve $C \subseteq X$ such that $\widetilde{X}\times_XC$ is connected, the restriction functor $\la\mathscr{V}\ra\rightarrow  \la \mathscr{V}|_{C}\ra$ is an equivalence of categories. \end{lemma}  \begin{proof}  The original statement in \cite[Lemma 4.4.5]{MD20} is restricted to case where $l=p$. However, the proof only involves étale fundamental groups and generalizes to the case where $l\neq p$.\end{proof}
\subsubsection{Monodromy of convergent isocrystals}\label{subsub:mciiso}As indicated in \cite{Crew1992}, the category $\mathbf{Isoc}(X_0)$ and $\FIsoc(X_0)$ are Tannkian, and we can attach monodromy group to ($F$-)isocrystals. Let $\mathbf{Isoc}(X_0)$ (resp. $\FIsoc(X_0)$) be the $\bQ_q$-linear (resp. $\bQ_p$-linear) tensor abelian category of convergent isocrystals (resp. $F$-isocrystals) over $X_0$. Consider an object $\mathcal{M}$ in $\mathbf{Isoc}(X_0)$. We denote by $\la\mathcal{M}\ra^{\otimes}\subseteq \mathbf{Isoc}(X_0)$ the tensor abelian subcategory generated by $\mathcal{M}$, and let $\la\mathcal{M}\ra^{\otimes}_K$ be its scalar extension to $K$. Following \cite{Crew1992}, we can define the following fiber functor 
\begin{align*}
 \omega_x^{\iso}: \langle \mathcal{M}\rangle^\otimes_K &\rightarrow \Vect_{K}\\
\mathcal{N} &\rightarrow \mathcal{N}_x.
\end{align*} 
This realizes $\langle \mathcal{M}\rangle^\otimes_K$ as a $K$-linear neutral Tannakian category. The monodromy group of $\mathcal{M}$ is defined to be the Tannakian fundamental group $\text{Aut}^\otimes(\omega^{\mathrm{iso}}_x)\subseteq \GL(\omega^{\mathrm{iso}}_x(\mathcal{M}))$. 

Suppose that $\mathcal{M}$ is further equipped with a Frobenius structure, i.e., $(\mathcal{M},F)$ is an object of $\FIsoc(X_0)$. One can makes sense of the monodromy group of $\mathcal{M}$ via the following two slightly different methods. \begin{enumerate}
    \item The first method is to take the fiber functor over $\la(\mathcal{M},F)\ra^{\otimes}$ directly as $\omega^{\mathrm{iso}}_x$ and define the monodromy group as $\text{Aut}^\otimes(\omega^{\mathrm{iso}}_x)$. 
    \item The second method is more refined: Let $e$ be the smallest positive integer such that the slopes of $\mathcal{M}_x$ multiplied by $e$ lie in $\mathbb{Z}$, and let $\la(\mathcal{M},F)\ra^{\otimes}_{\mathbb{Q}_{p^e}}$ be the scalar extension of $\la(\mathcal{M},F)\ra^{\otimes}$ to $\mathbb{Q}_{p^e}$. We define the Dieudonné--Manin fiber functor as \begin{align*}
 \omega_x^{\mathrm{DM}}: \langle (\mathcal{M},F)\rangle^\otimes_{\mathbb{Q}_{p^e}} &\rightarrow \Vect_{\mathbb{Q}_{p^e}}\\
(\mathcal{N},F)&\rightarrow \bigoplus_{i\in \bZ}\{v\in \mathcal{N}_x|({F^e_{x}}-p^i)v=0\}.
\end{align*}   
This realizes $\langle (\mathcal{M},F)\rangle^\otimes_{\mathbb{Q}_{p^e}}$ as a $\bQ_{p^e}$-linear neutral Tannakian category. The monodromy group of $\mathcal{M}$ is defined to be the Tannakian fundamental group $\text{Aut}^\otimes(\omega^{\mathrm{DM}}_x)\subseteq \GL(\omega^{\mathrm{DM}}_x(\mathcal{M}))$. 
\end{enumerate}
For an $\mathcal{M}\in \FIsoc(X_0)$, we can make a canonical identification $\GL(\omega_x^{\mathrm{DM}}(\mathcal{M}))\times_{\Spec \bQ_{p^e}} \Spec K=\GL(\omega_x^{\mathrm{iso}}(\mathcal{M}))$. Then $\text{Aut}^\otimes(\omega^{\mathrm{DM}}_x)\times_{\Spec \bQ_{p^e}} \Spec K$ identifies with $\text{Aut}^\otimes(\omega^{\mathrm{iso}}_x)$. If not otherwise specified, the monodromy group of $\mathcal{M}$ obtained via the fiber functor $\omega^{\mathrm{DM}}_x$ will be denoted by $G(\mathcal{M},x)$. In some situations, we will also use $G(\mathcal{M},x)$ to denote the monodromy group obtained from $\omega^{\mathrm{iso}}_x$.
\subsubsection{Monodromy of overconvergent isocrystals}\label{subsub:moiso} Let $\mathbf{Isoc}^\dagger(X_0)$ (resp. $\FIsoc^\dagger(X_0)$) be the $\bQ_q$-linear (resp. $\bQ_p$-linear) tensor abelian category of overconvergent isocrystals (resp. $F$-isocrystals) over $X_0$. We will write $\mathrm{Fgt}$ for the forgetful map  $\mathbf{Isoc}^\dagger(X_0)\rightarrow \mathbf{Isoc}(X_0)$ or $\FIsoc^\dagger(X_0)\rightarrow \FIsoc(X_0)$.

Let $\mathcal{M}^\dagger$ be an object in $\mathbf{Isoc}^\dagger(X_0)$. The monodromy of $\mathcal{M}^\dagger$ is defined as the Tannakian fundamental group of $\la\mathcal{M}^\dagger\ra^{\otimes}_K$ with fiber functor $\omega_x^{\iso}\comp \mathrm{Fgt}$. Similarly, if $\mathcal{M}^\dagger$ is an object of $\FIsoc(X_0)$, one can makes sense of the monodromy group of $\mathcal{M}^\dagger$ via the following two slightly different methods. The first is to take the monodromy of the underlying overconvergent isocrystal. The second is to take the Tannakian fundamental group of $\la\mathcal{M}^\dagger\ra^{\otimes}_{\bQ_{p^e}}$ with fiber functor $\omega^{\mathrm{DM}}_x\circ \mathrm{Fgt}$. 

For an $\mathcal{M}^\dagger\in \FIsoc^\dagger(X_0)$, the two notions of monodromy groups coincide after base change to $K$. If not otherwise specified, the monodromy group of $\mathcal{M}^\dagger_x$ obtained via the fiber functor $\omega^{\mathrm{DM}}_x\circ \mathrm{Fgt}$ will be denoted by $G(\mathcal{M}^\dagger,x)$. In some situations, we will also use $G(\mathcal{M}^\dagger,x)$ to denote the monodromy group obtained from $\omega^{\mathrm{iso}}_x\circ \mathrm{Fgt}$.

\subsection{Betti, étale, and crystalline monodromy for mod $p$ abelian schemes}\label{sub:becm} In the following, we will always assume that the hyperspecial level structure of
$\mathcal{A}_{g,\bK}$ is sufficiently small, and we will drop it from the notation. Let $\mathscr{A}_{g}$ be the canonical integral model. Let $X$ be a smooth connected variety over $\Fpbar$ and let \begin{equation}\label{eq:setupmaps}f:X\rightarrow {\mathscr{A}}_{g,\Fpbar}
\end{equation} be a morphism whose image lies generically in the ordinary locus. We will fix an ordinary base point $x\in X(\Fpbar)$. Let's denote by $A_X$ the abelian scheme over $X$ pulled back via the map $f$. We say that \begin{equation}\label{eq:setupmapsfinitemodel}
    f_0:X_0\rightarrow \mathscr{A}_{g,\mathbb{F}_q}
\end{equation}
is a \textbf{finite field model} of $f$, if $X_0/\bF_q$ is a geometrically connected smooth variety with a map $f_0$ whose base change to $\Fpbar$ is (\ref{eq:setupmaps}). Let's denote by $A_{X_0}$ the abelian scheme over $X_0$ pulled back via the map $f_0$.
\begin{defn}[Betti monodromy; cf.{\cite[\S 4]{J23}}]\label{def:MTgroups}Following \cite{J23}, we define $\mathrm{MS}_f(\mathcal{A}_g)$ as the set of minimal irreducible special subvarieties of $\mathcal{A}_{g,\overline{\bQ}}$ through whose naïve integral model the map $f$ factors. It is known from \cite[Lemma 4.1]{J23} that any two element in $\mathrm{MS}_f(\mathcal{A}_g)$ are $p$-power Hecke translates of each other. The \textbf{Mumford--Tate group} (\textbf{Betti monodromy group}) of $A_X$ is defined to be the generic Mumford--Tate group of any element in $\mathrm{MS}_f(\mathcal{A}_g)$, and is denote by $G_\mathrm{B}(f)$. \footnote{In \cite{J23}, it is denoted by $\MTT(f)$.} The group $G_\mathrm{B}(f)$ is connected and reductive, and is equipped with a tautological representation $\rho_\mathrm{B}$ on the $\bQ$-space $\bH_{\mathrm{B},\Tilde{x}_{\bC}}$, where $\Tilde{x}_{\bC}$ is the canonical lift of $x$, base changed to $\bC$ along the fixed map $W\hookrightarrow \bC$.   
\end{defn}

\begin{defn}[Étale and overconvergent crystalline monodromy] Fix a finite field model (\ref{eq:setupmapsfinitemodel}). Let $l\neq p$, we write $\bH_{l,f_0}$ or $\bH_{l,X_0}$ for the pullback étale sheaf on $X_0$ via the map $f_0$. We also write $\bH_{\cris,f_0}$ or $\bH_{\cris,X_0}$ for the pullback overconvergent $F$-isocrystal on $X_0$ via the map $f_0$. To ease notation, we will often write them as $\bH_{p,f_0}$ or $\bH_{p,X_0}$. Let $u\in \mathrm{fpl}(\bQ)$, an element $\bH_{u,f_0}$ is called a \textbf{$u$-adic coefficient object}. 

The collection $\{\bH_{l,f_0}\}_{l\neq p}$ is a \textbf{$\bQ$-compatible system of coefficient objects} in the sense of \cite{DA20}. The collection $\{\bH_{u,f_0}\}_{u\in \mathrm{fpl}(\bQ)}$ is a \textbf{weakly $\bQ$-compatible system of coefficient objects} in the sense of \cite[\S 3.4]{J23}. It can be upgraded to a $\bQ$-compatible system if we replace the overconvergent $F$-isocrystal $\bH_{u,f_0}$ by an overconvergent $F^a$-isocystal for a suitable $a\in \bN$. 

Let $(G(\bH_{u,f_0},x),\rho_u)$ be the monodromy group of $\bH_{u,f_0}$, together with its tataulogical representation. When $u\neq p$, this is just the étale monodromy group. When $u=p$, this is the monodromy group of the overconvergent $F$-isocrystal $\bH_{p,f_0}$ with respect to the fiber functor $\omega_x^{\mathrm{DM}}\circ \mathrm{Fgt}$; cf. \S\ref{subsub:moiso}. It is known that $G(\bH_{u,f_0},x)$ is a reductive $\bQ_u$-group. When $u=l\neq p$ (resp. $u=p$), we will call it the \textbf{$l$-adic étale monodromy} (resp. \textbf{overconvergent crystalline monodromy}) of $A_{X_0}$. For convenience, we will usually just call it the \textbf{$u$-adic monodromy}. To ease notation, we will often denote $G(\bH_{u,f_0},x)$ simply by $G_u(f_0)$. It is known that the neutral component $G_u(f_0)^\circ$ together with $\rho_u$ is 
``independent of $u$'' in the sense that it admits a rational model over a sufficiently large number field in the sense of \cite[Theorem 1.2.4]{DA20}.  
\end{defn}
\begin{defn}[Convergent crystalline monodromy]Let $\bH_{p,f_0}^{-}$ be the image of $\bH_{p,f_0}$ under the forgetful map
$\mathrm{Fgt}:\FIsoc^{\dagger}(X_0)\rightarrow \FIsoc(X_0)$. Let $(G(\bH_{p,f_0}^{-},x),\rho_p)$ be the monodromy of $\bH_{p,f_0}^{-}$ with respect to the fiber functor $\omega_x^{\mathrm{DM}}$; cf. \S\ref{subsub:mciiso}. This will be called the \textbf{convergent crystalline monodromy} of $A_{X_0}$.  

Suppose that the image of $f$ lies in the ordinary strata (so that $\bH_{p,f_0}^{-}$ 
 admits a slope filtration), then $G(\bH_{p,f_0}^{-},x)$ is the parabolic of $G(\bH_{p,f_0},x)$ fixing the Newton cocharacter $\nu_x$. This is a consequence of Crew's parabolicity conjecture, which is now a theorem of D'Addezio; cf. \cite{MD20}. 
\end{defn}

\begin{notation}[Independence on the  models]\label{not:hidingfinitefield} In practice, we often work with coefficient objects and their monodromy without referring to a specific finite field model. It is convenient to set up a model-free system of notation.  Let $f$ be as in (\ref{eq:setupmaps}), and let $f_0$ be a finite field model. For each $u$, the neutral component $G(\bH_{u,f_0},x)^\circ$ is independent of the finite field model chosen. When we vary the finite field models, there exist finite field models defined over certain sufficiently large $\Fpbar_q$ such that $G(\bH_{u,f_0},x)$ is the smallest among all possible $G(\bH_{u,f_0},x)$ over all 
possible finite field models. We will write $\bH_{u,X}$ or $\bH_{u,f}$ for the coefficient objects over those finite field models, and write $G(\bH_{u,f},x)$ for the monodromy group. To ease notation, we will often denote $G(\bH_{u,f},x)$ simply by $G_u(f)$. Similarly, we will write $\bH_{p,X}^-$ for the underlying convergent $F$-isocrystal of $\bH_{p,X}$. We will denote $G(\bH_{p,f}^-,x)$ by $G_p^-(f)$. 
\end{notation}
\subsubsection{The mod $p$ Mumford--Tate conjecture}\label{subsub:mdppmtintro} Let $f$ be as in (\ref{eq:setupmaps}), and let $f_0$ be a finite field model. For $l\neq p$, there is a canonical identification $\iota_l:\bH_{\mathrm{B},\Tilde{x}_{\bC}}\otimes \bQ_l\simeq \bH_{l,x}$. In the crystalline setting, we also have a  canonical identification $\iota_p:\bH_{\mathrm{B},\Tilde{x}_{\bC}}\otimes \bQ_p\simeq \omega_x^{\mathrm{DM}}(\bH_{p})$; cf. \cite[\S 4.2]{J23}

\begin{theorem}[{\cite[]{J23}}]\label{thm:GincludeinMT}
Let $u\in \mathrm{fpl}(\bQ)$. After making identification $\iota_u$ as above, we have a canonical embedding $G_u(f)^\circ\subseteq G_{\mathrm{B}}(f)_{\mathbb{Q}_u}$ intertwining the tautological representations.  
\end{theorem}
The {Mod $p$ Mumford--Tate conjecture}~\ref{conj:MTforAg} ($\mathrm{MT}_p$) claims that the above embedding is actually an equality. In what follows we summarize some results from \cite{J23}.  First, there are structural results parallel to the classical characteristic 0 setting: 
\begin{proposition}[{\cite[Proposition 4.8]{J23}}]\label{thm: Endequal}
Let $\eta$ be the generic point of $X$.   Then for any $u\in \mathrm{fpl}(\bQ)$,  $$\End_{G_u(f)^\circ}(\rho_u)=\End_{G_\mathrm{B}(f)}(\rho_\mathrm{B})_{\bQ_u}=\End({A}_{X}\times{\overline{\eta}})_{\bQ_u}.$$
\end{proposition}

\begin{proposition}[{\cite[Proposition 4.9]{J23}}]\label{thm:Z0equal}
Let $u\in \mathrm{fpl}(\bQ)$, then $Z^0G_u(f)= Z^0G_{\mathrm{B}}(f)_{\bQ_l}$, where $Z^0G$ stands for the neutral component of the center of $G$. 
\end{proposition}

There is another fundamental result relating the structure of $G_p(f)$ with the geometry of $f$ in terms of the canonical coordinates on $\mathscr{A}_{g,\Fpbar}$, which does not admit a characteristic 0 analogue. Write $\mu_x$ for the canonical Hodge cocharacter of $x$. Let $\mathscr{G}_0$ be the $p$-divisible group $A_x[p^{\infty}]$. Let the notation be as in \S\ref{sub:ST}. By the theory of canonical coordinates,  the formal deformation space $\Def(\mathscr{G}_0/W)$ is canonically a formal torus with cocharacter lattice $\mathrm{Lie} U_{\GL(\Gamma_0),\mu^{-1}}$. On the other hand, we can canonically identify $\omega_x^{\mathrm{DM}}(\bH_{p})$ with $\Gamma_0\otimes{\bQ_p}$. Therefore there is a canonical identification  
\begin{equation}\label{eq:canide}
\kappa_{x}^{\mathrm{DM}}:\mathrm{Lie} U_{\GL(\Gamma_0),\mu^{-1}}\otimes\bQ_p\simeq    \mathrm{Lie}  U_{\GL(\omega_x^{\mathrm{DM}}(\bH_{p})),\mu^{-1}}.
\end{equation} Let $\mathscr{T}_{f,x}$ be the smallest formal subtorus of $\Def(\mathscr{G}_0/W)\otimes \Fpbar$ containing the image of $f^{/x}: X^{/x}\rightarrow \mathscr{A}_{g,\Fpbar}^{/x}\subseteq \Def(\mathscr{G}_0/W)\otimes \Fpbar$. Let ${T}_{f,x}\subseteq \mathrm{Lie} U_{\GL(\Gamma_0),\mu^{-1}}\otimes\bQ_p$ be the rational cocharacter lattice of  $\mathscr{T}_{f,x}$.   

\begin{theorem}[{\cite[Theorem 4.2]{J23}}]\label{Thm:Tatelocal} Notation as above. Suppose that the image of $f$ lies in the ordinary locus , then  \begin{enumerate} 
\item\label{itcor1} 
$G_p^{-}(f)$ is the opposite parabolic of $G_p(f)$ with respect to $\mu_x$.  
\item\label{itcor2} Let $\mathrm{gr}\bH_{p,f}$ be the graded object of $\bH_{p,f}^{-}$ with respect to the slope filtration, then $G(\gr\bH_{p,f},x)$ is the canonical Levi of $G_p^{-}(f)$ fixing $\mu_x$.  
\item\label{itcor3}
The Lie algebra of the opposite unipotent of $G_p^{-}(f)$ with respect to $\mu_x$ is canonically identified with $T_{f,x}$ under (\ref{eq:canide}). 
\end{enumerate}
\end{theorem}
Finally, we record the following convenient technical lemma which enables us to freely shrink $X$, or replace $f$ by a dominant cover: 
\begin{lemma}
Consider a dominant morphism  $g:X'\rightarrow X$ containing $x$ in the image, and let $x'\in X'(\Fpbar)$ be a point with image $x$. Let $f'=f\circ g$. 
Consider the monodromy groups $G_\mathrm{B}(f')$ and $G_u(f')$ with respect to $x'$. Then $G_\mathrm{B}(f)=G_\mathrm{B}(f')$ and $G_u(f)^{\circ}=G_u(f')^\circ$. Furthermore, there exists an étale cover $g$ such that $G_u(f')=G_u(f)^{\circ}$ for all $u$.  
\end{lemma}
\begin{proof}
    The claim for $G_\mathrm{B}(f)=G_\mathrm{B}(f')$ is clear from the definition. The claim for $G_u(f)^{\circ}=G_u(f')^\circ$ is \cite[Lemma 4.7]{J23}. The claim for $G_u(f')=G_u(f)^{\circ}$ is \cite[Lemma 3.4]{J23}.
\end{proof}

\subsection{Cycle conjectures}
Let $A/k$ be an abelian variety over a field finitely generated over a prime field. Let $l\neq \mathrm{char} \,k$ be a prime.  Following \cite{Milne}, we call a $\gamma\in H^{2i}_l(A)(i):=H^{2i}(A_{\overline{k}},\bQ_l(i))$ a \textbf{Tate class}, if it is fixed by an open subgroup of $\Gal(k^s/k)$. The set of Tate classes are denoted by $\mathcal{T}_l^{i}(A)$. It is clear that base changing $A$ to a finite extension does not change $\mathcal{T}_l^i(A)$. It is also known that there exists a sufficiently large finite extension $k'/k$ such that $\Gal(k^s/k')$ fixes $\mathcal{T}_l^i(A)$ (In fact, one can take $k'$ to be the extension such that the $l$-adic monodromy group $G_l(A_{k'})$ is connected). 

In what follows, we will denote by $\mathrm{Alg}^i_l(A)$ the image of the cycle class map $$c:\mathrm{CH}^i(A_{\overline{k}})_{\bQ}\rightarrow H^{2i}_l(A)(i).$$
An element of $\mathrm{Alg}^i_l(A)$  will be called an \textbf{algebraic class}. It is known that $\mathrm{Alg}^i_l(A)$  is a $\bQ$-vector space that lies in $\mathcal{T}_l^i(A)$. The following is a version of the Tate conjecture:
\begin{conj}[Tate conjecture] \label{conj:Tate}$\mathcal{T}_l^i(A)$ is the $\bQ_l$-span of $\mathrm{Alg}^i_l(A)$. 
\end{conj}

\begin{lemma}\label{lm:liftTatecycles}
Let $k$ be a number field, and let $A/k$ be a CM abelian variety with ordinary reduction at a prime lying above $p$. Denote by $A_0$ the special fiber. Let $l\neq p$ be a prime. Assume Conjecture~\ref{conj:Tate} for $A$.\footnote{The Tate conjecture for CM abelian varieties over finite generated fields is claimed by OpenAI in October 2026; cf. \cite{OAI2026}} Then the natural specialization map  $\mathrm{Alg}^i_l(A) \rightarrow \mathrm{Alg}^i_l(A_0)$ is an isomorphism.  
\end{lemma}
\begin{proof} 
We briefly explain the  map $\mathrm{Alg}^i_l(A) \rightarrow \mathrm{Alg}^i_l(A_0)$. We have the following commutative diagram
\begin{equation}\label{eq:cycleclassspecialization}
\begin{tikzcd}
\mathrm{CH}^i(A_{\overline{k}})_{\bQ} \arrow[r] \arrow[d,"c"]           & \mathrm{CH}^i(A_{0,\Fpbar})_{\bQ} \arrow[d,"c_0"]             \\
H^{2i}_l(A)(i) \arrow[r] & H^{2i}_l(A_{0})(i)
\end{tikzcd} 
\end{equation}
It is then clear that we have the desired map between algebraic classes. 

The specialization map induces an isomorphism 
$H^{1}_l(A)\xrightarrow{\sim} H^{1}_l(A_{0})$ under which the neutral component of the $l$-adic monodromy group $G_l(A)^\circ$ identifies with $G_l(A_0)^{\circ}$ (here we have used the fact that $A$ is CM with ordinary reduction). It follows that the second row of (\ref{eq:cycleclassspecialization}) induces an isomorphism $$T_l^i(A)\xrightarrow{\sim} T_l^i(A_0).$$
In particular, we can identify $\mathrm{Alg}^i_l(A)$ as a $\bQ$-subspace of $ \mathrm{Alg}^i_l(A_0)$.

Assume Conjecture~\ref{conj:Tate} for $A$. Then $\mathrm{Alg}^i_l(A_0)$ lies in the $\bQ_l$-span of $\mathrm{Alg}^i_l(A)$.
On the other hand, specializing the algebraic cycles, we see that Conjecture~\ref{conj:Tate} holds for $A_0$ as well. Applying \cite[Theorem 1.7]{Milne} to the algebraic classes, we see that the natural map $$\mathrm{Alg}^i_l(A_0)\otimes\bQ_l \rightarrow T_l^i(A_0)$$
is an isomorphism. It follows that (1) $\dim_{\bQ}\mathrm{Alg}^i_l(A_0)<\infty$, and (2) $\mathrm{Alg}^i_l(A_0)\otimes 
\bQ_l=\mathrm{Alg}^i_l(A)\otimes\bQ_l$. Counting dimension, we deduce that $\mathrm{Alg}^i_l(A)=\mathrm{Alg}^i_l(A_0)$. 
\end{proof}
\begin{notation}\label{not:tatesetup}
In the following, let's assume that $\mathrm{char}\, k= p$. Up to a finite extension, we will assume that $A$ extends to an abelian scheme over a generically ordinary $\Fpbar_q$-variety $X_0$, which corresponds to a map $f_0: X_0\rightarrow \mathscr{A}_{g,\Fpbar_q}$. By abuse of notation, this abelian scheme will still be denoted by $A$. We denote by $G_{\mathrm{B}}(A)$ and $G_{u}(A)$ its monodromy groups. Let $x_0$ be an ordinary $\Fpbar_q$-point of $X_0$ and let $x$ be an $\Fpbar$-point that factors through $x_0$. Write $A_{x_0}$ for the fiber of $A$ over $x_0$, and let $A_{\tilde{x}_0}$ be the canonical lift of $A_{x_0}$, considered as a CM abelian variety over a number field.  
\end{notation}
The following is a characteristic $p$ analogue of the classical relation between the Hodge, Tate and the Mumford--Tate conjectures. 
\begin{proposition}\label{prop:relationbetweencycleconjectures} Notation as above. The following are true: 
\begin{enumerate}
    \item 
Assume that the Tate conjecture holds for powers of $A$ and powers of  $A_{\Tilde{x}_0}$, then $\mathrm{MT}_p$ holds for $A$. 
\item Assume that $\mathrm{MT}_p$ holds for $A$ and the Hodge conjecture holds for abelian varieties (over $\bC$) with  dimension $\dim A$, then the Tate conjecture holds for $A$.  
\end{enumerate}
\end{proposition}
\begin{proof}\begin{enumerate}
    \item 
Let $l\neq p$ be a prime. Since the Tate conjecture holds for powers of $A$, there are finitely many algebraic classes $\{\gamma_{\alpha}\}$ (of powers of $A_{\overline{k}}$, which extends to $X_0$ if we choose $X_0$ to be étale locally small) that cut out $G_l(A)^\circ$ from $\GL(\bH_{l,x})$. Let $\{\gamma_{\alpha,x_0}\}$ be the specialization of $\{\gamma_{\alpha}\}$ to powers of $A_{x_0}$, then lift $\{\gamma_{\alpha,x_0}\}$ to algebraic classes $\{{{\gamma}}_{\alpha,\tilde{x}_0}\}$ on powers of $A_{\tilde{x}_0}$ by  Lemma~\ref{lm:liftTatecycles}. These algebraic classes then define Hodge cycles $\{{{\gamma}}_{\alpha,\mathrm{B},\tilde{x}_0}\}$ on $A_{\tilde{x}_0}$. Let $\mathcal{X}$ be a maximal special subvariety such that $\{{{\gamma}}_{\alpha,\mathrm{B},\tilde{x}_0}\}$ extends horizontally. Let $\mathscr{X}$ be the naïve integral model of $\mathcal{X}$.

By a theorem of Ogus \cite[VI. \S4.1]{DP82}, $\{{{\gamma}}_{\alpha,\mathrm{B},\tilde{x}_0}\}$ gives rise to absolute Tate cycles $\{{{\gamma}}_{\alpha,\mathrm{dR},\tilde{x}_0}\}$ of $A_{\tilde{x}_0}$. By \cite[Theorem 2.8]{N96}, there exists a biggest formal subscheme $\mathcal{N} \subseteq \mathscr{A}_{g,\overline{W}}^{/x}$ containing $\tilde{x}_{0,\overline{W}}$, such that some $p$-power multiples of the Tate cycles $\{{{\gamma}}_{\alpha,\mathrm{dR},\tilde{x}_0}\}$ extend horizontally to $\mathcal{N}$. Then the map $f_{0,\Fpbar}$ factors through $\mathcal{N}$. On the other hand, let $\mathcal{C}$ be any irreducible component of
$\mathscr{X}_{\overline{W}}^{/x}$ that passes through $\tilde{x}_{0,\overline{W}}$. Then $\mathcal{C}\subseteq \mathcal{N}$ as well. By the dimension formula  \cite[Theorem 2.8 B2]{N96}, we find that $\dim \mathcal{C}=\dim \mathcal{N}$. So $\mathcal{C}=\mathcal{N}$. Therefore $f_{0,\Fpbar}$ factors through $\mathcal{C}$. This shows that $\mathscr{X}_{\Fpbar}$ contains the image of $f_{0,\Fpbar}$. 

As a result, $G_{\mathrm{B}}(A)\subseteq G_{\mathrm{B}}(\mathcal{X})$, where $G_{\mathrm{B}}(\mathcal{X})$ is the generic Mumford--Tate group of $\mathcal{X}$. On the other hand, since $G_{\mathrm{B}}(\mathcal{X})$ fixes all $\{{{\gamma}}_{\alpha,\mathrm{B},\widetilde{x}_0}\}$, we have $G_{\mathrm{B}}(\mathcal{X})_{\bQ_l}\subseteq G_l(A)^\circ$. These inequalities, together with Theorem~\ref{thm:GincludeinMT}, implies that $G_{\mathrm{B}}(A)_{\bQ_l}=G_l(A)^\circ$. 
\item Now assume that $\mathrm{MT}_p$ holds for $A$, and the Hodge conjecture holds for $A_{\eta,\bC}:=A_\eta\times_{\eta} \Spec \bC$, where $\eta$ is the generic point of a minimal special subvariety $\mathcal{X}$ whose reduction contains the image of $f_{0,\Fpbar}$ (and $\Spec \bC\rightarrow \eta$ is an arbitrary non-trivial map). By definition, $\mathcal{T}_l^i(A)$ is the subspace of $H^{2i}_l(A)(i)$ fixed by $G_l(A)^\circ$. Let $\mathcal{H}^i(A)\subseteq H^{2i}_\mathrm{B}(A_{\eta,\bC})(i)$ be the subspace fixed by $G_{\mathrm{B}}(A)$. Then we can identify $\mathcal{T}_l^i(A)$ with $\mathcal{H}^i(A)\otimes\bQ_l$. The Hodge conjecture implies that $\mathcal{H}^i(A)$ is spanned by algebraic cycles over the algebraic closure $\overline{\eta}$, and so we can spread them out to algebraic cycles over some finite cover of $\mathcal{X}$ and then take reduction modulo $p$ to get the desired algebraic classes on $A_{\overline{k}}$ that span $\mathcal{T}_l^i(A)$. 
\end{enumerate}
\end{proof}

\section{Proof of MT$_p$ $\Leftrightarrow$ $\mathrm{logAL}_p$}\label{sec:ProofMT} 
Let $X$ be a smooth connected variety over $\Fpbar$ and let $f:X\rightarrow {\mathscr{A}}_{g,\Fpbar}$
be a morphism whose image lies in the ordinary locus. Let $x\in X(\Fpbar)$ and let $\mathscr{T}_{f,x}$ be the smallest formal subtorus of $\Def(\mathscr{G}_0/\Fpbar)$ containing the image of $f^{/x}: X^{/x}\rightarrow \mathscr{A}_{g,\Fpbar}^{/x}$. We say that ``$\mathrm{logAL}_p$ holds for $f$ at $x$'' if $\mathscr{T}_{f,x}$ is an irreducible component of the formal germ at $x$ of a special subvariety of $\mathscr{A}_{g,\Fpbar}$ (in other words, Conjecture~\ref{conj:AxSchanuel} holds for $f$). The validity of this conjecture is independent of the base point $x$ chosen; cf, Lemma~\ref{lm:MTGsameunip}. So it also makes sense to say that ``$\mathrm{logAL}_p$ holds for $f$'' without referring to the base point.  In this chapter, we establish the following
\begin{theorem}\label{thm:logALimpliesMT}
$\mathrm{logAL}_p$ holds for $f$ if and only if $\MTT_p$ holds for $f$.
\end{theorem}
In \cite{J23}, we have shown that if $\MTT_p$ holds for $f$, then so is $\mathrm{logAL}_p$. So it suffices to prove the opposite direction. We achieve this by combining Theorem~\ref{Thm:Tatelocal} with the theory of Mumford--Tate pairs.

\subsection{Frobenius tori}\label{sub:Frobtorus} Recall that $x\in X(\Fpbar)$ is an ordinary point. The Mumford--Tate group $G_{\mathrm{B}}(x)\subseteq \GL(\bH_{\mathrm{B},\tilde{x}_{\bC}})$ of this single point is a torus, and can be canonically identified with the Mumford--Tate group $G_{\mathrm{B}}(\Tilde{x}_{\bC})$ of the canonical lifting $\Tilde{x}_{\bC}$. It is classically known that $\MTT_p$ holds for a single point, i.e., $G_{\mathrm{B}}(x)_{\bQ_u}=G_u(x)^\circ$ for all $u\in \mathrm{fpl}(\bQ)$. 

Let $x_0$ be a finite field model of $x$. 
The geometric Frobenius of the abelian variety $A_{x_0}$ lifts to $\Tilde{x}_{\bC}$, and gives rise to an element $t_{x}\in \End(\bH_{\mathrm{B},\Tilde{x}_{\bC}})$ that commutes with $G_{\mathrm{B}}(x)$. It turns out that $G_{\mathrm{B}}(x)$ is the neutral component of the Zariski closure of $\{t_x^n\}_{n\in \bZ}$. Therefore, the torus $G_{\mathrm{B}}(x)$ is an analogue of the Frobenius torus in Serre's theory (cf. \cite{Serre1,Chi}).

Let $\mu_{x}:\bG_m\rightarrow \GL(\omega_x^{\mathrm{DM}}(\bH_{p,x}))$ be the canonical Hodge cocharacter of $x$, which can be identified with the Hodge cocharacter  $\mu_{\tilde{x}_\bC}:\bG_m\rightarrow \GL(\bH_{\mathrm{B},\tilde{x}_\bC}\otimes E)$ via $\iota_p$ as per \S\ref{subsub:mdppmtintro}, where $E$ is the reflex field of $\tilde{x}_\bC$, whose completion at the place corresponding to $E\subseteq \bC_p$ is $\bQ_p$. In particular, we can regard $\mu_x$ as a cocharacter mapping into $\GL(\bH_{\mathrm{B},\tilde{x}_\bC}\otimes E)$. The following is a characteristic $p$ analogue of \cite[p.10]{Serre1} (or \cite[Theorem 3.4]{Chi}, or \cite[Proposition 3.5]{P98}): 
\begin{proposition}\label{prop:genbyHodge}
Notation as above. Then $G_{\mathrm{B}}(x)$ is generated
by the $\Gal(\bQ)$-orbit of $\mu_{x}$.
\end{proposition}
\begin{proof}
This follows from the fact that $G_\mathrm{B}(\tilde{x}_\bC)$ is the smallest $\bQ$-group containing  $\im \mu_{\tilde{x}_\bC}$. 
\end{proof}

%It is classically known that for an abelian variety over number field, the  (\cite{}). This theorem has an analogue over characteristic $p$ function fields. Before we state it, we need to first make sense of Dirichlet density in characteristic $p$. 
\subsubsection{Density of $V_{\mathrm{max}}$}Let $X_0$ be a finite field model of $X$. Let $u\in \mathrm{fpl}(\bQ)$. We define 
\begin{equation}\label{eq:V_max}
    V_{\max}(X_0)=\{x_0\in |X_0|: \rk G_\mathrm{B}(x)= \rk G_u(f)\}.
\end{equation}
The definition does not depend on $u$. The condition ``$\rk G_\mathrm{B}(x)= \rk G_u(f)$'' is equivalent to ``$G_u(x)^\circ\subseteq G_u(f)$ is a maximal torus''. This explains the notation ``$V_{\max}$''. 

Our goal is to show that $V_{\max}(X_0)$ is a subset of $|X_0|$ of ``large density''. Let $\mathcal{X}\rightarrow\Spec \bZ$ be a finite type integral scheme of dimension $d$. Denote by $|\mathcal{X}|$ the set of closed points of $\mathcal{X}$. There is a notion of \textbf{(Dirichlet) density} for a subset of $|\mathcal{X}|$, generalizing the classical density defined for number fields. We will refer the readers to \cite{S56} or \cite[Appendix B]{PinkMTDrinfeld} for the precise definition. 

In the following, let
$\pi:\widetilde{\mathcal{X}}\rightarrow\mathcal{X}$ be an irreducible finite étale cover with Galois group $G$. Then every point $x_0\in |\mathcal{X}|$ determines a conjugacy class $\la\Frob_{\pi^{-1}(x_0)/x_0}\ra\subseteq G$. The following version of Chebotarev's density theorem is \cite[Theorem B.9]{PinkMTDrinfeld} (there is another version by Serre under slight different setup, see \cite[Theorem 7]{S56}). 
\begin{theorem}[Chebotarev's density theorem]\label{thm:chebotarev}
    Notation as above. For every conjugacy class $\mathcal{C}\subseteq G$, the set $\{x_0\in |\mathcal{X}|: \la\Frob_{\pi^{-1}(x_0)/x_0}\ra=\mathcal{C}\}$ has density ${|\mathcal{C}|}/{|G|}$.
\end{theorem}

\begin{proposition}[Chin]\label{vmax_positivedensity} $V_{\max}(X_0)$ contains a subset of positive density. 
\end{proposition}
\begin{proof}
Take $u$ in (\ref{eq:V_max}) be a prime  $l$ not equal to $p$. Possibly replacing $X_0$ by a Galois cover, we can assume that $G_l(f_0)$ is connected. Let $U\subseteq G_l(f_0)$ be the conjugation-stable Zariski dense open subset in \cite[Theorem 5.7]{Chin} (note that the conditions of this theorem are met since we are working with the first $l$-adic étale cohomology of an abelian scheme). It has the property that $$\rho_l^{\et}\la\Frob_{x_0}\ra\in U \Rightarrow x_0\in V_{\max}(X_0),$$
where $\rho_l^{\et}:\pi_1^{\et}(X_0)\rightarrow G_l(f_0)$ and $\la\Frob_{x_0}\ra$ is the conjugacy class of the geometric Frobenius of $x_0$. The map $\rho_l^{\et}$ is continuous with respect to the pro-finite topology on $\pi_1^{\et}(X_0)$ and $l$-adic topology on $G_l(f_0)$. On the other hand, the $l$-adic topology on $G_l(f_0)$ is finer than the Zariski topology. It follows that $(\rho_l^{\et})^{-1}(U)\subseteq\pi_1^{\et}(X_0)$  contains a conjugate-invariant non-empty clopen subset $\cO$. As a result, there exists a Galois cover $\pi:\widetilde{X}_0\rightarrow X_0$ such that $\cO$ is a finite union of cosets of $\pi_1^{\et}(\widetilde{X}_0)$. Apply Theorem~\ref{thm:chebotarev} to the cover $\pi$ to conclude. 
\end{proof}
\subsection{Mumford--Tate pairs} We review the theory of Mumford--Tate pairs, a crucial Lie theoretic tool for attacking the Mumford--Tate conjecture. We show that the theory is still valid for ordinary abelian varieties in characteristic $p$. 
\begin{defn}[{\cite[\S4]{P98}}]
    Let $(G, \rho)$ be a faithful finite dimensional representation of a reductive group over
a characteristic $0$ field $F$.  The pair $(G,\rho)$ is called a \textbf{weak Mumford–Tate pair of weights $\{0,1\}$} if there exist finitely many cocharacters $\mu_1$, ..., $\mu_d$ of $G_{\overline{F}}$ such that $\rho\comp \mu_i$ are all of weights 0 and 1 and that $G_{\overline{F}}$ is generated by $G({\overline{F}})$-conjugates of $\mu_1$, ..., $\mu_d$. The pair $(G,\rho)$ is called a \textbf{strong Mumford–Tate pair of weights $\{0,1\}$} if it is a {weak Mumford–Tate pair of weights $\{0,1\}$} and that the $\mu_i$'s are conjugate under $\Gal(F)$. 
       \end{defn}
\begin{example}\label{example:classical}
It is classically known (cf. \cite[\S 5]{P98}) that the Mumford--Tate group of an abelian variety $A/\overline{\bQ}$, together with its representation, form a strong Mumford--Tate pair over $\bQ$. Moreover, the Mumford--Tate group is generated by the $\Gal(\bQ)$-conjugates of the Hodge cocharacter of $A$.
\end{example}\begin{remark}
It is known that simple factors of a weak Mumford--Tate pair are again weak Mumford--Tate pairs (\cite[\S4]{P98}). Simple weak Mumford--Tate pairs are classified thoroughly, see \cite[Table 4.2, Corollary 5.11]{P98}. On the other hand, absolutely irreducible strong Mumford--Tate pairs have several good structural properties, see \cite[Proposition 4.4, 4.5]{P98}.   
\end{remark} 
The theory of Mumford--Tate pairs works equally well for ordinary abelian schemes in char $p$: 
\begin{proposition}\label{prop:MTpairs} Recall that $f:X\rightarrow {\mathscr{A}}_{g,\Fpbar}$
is a morphism whose image lies in the ordinary locus. The following are true: \begin{enumerate}
    \item $(G_{\mathrm{B}}(f),\rho_\mathrm{B})$ is a strong Mumford--Tate pair over $\bQ$.
    \item For each $u\in \mathrm{fpl}(\bQ)$, $(G_u(f)^\circ,\rho_u)$ is a weak Mumford--Tate pair over $\bQ_u$.
\end{enumerate}
\end{proposition}
\begin{proof} By definition, $G_{\mathrm{B}}(f)$ is the  generic Mumford--Tate group of a special subvariety $\mathcal{X}\subseteq \mathcal{A}_g$, which is defined over some number field $E$. By \cite[Fact 3.3.1.1]{Cado}, $\mathcal{X}$ contains an $l$-Galois-generic point $z$ over a finite extension of $E$. By \cite[Theorem A]{Cado}, $z$ is Galois generic, hence Hodge generic by \cite[Proposition 6.2.1]{Cado}. This means that the Mumford--Tate group of $z$ equals $G_{\mathrm{B}}(f)$. By Example~\ref{example:classical}, we see that the Mumford--Tate group of $z$ and its representation is a strong Mumford--Tate pair. This concludes (1).   

By Proposition~\ref{vmax_positivedensity}, there is an ordinary point $y\in X(\Fpbar)$ such that $G_u(y)^\circ$ is a maximal torus of $G_u(f)^\circ$. Since  $G_u(y)^\circ=G_{\mathrm{B}}(y)_{\bQ_u}$, we conclude by Proposition~\ref{prop:genbyHodge} that $G_u(y)^\circ$ is generated by the $\Gal(\bQ)$-orbit of $\mu_y$. Since $\mu_y$ is defined over a number field,  the $\Gal(\bQ)$-orbit of $\mu_y$ contains only finitely many cocharacters $\mu_1,...,\mu_d$, and they generate a maximal torus of $G_u(f)^\circ$. Since $G_u(f)^\circ(\overline{\bQ}_u)$-conjugates of a maximal torus generate $G_u(f)^\circ_{\overline{\bQ}_u}$, this proves (2). 
\end{proof}
Proposition~\ref{prop:MTpairs} enables us to adapt some classical results to the characteristic $p$ setting.  
\begin{theorem}[Analogue of {\cite[Theorem 5.3]{P98}}]\label{thm:Pinkcase2}
Let $\eta$ be the generic point of $X$. Assume that $\End({A}_{\overline{\eta}}) = \mathbb{Z}$ and that the root system of each simple
factor of $G_\mathrm{B}(f)$ has type $A_{2s-1}$ with $s\geq 1$ or $B_r$ with $r\geq 1$. Then $\MTT_p$ holds for $f$. 
\end{theorem}
\begin{proof}
Let $u$ be a finite place of $\bQ$,  Proposition~\ref{thm:GincludeinMT}, Proposition~\ref{thm: Endequal} and Proposition~\ref{prop:MTpairs} show that $G_u(f)^\circ_{\overline{\bQ}_u}$ is a subgroup $G_{\mathrm{B}}(f)_{\overline{\bQ}_u}$ whose tautological representation is irreducible and is a weak Mumford-Tate pair of weights $\{0, 1\}$. It follows from \cite[Proposition 4.3]{P98} that $G_{\mathrm{B}}(f)_{\bQ_u}$ does not possess any proper subgroup whose tautological representation is irreducible and is a weak Mumford-Tate pair of weights $\{0, 1\}$.
\end{proof}

\subsubsection{Tensor decomposition}\label{subsub:STfactor}
Let $F$ be a characteristic 0 algebraically closed field, and let $(G,\rho)$ be a weak Mumford--Tate pair of weights $\{0,1\}$ over $F$. Fix a cocharacter $\mu:\bG_{m,F}\rightarrow G$ of weights $\{0,1\}$. There is an almost direct product $$G=Z\cdot G_1\cdot G_2\cdot...\cdot G_n,$$ such that $Z$ is the center, and each $G_i$ is almost simple.  We say that $G_i$ is a \textbf{Serre--Tate factor with respect to $\mu$}, if $\mu$ factors non-trivially through it. The collection of Serre--Tate factors depends only on the conjugate class of $\mu$.  %Corresponding to this is a splitting $$\Lie G=\Lie Z\oplus \bigoplus_{i=1}^n\Lie G_i.$$ 

Now suppose that $\rho$ is irreducible. There are irreducible representations $\rho_{0},\rho_{1},...,\rho_{n}$ of $Z,G_{1},...,G_{n}$, respectively, such that $\rho\simeq \rho_{0}\boxtimes\rho_{1}\boxtimes...\boxtimes\rho_{n}$. Arguing as \cite[\S4, Tensor decomposition]{P98}, we have \begin{enumerate}
    \item  $Z\simeq \bG_{m,F}$, where $\rho_0$ is the standard representation.
    \item\label{it:simplMTtriple}  Let $G_i$ be a Serre--Tate factor with respect to $\mu$. Then $\mu$ factors through $Z\cdot G_{i}$, and the component of $\mu$ in $Z$ is non-trivial. The pair $(Z\cdot G_{i},\rho_{0}\boxtimes\rho_{i})$ is a Mumford--Tate pair of weights $\{0,1\}$, generated by the conjugates of $\mu$.  
\end{enumerate}
We say that $(G,\rho, \mu)$ (over $F$) is a \textbf{simple Mumford--Tate triple of weights $\{0,1\}$}, if $(G,\rho)$ is an irreducible weak Mumford--Tate pair of weights $\{0,1\}$, $G=\bG_m\cdot G^{\der}$, $G^{\der}$ is almost simple, and $\mu$ is a cocharacter of weights $\{0,1\}$ whose $G(\overline{F})$-conjugates generate $G$. When $\mu$ is clear, $(G,\rho)$ is called a \textbf{simple Mumford--Tate pair of weights $\{0,1\}$}. We will often omit ``of weights $\{0,1\}$'' from the name. The triple $(Z\cdot G_{i},\rho_{0}\boxtimes\rho_{i},\mu)$ in (\ref{it:simplMTtriple}) is an example of simple Mumford--Tate triple of weights $\{0,1\}$. Such triples are classified in \cite[Table 4.2]{P98}.

%When $\rho$ is not necessarily irreducible. We can decompose it into irreducible summands, and The representation $\rho$ splits into (geometrically) irreducible ones, i.e., $\rho=\phi_1\oplus \phi_2\oplus...\oplus\phi_n$.  The representation $(\phi_i(G),\phi_i)$ can further be decomposed into a tensor product. More precisely, there is an almost direct product decomposition $\phi_i(G)=Z_i\cdot G_{i,1}\cdot ... \cdot G_{i,n_i}$, such that $Z_i$ is the center, and $G_{i,j}$'s are almost simple.Given a cocharacter $\mu:\bG_{m,F}\rightarrow G$ of weights $\{0,1\}$, we call a $Z_i\cdot G_{i,j}$,  through which $\phi_i\comp \mu$ factors with both weights 0 and 1, a \textbf{Serre--Tate factor with respect to $\mu$}. Since $G_{i,j}$ itself is an almost simple factor of $G$, sometimes we will ignore $Z_i$, and directly call $G_{i,j}$ a Serre--Tate factor. 

     \iffalse
\begin{proposition} Let $l$ be a finite place of $\mathbb{Q}$, then $(G^\circ_l(f),\rho_l)\subseteq (\MTT(f)_{\mathbb{Q}_l},\rho_l)$ is an inclusion of weak Mumford--Tate pairs.
\end{proposition}
\begin{proof}The assertion that
$(\MTT(f)_{\mathbb{Q}_l},\rho_l)$ is a weak Mumford--Tate pair follows from the fact that $\MTT(f)$ is a generic Mumford--Tate group. The assertion that $(G^\circ_l(f),\rho_l)$ is a weak Mumford--Tate pair follows from the existence of an ordinary point $y\in X(\Fpbar)$ such that $G^\circ_l(y)$ is a maximal torus of $G^\circ_l(f)$ (\cite[]{DA20}), the fact that $G^\circ_l(y)=\MTT(y)_{\bQ_l}$, and Proposition~\ref{prop:genbyHodge}. 
\end{proof}
\fi

\subsection{Big unipotents} Under the assumption that logAL$_p$ holds for $f$, we show that a the $u$-adic monodromy have big unipotents. In addition to Mumford--Tate pairs, this produces another piece of crucial information, as illustrated by the following pure Lie theoretic result: 
\begin{lemma}\label{lm:smaeunipimplyequal}
    Let $(G',\rho)\subseteq (G,\rho)$ be an inclusion of weak Mumford--Tate pairs of weights $\{0,1\}$ over $F=\overline{F}$. Let $\mu: \mathbb{G}_m\rightarrow G'$ be a cocharacter of weights $\{0,1\}$. Suppose that 
    \begin{enumerate}
        \item $(G,\rho,\mu)$ is a simple Mumford--Tate triple of weights $\{0,1\}$. %(so it corresponds to one of the six types in \cite[Table 4.2]{P98}).
        \item $G$ and $G'$ have the same unipotent and opposite unipotent with respect to $\mu$. 
    \end{enumerate}
    Then $G=G'$. 
\end{lemma}
  \begin{proof}
 Fix a maximal torus of $G$ containing the image of $\mu$ and consider the root system $\Phi_G$. It suffices to show that the positive and negative roots with respect to $\mu$ generate $\Phi_G$. In fact, by condition (2), this immediately implies that $G'\supseteq G^{\der}$. It then follows that $G=G'$. 

 To show that roots generate $\Phi_G$, we can check \cite[Table 4.2]{P98} case by case. For simplicity we only check the first column of \cite[Table 4.2]{P98} ($A_r$ with standard representation), and the rest is similar. 
The root system of $A_r$ is $\{e_i- e_j\}$. From Pink's chart, we see that for any $1\leq s\leq r$, the set of positive roots with respect to $\omega_s^{\vee}$ is $\{e_{i}-e_{j}|1\leq i\leq s, s+1\leq j\leq r+1\}$. Clearly, any $e_i-e_j$ is either a  positive root, or a  negative root, or the sum of a positive root and a negative root.
  \end{proof}
Now we deduce some geometric consequences from the group theoretic result. Let $y$ be any ordinary point of $X(\Fpbar)$. We will identify $G(\bH_{p,f},y)$ with $G_p(f)$. The following lemma says that the condition ``$\mathrm{logAL}_p$ holds for $f$'' can be translated to a condition on the unipotents of $G_p(f)$, and implies that a ``unipotent variant'' of $\MTT_p$ holds for $f$.
\begin{lemma}\label{lm:MTGsameunip}
Let $\mathcal{S}\in \mathrm{MS}_f(\mathcal{A}_g)$. The following are equivalent: \begin{enumerate}
     \item\label{it:logALnotdependonx1} $\mathrm{logAL}_p$ holds for $f$ at $y$.
     \item\label{it:logALnotdependonx1.5}  $\dim \mathcal{S}=\rk \mathscr{T}_{f,y}$.
     \item\label{it:logALnotdependonx2} $\dim\mathcal{S}=\dim U_{G_p(f),\mu_y}$. 
     \item\label{it:logALnotdependonx3} $G_p(f)$ and $G_\mathrm{B}(f)_{\mathbb{Q}_p}$ have the same unipotent and opposite unipotent with respect to $\mu_{y}$.
 \end{enumerate}
  Furthermore, if $\mathrm{logAL}_p$ holds for $f$ at $y$, then it holds for $f$ at any other ordinary point.  
\end{lemma}
\begin{proof}
Let $\mathscr{S}$ be the naïve integral model of an $\mathcal{S}\in \mathrm{MS}_f(\mathcal{A}_g)$. By Theorem~\ref{Thm:Tatelocal}, we have 
$$\dim U_{G_p(f),\mu_y^{-1}}=\rk \mathscr{T}_{f,y}.$$ 
So (\ref{it:logALnotdependonx1.5})$\Leftrightarrow$ (\ref{it:logALnotdependonx2}). Since $\mu_{y}$ identifies with the Hodge cocharacter of its canonical lifting, we have 
$$\dim U_{G_{\mathrm{B}}(f)_{\bQ_p},\mu_y^{-1}}=\dim \mathcal{S}.$$
So (\ref{it:logALnotdependonx1.5})$\Leftrightarrow$ (\ref{it:logALnotdependonx3}). Note that $\mathscr{S}_{\overline{W}}^{/y}$ is quasi-linear by Theorem~\ref{thm:noottheorem}. So (\ref{it:logALnotdependonx1.5})$\Rightarrow$ (\ref{it:logALnotdependonx1}). Conversely, 
there exists a special subvariety $\mathcal{S}'$, such that $\mathscr{T}_{f,y}$ is an irreducible component of the formal germ at $y$ of the mod $p$ reduction of $\mathcal{S}'$. By minimality of $\mathcal{S}$ and the fact that any two elements in $\mathrm{MS}_f(\mathcal{A}_g)$ are $p$-power Hecke translates of each other, we find that $\mathcal{S}=\mathcal{S}'$ up to a $p$-power Hecke translate. So (\ref{it:logALnotdependonx1})$\Rightarrow$(\ref{it:logALnotdependonx1.5}).

Finally, if $\mathrm{logAL}_p$ holds for $f$ at $y$, then (\ref{it:logALnotdependonx2}) holds. Since $\dim U_{G_p(f),\mu_y}$ does not depend on $y$, (\ref{it:logALnotdependonx2}) holds true for other choices of $y$ as well. Therefore $\mathrm{logAL}_p$ holds for $f$ any other ordinary point. 
\end{proof}

  \iffalse
  \begin{lemma}\label{lm:MTGsameunip}
    Suppose that $f$ is AS. Let $y\in X(\Fpbar)$ be an ordinary point. Then $G_p(f)$ and $\MTT(f)_{\mathbb{Q}_p}$ have the same unipotent and opposite unipotent with respect to $\mu_{y}$.
\end{lemma} 
\begin{proof}
Let $\mathscr{S}$ be the naïve integral model of a special subvariety whose reduction admits $\mathscr{T}_{f,y}$ as an irreducible component. Then by Noot's result, $\mathscr{S}$ is indeed the smallest special variety through which $f$ factors. Now by Theorem~\ref{Thm:Tatelocal}, we have 
$$\dim U_{G_p(f),\mu_y^{-1}}=\rk \mathscr{T}_{f,y}.$$ 
Since $\mu_{y}$ identifies with the Hodge cocharacter of any of its quasi-canonical lifting, we have 
$$\dim U_{\MTT(f)_{\bQ_p},\mu_y^{-1}}=\dim \mathscr{S}.$$
Finally, by Theorem~\ref{thm:noottheorem} we have $$\dim \mathscr{S} = \rk \mathscr{T}_{f,y}.$$
It follows that $U_{G_p(f),\mu_y^{-1}}=U_{\MTT(f)_{\bQ_p},\mu_y^{-1}}$. Now since unipotent has the same dimension with the opposite unipotent. We have $U_{G_p(f),\mu_y}=U_{\MTT(f)_{\bQ_p},\mu_y}$ as well.
\end{proof}
\fi
\begin{corollary}\label{cor:onefactor}
Suppose that $\mathrm{logAL}_p$ holds for $f$. Let $y\in X(\Fpbar)$ be an ordinary point. Following \S\ref{subsub:STfactor}, let $\{G_1,...,G_r\}, r\geq 1$ be the set of Serre--Tate factors of the weak Mumford--Tate pair $(\MTT(f)_{\bC_p},\rho_{\mathrm{B}})$ with respect to the canonical Hodge cocharacter $\mu_y$. Let $\mathfrak{g}_i=\Lie G_i$. Then $\mathfrak{g}_1\oplus...\oplus\mathfrak{g}_r$ is a direct summand of $\Lie G_p(f)_{\bC_p}$. 
\end{corollary}
\begin{proof}
Suppose that $\Lie G_{\mathrm{B}}(f)_{\bC_p}$ factors as $\mathfrak{z}\oplus \mathfrak{g}_1\oplus...\oplus\mathfrak{g}_r \oplus \mathfrak{g}_{r+1}\oplus...\oplus \mathfrak{g}_{m}$, where $\mathfrak{z}$ is the center, and $\mathfrak{g}_{r+1},...,\mathfrak{g}_m$ are the Lie algebras of non-Serre--Tate factors. Denote by $p_i$ the projection map $\Lie G_p(f)_{\bC_p}\rightarrow \mathfrak{g}_i$. Let $\mathfrak{h}$ and $\mathfrak{h}'$ be the projection of $\Lie G_p(f)_{\bC_p}$ in $\mathfrak{g}_1\oplus...\oplus\mathfrak{g}_r$ and $ \mathfrak{z} \oplus \mathfrak{g}_{r+1}\oplus...\oplus \mathfrak{g}_{m}$, respectively. We need to show that $\mathfrak{h}=\mathfrak{g}_1\oplus...\oplus\mathfrak{g}_r$, and $\Lie G_p(f)_{\bC_p}=\mathfrak{h}\oplus \mathfrak{h}'$.

Since the weak Mumford--Tate pair $(G_{\mathrm{B}}(f)_{\bC_p},\rho_{\mathrm{B}})$ is not necessarily irreducible, we need to first decompose it into irreducible summands:  $\rho_{\mathrm{B}}=\phi_1\oplus \phi_2\oplus...\oplus\phi_n$. For each $j\in [1,n]$, $(\phi_j(G_{\mathrm{B}}(f)_{\bC_p}),\phi_j)$
is an irreducible weak Mumford--Tate pair, and we can use Pink's observation outlined in \S\ref{subsub:STfactor}.  The set of Serre--Tate factors of $(G_{\mathrm{B}}(f)_{\bC_p},\rho_{\mathrm{B}})$ is the union of the sets of Serre--Tate factors of $(\phi_j(G_{\mathrm{B}}(f)_{\bC_p}),\phi_j)$, for $j$ running through $1,2,...,n$. By Lemma~\ref{lm:MTGsameunip}, $G_p(f)$ and $G_{\mathrm{B}}(f)_{\mathbb{Q}_p}$ have the same unipotent and opposite unipotent with respect to $\mu_y$. This property carries over to the Serre--Tate factors, and Lemma~\ref{lm:smaeunipimplyequal} guarantees that $p_i$ is surjective for $1\leq i\leq r$.
Since each $\mathfrak{g}_i$ is simple, the subalgebra $\mathfrak{h}\subseteq \mathfrak{g}_1\oplus...\oplus\mathfrak{g}_r$ can be easily described via Goursat's lemma. If $\mathfrak{h}$ was strictly smaller than $\mathfrak{g}_1\oplus...\oplus\mathfrak{g}_r$, then it has a simple factor $\mathfrak{g}'$ that projects isomorphically onto two different $\mathfrak{g}_i$'s. Then $G_p(f)$ and $G_{\mathrm{B}}(f)_{\mathbb{Q}_p}$ won't have the same unipotent with respect to $\mu_y$.  Therefore, we must have  $\mathfrak{h}=\mathfrak{g}_1\oplus...\oplus\mathfrak{g}_r$.

Now we show that $\Lie G_p(f)_{\bC_p}=\mathfrak{h}\oplus \mathfrak{h}'$. It suffices to show that    $\Lie G_p(f)_{\bC_p}$ does not admit a simple factor that admits non-trivial projections to both $\mathfrak{h}$ and $\mathfrak{h}'$. Now suppose that $\mathfrak{g}'$ is a simple factor of $\Lie G_p(f)_{\bC_p}$ that admits non-trivial projection to $\mathfrak{h}$. Then $p_i(\mathfrak{g}')=\mathfrak{g}_i$ for some $1\leq i\leq r$. Since $\mu_y$ factors nontrivially through $\mathfrak{g}_i$, it also factors nontrivially through $\mathfrak{g}'$. Since 
$\mathfrak{g}_{r+1},...,\mathfrak{g}_m$ are not Serre--Tate, the projection of $\mathfrak{g}'$ in $\mathfrak{h}'$ can only lie in $\mathfrak{z}$. But this means that the projection is trivial. This finishes the proof.
\end{proof}
\subsection{Proof of Theorem~\ref{thm:logALimpliesMT}}
Corollary~\ref{cor:onefactor} claims that that $G_{\mathrm{B}}(f)$ and $G_p(f)^\circ$ agree over Serre--Tate factors. To establish $\MTT_p$, we also need to prove the equality over other factors. To achieve this, we should not just confine ourselves to $p$-adic monodromy. Instead, we need to look at $u$-adic monodromy for all $u\in \mathrm{fpl}(\bQ)$ all at once. The proof in this section is reminiscent of \cite[\S 5.2.3, \S 5.3]{J23}. Roughly speaking, we first decompose $G_\mathrm{B}(f)^{\ad}$ into simple factors over a number field $F$. Fixing a Galois closure $C\supseteq F$. We can then split a faithful irreducible $\bQ$-representation of $G_\mathrm{B}(f)^{\ad}$ over $C$. By Tannakian formalism, this enables us to construct several families of weakly $C$-compatible system of coefficient objects. We then use Galois twist to travel between them, and use Chebotarev's density theorem to establish the agreement of $G_{\mathrm{B}}(f)$ and $G_p(f)^\circ$ over all factors. 
\begin{ass}
   Replacing $X$ by sufficiently large étale cover, we will assume that all $G_u(f)$ are connected. This is always achievable by \cite[Lemma 3.4]{J23}. We also write $X_0$ for a finite field model of $X$ such that $G_u(f_0)=G_u(f)$. 
\end{ass}
\subsubsection{The simple case}\label{sub:simplecase}
We first suppose that $G_{\mathrm{B}}(f)^{\ad}$ is simple. Then there is an absolutely simple reductive group $\mathcal{G}$ over a number field $F$, such that $G_{\mathrm{B}}(f)^{\ad}=\Res_{F/\bQ}\mathcal{G}$. Let $\phi:\mathcal{G}\hookrightarrow \GL(\mathcal{V})$ be a faithful irreducible representation, where $\mathcal{V}$ is a vector space over $F$. Let $C$ be a Galois closure of $F$, with a fixed embedding into $\bC$. Let $\Sigma$ be the set of embeddings of $F$ into $C$. Then $G_{\mathrm{B}}(f)^{\ad}_C= \prod_{\sigma\in \Sigma} \mathcal{G}_{\sigma,C}$. We also get representations $\phi_{\sigma,C}: \mathcal{G}_{\sigma,C}\acts \mathcal{V}_{\sigma,C}$. The sum $G_{\mathrm{B}}(f)^{\ad}_C\acts \bigoplus_{\sigma\in \Sigma} \mathcal{V}_{\sigma,C}$ descends to a $\bQ$-representation $\varphi:G_{\mathrm{B}}(f)\rightarrow \GL(V)$.  By Tannakian formalism, for each $u\in \text{fpl}(\bQ)$, the 
representation $$G_u(f)\hookrightarrow G_{\mathrm{B}}(f)_{\bQ_u}\xrightarrow{\varphi_u} \GL(V_{\bQ_u})$$
corresponds to a $\bQ_u$-coefficient object $\mathscr{E}_u$ over a finite étale cover $X_0'\rightarrow X_0$. Similarly, for each $u\in \text{fpl}(\bQ)$ and each $v\in\text{fpl}(C)$ dividing $u$, the representation 
$$G_u(f)_{C_v}\hookrightarrow G_{\mathrm{B}}(f)_{C_v}\rightarrow \mathcal{G}_{\sigma,C_v}\xrightarrow{\phi_{\sigma,C_v}}\GL(\mathcal{V}_{\sigma,C_v})$$
gives rise to a $C_v$-coefficient object $\mathscr{E}_{\sigma,v}\subseteq \mathscr{E}_u\otimes C_v$ over $X_0'$. We have $\mathscr{E}_u\otimes C_v=\bigoplus_{\sigma\in \Sigma}\mathscr{E}_{\sigma,v}$. 
We also have the following relations between monodromy groups: $$G(\mathscr{E}_{\sigma,v},x)\subseteq \mathcal{G}_{\sigma,C_v},\;\;G(\mathscr{E}_{u}\otimes C_v,x)\subseteq \bigoplus_{\sigma\in \Sigma}G(\mathscr{E}_{\sigma,v},x).$$
\begin{lemma}\label{lm:compatibleinsinglecase}
The collection $\{\mathscr{E}_u\}_{u\in\mathrm{fpl}(\bQ)}$ is a weakly $\bQ$-compatible system in the sense of \cite[\S 3.4]{J23}. For each $\sigma$, the collection $\{\mathscr{E}_{\sigma,v}\}_{v\in\mathrm{fpl}(C)}$ is a weakly $C$-compatible system. 
\end{lemma}
\begin{proof}
To show that $\{\mathscr{E}_u\}_{u\in\mathrm{fpl}(\bQ)}$ is weakly $\bQ$-compatible, it suffices to show that for every closed point $x_0\in |X_0|$ the characteristic polynomial of $\mathscr{E}_{u,x_0}$ is a $\mathbb{Q}$-polynomial independent of $u$. Since $x_0$ is ordinary, it admits a canonical lift $\Tilde{x}_0$. The geometric Frobenius lifts to an element $t_{x_0}\in G_{\mathrm{B}}(\Tilde{x}_0)\subseteq G_{\mathrm{B}}(f)$. Let $P\in \bQ[t]$
be the characteristic polynomial of $t_{x_0}$ with respect to the representation $\varphi:G_{\mathrm{B}}(f)\rightarrow \GL(V)$. When $u\neq p$, the charateristic polynomial of $\mathscr{E}_{u,x_0}$ coincides with $P$. When $u=p$, one can replace the Frobenius on $\mathscr{E}_{u,x_0}$ by some power so that the charateristic polynomial again coincides with $P$. This shows that $\{\mathscr{E}_u\}_{u\in\mathrm{fpl}(\bQ)}$ is weakly $\bQ$-compatible.  

The proof of the fact that $\{\mathscr{E}_{\sigma,v}\}_{v\in\mathrm{fpl}(C)}$  is weakly $C$-compatible is similar. By definition, the base change $\varphi_C$ splits up into various $\phi_{\sigma,C}$. Accordingly, the polynomial $P$ splits into a product of $P_\sigma\in C[t]$, which are permuted under the $\Gal(C/\bQ)$-action. Now it is easy to see that each element in $\{\mathscr{E}_{\sigma,v}\}_{v\in\mathrm{fpl}(C)}$ has characteristic polynomial $P_\sigma$ over $x_0$ (again, when $v|p$, we need to replace the Frobenius by a power), hence the compatibility. 
\end{proof}
\begin{construction}[Galois twists]
  Note that $\Gal({C}/\mathbb{Q})$ acts on $\Sigma$. Let $g\in \Gal({C}/\mathbb{Q})$, the \textbf{$g$-twist} of a $C_v$-scheme $X$ is the $C_{gv}$-scheme $X\times_g \Spec C_{gv}$, denoted by $gX$. There is also an obvious notion of $g$-twist of morphisms between $C_v$-schemes. There is also a notion of \textbf{$g$-twist of $C_v$-coefficient objects} defined as $g\mathcal{E}:= \mathcal{E}\otimes_{g}C_{gv}$. We have 
\begin{equation}\label{lm:Galtwist}
g\mathcal{G}_{\sigma,{C}_v}=  \mathcal{G}_{g\sigma,{C}_{gv}},\,\,\,g\mathscr{E}_{\sigma,v}=\mathscr{E}_{g\sigma,gv},\,\,\,,gG(\mathscr{E}_{\sigma,v})=G(g\mathscr{E}_{\sigma,v}).
\end{equation}  \end{construction}
%\subsubsection{The conclusion}\label{subsub:theconclusion}
\begin{lemma}\label{lm:monodromyEandMT}
Suppose that $\mathrm{logAL}_p$ holds for $f$, then $G(\mathscr{E}_u,x)=(\Res_{F/\bQ}\mathcal{G})_{\bQ_u}$ for any $u\in \mathrm{fpl}(\bQ)$.
\end{lemma}
\begin{proof}
The proof is similar to \cite[Theorem 5.11]{J23}. First, there exists a place $\tau\in \Sigma$, such that $\mathcal{G}_{\tau,{\bC_p}}$ is a Serre--Tate factor of $(G_{\mathrm{B}}(f)_{\bC_p}, \rho_\mathrm{B})$ with respect to $\mu_x$. For every $u\in \mathrm{fpl}(\bQ)$ and $v\in \mathrm{fpl}(C)$ dividing $u$, we always have 
\begin{equation}\label{eq:uvsub}
    G(\mathscr{E}_u\otimes C_v, x)\subseteq G(\mathscr{E}_{\tau,v},x) \times G(\bigoplus_{\sigma\in \Sigma-\{\tau\}}\mathscr{E}_{\sigma,v},x) \text{ and } G(\mathscr{E}_{\tau,v},x)\subseteq \mathcal{G}_{\tau,C_v}.
\end{equation} 
Let $\mathfrak{p}$ be the place of $C$ induced from $C\hookrightarrow \bC\simeq \bC_p$. Then by Corollary~\ref{cor:onefactor}, the two inclusions in (\ref{eq:uvsub}) are equalities when $(u,v)=(p,\mathfrak{p})$. By Lemma~\ref{lm:compatibleinsinglecase}, \cite[Lemma 3.4]{J23}, and a simple dimension counting argument, we find that the two inclusions in (\ref{eq:uvsub}) are equalities for all $(u,v)$. %(in order to apply MD, we need to first base change to a larger finite extension $C'/C$ and establish the equalities there, and then conclude the equalities over $C$ by dimension counting).there is a finite extension $\mathfrak{C}/C$ with the property that for each $v\in \mathrm{fpl}(C)$ and  $\mathfrak{v}\in \mathrm{fpl}(\mathfrak{C})$ with $\mathfrak{v}|v$, the base changes of (\ref{eq:goodequalityilikeit})$_{(\tau,v)}$ to $\mathfrak{C}_{\mathfrak{v}}$ are equalities. By dimension counting, (\ref{eq:goodequalityilikeit})$_{(\tau,v)}$ are already equalities. 

Let $\tau'\in \Sigma$ be any element. Chebotarev's density theorem guarantees the existence of a place $v_0\in \mathrm{fpl}(C)$ and an element $g\in \Gal(C/\bQ)$ fixing $v_0$ such that $g\tau=\tau'$. Let $v=v_0$ in (\ref{eq:uvsub}) and consider the $g$-twist, we obtain
\begin{equation}\label{eq:uvsub2}
    G(\mathscr{E}_u\otimes C_{v_0}, x)=G(\mathscr{E}_{\tau',v_0},x) \times G(\bigoplus_{\sigma\in \Sigma-\{\tau'\}}\mathscr{E}_{\sigma,v_0},x) \text{ and } G(\mathscr{E}_{\tau',v_0},x)=\mathcal{G}_{\tau',C_{v_0}}.
\end{equation} 
By \cite[Lemma 3.4]{J23} again, we see that (\ref{eq:uvsub2}) holds for all $v\in \mathrm{fpl}(C)$. As a result, for any $v$, we have 
$$  G(\mathscr{E}_u\otimes C_{v}, x)= \prod_{\sigma\in \Sigma} G(\mathscr{E}_{\sigma,v},x)=\prod_{\sigma\in \Sigma} \mathcal{G}_{\sigma,C_{v}}=(\Res_{F/\bQ}\mathcal{G})_{C_v}.$$
This concludes the lemma. 
\end{proof}
It follows from the lemma that $G_{\mathrm{B}}(f)_{\bQ_u}^{\ad}=G_u(f)^{\ad}$. Together with Theorem~\ref{thm:Z0equal}, this proves  Theorem~\ref{thm:logALimpliesMT} when $G_{\mathrm{B}}(f)^{\ad}$ is simple.
 \subsubsection{The general case}\label{subsub:MTgen}
Note that Theorem~\ref{thm:logALimpliesMT} readily follows from the simple case by Theorem~\ref{modpCommelin}. However, we give an alternate proof. The proof is essentially already written down in \cite[\S 5.3.1]{J23}. 

Let $\bbI$ be a finite index set, and let $G_{\mathrm{B}}(f)^{\ad}=\prod_{i\in \bbI}\Res_{F_i/\bQ}\mathcal{G}_i$, where each $\mathcal{G}_i$ is an absolute simple reductive group over a number field $N_i$. Let $\Sigma'_{i}$ be the set of complex embeddings of $N_i$. Fix $C\subseteq\bC$ to be a sufficiently large number field which contains the image of all complex embedding of all $N_i$. For each $i\in \bbI$, let $\phi_i:\mathcal{G}_i\hookrightarrow \GL(\mathcal{V}_i)$ be a faithful irreducible representation over $N_i$. By Tannakian formalism, for each $i\in \bbI$ and $u\in \mathrm{fpl}(\bQ)$, the representation $$G_u(f)\hookrightarrow G_{\mathrm{B}}(f)_{\bQ_u}\xrightarrow{(\Res_{N_i/\bQ}\phi_{i})_{\bQ_u}} \GL((\Res_{N_i/\bQ} \mathcal{V}_i)_{\bQ_u})$$
gives rise to a $\bQ_u$-coefficient object $\mathscr{E}_{i,u}$. For each $i\in \bbI$, $\sigma\in \Sigma'_i$, $u\in \mathrm{fpl}(\bQ)$ and $v\in\text{fpl}(C)$ dividing $u$, the representation
$$G_u(f)_{C_v}\hookrightarrow G_{\mathrm{B}}(f)_{C_v}\rightarrow (\mathcal{G}_i)_{\sigma,C_v}\xrightarrow{(\phi_i)_{\sigma,C_v}}\GL((\mathcal{V}_{i})_{\sigma,C_v})$$
gives rise to a $C_v$-coefficient object $\mathscr{E}_{i,\sigma,v}$. Argue as Lemma~\ref{lm:compatibleinsinglecase}: for each $i\in \bbI$, $\{\mathscr{E}_{i,u}\}_{u\in\mathrm{fpl}(\bQ)}$ is a weakly $\bQ$-compatible system, and for each $\sigma\in \Sigma'_i$ the collection $\{\mathscr{E}_{i,\sigma,v}\}_{v\in\mathrm{fpl}(C)}$ is a weakly $C$-compatible system.    
\begin{lemma}\label{lm:GC_v}
Suppose that $\mathrm{logAL}_p$ holds for $f$, then 
 $G(\mathscr{E}_{i,u}, x)=(\Res_{N_i/\bQ}\mathcal{G}_i)_{\bQ_u}$, and $G(\mathscr{E}_{i,\sigma,v}, x)=(\mathcal{G}_i)_{\sigma,C_v}$. 
 \end{lemma}
\begin{proof}
This follows from Lemma~\ref{lm:monodromyEandMT} by projecting to each factor. 
\end{proof}
As a consequence, the projection $G_u(f)\rightarrow (\Res_{N_i/\bQ}\mathcal{G}_i)_{\bQ_u}$ is surjective. This gives rise to a surjection $G_u(f)^{\ad}\twoheadrightarrow (\Res_{N_i/\bQ}\mathcal{G}_i)_{\bQ_u}$. The induced map $G_u(f)^{\ad}\rightarrow \prod_{i\in\bbI}(\Res_{N_i/\bQ}\mathcal{G}_i)_{\bQ_u}= G_{\mathrm{B}}(f)_{\bQ_u}^{\ad}$ is an injection, and identifies $G_u(f)^{\ad}$ as the monodromy group of $\bigoplus_{i\in \bbI}\mathscr{E}_{i,u}$.

\begin{defn}
    Let $\{G_a\}_{a\in \bbA}$ be a finite collection of (algebraic) groups over a field. Suppose there is a subgroup $G\subseteq \prod_{a\in \bbA} G_a$ such that each projection $G\rightarrow G_a$ is an isomorphism. We will say that $G$ is an \textbf{iso-graph} over the collection $\{G_a\}_{a\in \bbA}$. 
\end{defn}

Let $\bbJ$ be a subset of $\bbI$. To compactify the notation, we write, for $u\in \mathrm{fpl}(Q)$, $\mathscr{E}_{\bbJ,u}:= \bigoplus_{j\in \bbJ}\mathscr{E}_{j,u}$.
%We use $\sigma_i$ to denote an element in $\Sigma'_i$, and use $\sigma_\bbJ=(\sigma_{j})_{j\in \bbJ}$ to denote an element in $\prod_{j\in \bbJ} \Sigma'_{j}$. $$\mathscr{E}_{\bbJ,\sigma_{\bbJ},v}:=\bigoplus_{j\in \bbJ}\mathscr{E}_{j,\sigma_{j},v},\,\,\,\,G_{\bbJ,\sigma_{\bbJ},v}:= G(\mathscr{E}_{\bbJ,\sigma_{\bbJ},v},x).$$We will also write $G_{i,u}$ (\textit{resp}. $G_{i,\sigma_i,v}$) for $G(\mathscr{E}_{i,u},x)$ (\textit{resp}. $G(\mathscr{E}_{i,\sigma_{i},v},x)$). The group $G_{i,\sigma_i,v}$ is absolutely simple by Lemma~\ref{lm:GC_v}. It is also clear that for fixed $\bbJ$ and $\sigma_\bbJ$, $\{\mathscr{E}_{\bbJ,u}\}_{u\in \mathrm{fpl}(\bQ)}$ is a weakly $\bQ$-compatible system and  $\{\mathscr{E}_{\bbJ,\sigma_{\bbJ},v}\}_{v\in \mathrm{fpl}(C)}$ is a  weakly $C$-compatible system. 

\begin{lemma}\label{lm:goodlemma}
Suppose that $\mathrm{logAL}_p$ holds for $f$. There is a partition $\bbI=\bigsqcup_{h\in \mathcal{G}}\bbI_h$, such that for any $u\in \mathrm{fpl}(\bQ)$, we have $G_u(f)^{\ad}=\prod_{h\in \mathcal{G}} G(\mathscr{E}_{\bbI_h,u},x)$, where each $G(\mathscr{E}_{\bbI_h,u},x)$ is an iso-graph over $\{G(\mathscr{E}_{i,u},x)\}_{i\in \mathbf{I}_h}$. %Furthermore, for each $h\in \mathcal{G}$ and any pair $i,j\in \mathbf{I}_h$, we have $\tau_i N_i =\tau_j N_j$.
\end{lemma}
\begin{proof}
    The proof is identical to \cite[Lemma 5.16]{J23}.  
\end{proof}
\begin{proof}[Proof of Theorem~\ref{thm:logALimpliesMT}]
    By Lemma~\ref{lm:GC_v}, we have $$G_{\mathrm{B}}(f)^{\ad}_{\bQ_p}= \prod_{i\in \bbI}(\Res_{N_i/\bQ}\mathcal{G}_i)_{\bQ_p}=\prod_{i\in \bbI} G_{i,p}.$$ By Lemma~\ref{lm:goodlemma}, $G_p(f)^{\ad}$ is a product of iso-graphs over subcollections of $\{G(\mathscr{E}_{i,p},x)\}_{i\in \bbI}$. Now $G_p(f)^{\ad}\subseteq G_{\mathrm{B}}(f)^{\ad}_{\bQ_p}$. By Lemma~\ref{lm:MTGsameunip}, 
$G_{\mathrm{B}}(f)^{\ad}_{\bQ_p}$ and $G_p(f)^{\ad}$ have the same unipotent with respect to $\mu^{-1}_x$. It then follows that $G_{\mathrm{B}}(f)^{\ad}_{\bQ_p}=G_p(f)^{\ad}$. We conclude by Proposition~\ref{thm:Z0equal}. 
\end{proof}

\section{Proof of MT$_p$ $\Rightarrow$ $\mathrm{geoAO}_{p}$}\label{sec:ProofAO} 
Let $f:X\hookrightarrow \mathscr{A}_{g,\Fpbar}$ be a locally closed embedding whose image lies in the ordinary locus. Our main goal of this section is to prove the following
\begin{theorem}\label{thm:MTimpliesAO} 
If $\MTT_p$ holds for $f$, then $\mathrm{geoAO}_p$ holds for $X$, i.e., either $X$ is quasi-weakly special, or it does not contain a Zarski dense collection of positive dimensional special subvarieties. 
\end{theorem}
We achieve this by combining the theory of Mumford--Tate pairs with the geometric squeeze theorem introduced in \cite[\S 6]{J23}. 

\subsection{Preparation}\label{subsub:varconjAO}
\subsubsection{Morphism that separates factors}Let $\bbI$ be a finite index set. For each $i\in \bbI$, let $\mathcal{A}_{g_i}$ be a Siegel modular variety and let $\mathscr{A}_{g_i}$ be its canonical integral model. We use $\mathcal{A}_{g_\bbI}$ to denote the product $\prod_{i\in \bbI}\mathcal{A}_{g_i}$, regarded as a Shimura subvariety of an ambient space $\mathcal{A}_{g}$. Let $\mathscr{A}_{g_\bbI}$ be the product of $\mathscr{A}_{g_i}$'s. 

Suppose that $f:X\rightarrow \mathscr{A}_{g,\Fpbar}$ factors through $\mathscr{A}_{g_\bbI, \Fpbar}$. We will write $f_\bbI$ for the map $X\rightarrow \mathscr{A}_{g_\bbI, \Fpbar}$ (just to distinguish it from $f$). Let $f_i$ be the composition \begin{equation}\label{eq:factorize}
    X\xrightarrow{f_\bbI} \mathscr{A}_{g_\bbI, \Fpbar}\rightarrow \mathscr{A}_{g_i, \Fpbar}.
\end{equation}
We say that $f$ \textbf{separates factors} over $\mathscr{A}_{g_\bbI,\Fpbar}$, if $G_{\mathrm{B}}(f_i)^{\ad}$ is simple for each $i\in \bbI$, and $\prod\mathcal{X}_{f_i}\in \mathrm{MS}_{f_\bbI}(\mathcal{A}_{g_\bbI})$ (where $\mathcal{X}_{f_i}\in \mathrm{MS}_{f_i}(\mathcal{A}_{g_i})$); cf. \cite[\S 4.4.2]{J23}. 
It is shown in \cite[\S 4.4]{J23} that any $f$ can be modified to one which separates factors (possibly changing the ambient $\mathscr{A}_{g,\Fpbar}$). Furthermore, to prove geoAO$_p$ for $f$, it suffices to assume that $f$ separates factors, and prove the following variant: 

\begin{conj}[Mod $p$ geometric André--Oort for subvarieties that separate factors, cf. {\cite[Conjecture 4.21]{J23}}]\label{conj:varAO}
Suppose that $f:X\rightarrow \mathscr{A}_{g,\Fpbar}$ separates factors over $\mathscr{A}_{g_\bbI, \Fpbar}$. Let $\mathscr{X}_{f_i}$ be the naïve integral model of an element in $\mathrm{MS}_{f_i}(\mathcal{A}_{g_i})$. If $X$ contains a Zariski dense collection $\Xi$ of positive dimensional special subvarieties. 
Then there exists an index $i_0\in \bbI$, such that  the Zariski closure of $X$ in $\mathscr{A}_{g_\bbI,\Fpbar}^\ord$, is the product of a component of $\mathscr{X}_{f_{i_0}, \Fpbar}^\ord$ with a subvariety $Y\subseteq\prod_{i\in \bbI\setminus\{i_0\}}\mathscr{X}_{f_{i}, \Fpbar}^\ord$. %In particular, if $\bbI=\mathbf{I}_{\Xi}$, then $X$ is special.
\end{conj}
\begin{notation}\label{not:H}
    Suppose that $f$ {separates factors} over $\mathscr{A}_{g_\bbI,\Fpbar}$ as in (\ref{eq:factorize}). Let $x\in X(\Fpbar)$ be an ordinary point. \begin{enumerate}
        \item write  $\bH_{i,\bullet}$ for the various cohomology sheaves on $\mathscr{A}_{g_i}$. 
\item write  $\mathcal{X}_{f_i}$ for an element in $\mathrm{MS}_{f_i}(\mathcal{A}_{g_i})$, and denote by $\mathscr{X}_{f_i}$ its naïve integral model. 
\item  let $\mathscr{T}_{f_i,x}$ be the smallest formal subtorus of $\mathscr{A}_{g_i,\Fpbar}^{/x}$ that contains the image of $f^{/x}_i$. Note that $\mathscr{T}_{f_i,x}\subseteq \mathscr{T}_{f,x}$, where $\mathscr{T}_{f,x}$ is the smallest formal subtorus of $\mathscr{A}_{g,\Fpbar}^{/x}$ that contains the image of $f^{/x}$. Let $X_*(\mathscr{T}_{f_i,x})$ be the cocharacter lattice of $\mathscr{T}_{f_i,x}$. Define $T_{f_i,x}= X_*(\mathscr{T}_{f_i,x})_{\bQ_p}\subseteq T_{f,x}$. 
\item write $\mathscr{G}_i$ for the universal $p$-divisible group over $\mathscr{A}_{g_i}$, and denote by $\mathscr{F}_i$ the lisse sheaf $X_*(\mathscr{G}^{\loc}_i)\otimes T_p(\mathscr{G}^{\et}_i)^\vee$ on $\mathscr{A}_{g_i,\Fpbar_q}^{\ord}$ (base change to suitable $\Fpbar_q$).
%\item For an ordinary point $x\in X(\Fpbar)$, let $\mathscr{T}_{f,x}$ be the smallest formal subtorus of $\mathscr{A}_{g,\Fpbar}^{/x}$ that contains the image of $f^{/x}$, and let $X_*(\mathscr{T}_{f,x})$ be its cocharacter lattice. Recall that we have defined that $T_{f,x}= X_*(\mathscr{T}_{f,x})_{\bQ_p}$; cf, \S\ref{subsub:mdppmtintro}. 
    \end{enumerate}
Let $\bbJ\subseteq\bbI$ be a subset. 
   \begin{enumerate}\setcounter{enumi}{4}
        \item write $\mathscr{A}_{g_\bbJ}$ for the product of $\mathscr{A}_{g_i}$ over $\bbJ$.  Write $f_\bbJ$ for the composition of $f_\bbI$ with the projection $\mathscr{A}_{g_\bbI,\Fpbar}\rightarrow \mathscr{A}_{g_\bbJ,\Fpbar}$.
        
        \item write $\bH_{\bbJ,\bullet}$, $\mathscr{F}_\bbJ$ as the direct sum of the pullback of various $\mathbb{H}_{i,\bullet}$, $\mathscr{F}_i$ to $\mathscr{A}_{g_\bbJ}$. 
\item write $\mathcal{X}_{f_\bbJ}$ for the product of $\prod_{i\in \bbJ}\mathcal{X}_{f_i}$. Note that $\mathcal{X}_{f_\bbJ}\in \mathrm{MS}_{f_\bbJ}(\mathcal{A}_{g_\bbJ})$, i.e., it is a minimal special subvariety of $\mathcal{A}_{g_\bbJ}$ though whose naïve integral model $f_\bbJ$ factors.
\item  let $\mathscr{T}_{f_\bbJ,x}\subseteq \mathscr{T}_{f,x}$ be the smallest formal subtorus of $\mathscr{A}_{g_\bbJ,\Fpbar}^{/x}$ that contains the image of $f^{/x}_\bbJ$, and let $X_*(\mathscr{T}_{f_{\bbJ},x})$ be its cocharacter lattice. Define $T_{f_\bbJ,x}= X_*(\mathscr{T}_{f_\bbJ,x})_{\bQ_p}\subseteq T_{f,x}$. 
\item let $G_u(f_\bbJ)$ be the monodromy group $G(\bH_{\bbJ,u,X},x)$. Let $G_\mathrm{B}(f_\bbJ)$ be the generic Mumford--Tate group of $\bH_{\bbJ,\mathrm{B},\mathcal{X}_{f_\bbJ}}$. We will say that \textbf{$\MTT_p$ holds for $f_\bbJ$}, if $G_\mathrm{B}(f_\bbJ)_{\bQ_u}= G_u(f_\bbJ)^\circ$.   
    \end{enumerate}
\end{notation}
\begin{remark}
  Suppose that $f$ separates factors over $\mathscr{A}_{g_\bbI,\Fpbar}$ as in (\ref{eq:factorize}). Then $G_\mathrm{B}(f_\bbI)^{\ad}=G_\mathrm{B}(f)^{\ad}$, and $G_u(f_\bbI)^{\circ\ad}= G_u(f)^{\circ\ad}$. 
  It is a consequence of Proposition~\ref{thm:Z0equal} that $\MTT_p$ holds for $f_\bbI$ if and only if $\MTT_p$ holds for $f$. We will use this fact without mentioning. 
\end{remark}

\begin{ass}\label{ass:AOsep}
  From now on, we will assume the following \begin{enumerate}
      \item  $f$ separates factors over $\mathscr{A}_{g_\bbI,\Fpbar}$ as in (\ref{eq:factorize}).
      \item $X$ contains a Zariski dense collection $\Xi$ of positive dimensional special subvarieties.
      \item  There are two further simplifications that one can make on the collection $\Xi$: 
\begin{itemize}
    \item We can arbitrarily shrink $X$ to open dense, and replace $\Xi$ by its restriction. 
    \item   We can assume that each $Z\in \Xi$ is smooth and geometrically connected, by throwing away the singular loci, and re-indexing the components. 
\end{itemize} We will assume that $\Xi$ is simplified as above. 
  \end{enumerate}
\end{ass}
Our goal is to show Conjecture~\ref{conj:varAO} under the condition that $\MTT_p$ holds for $f$. 

\subsubsection{Serre--Tate representations}
Let $G$ be a reductive group over a characteristic 0 field $F$. Consider a cocharacter $\mu: \bG_m\rightarrow G$. Recall that $\Lv_{G,\mu}$ and $U_{G,\mu^{-1}}$ are Levi and opposite unipotent with respect to $\mu$. The adjoint representation 
$$\ad_{G,\mu}:\Lv_{G,\mu} \acts \Lie U_{G,\mu^{-1}}$$
is called the \textbf{Serre--Tate representation} (cf. \cite[\S6.2.2]{J23}). In the most relevant situations, $G$ is a $p$-adic
monodromy group, $\mu$ is a Hodge cocharacter, and the Serre--Tate representation gives rise to a \textbf{Serre--Tate lisse sheaf} in the sense of \cite[\S 6.2.1]{J23}. Representations of this sort play a fundamental role in the studying the mod $p$ geometric André--Oort conjecture. 

\begin{lemma}\label{lm:simplecomponentTtrep}
    Suppose that $F$ is algebraically closed. Let $(G,\rho,\mu)$ be a simple Mumford--Tate triple of weights $\{0,1\}$. Then the corresponding Serre--Tate representation $\ad_{G,\mu}$ is irreducible. 
\end{lemma}
\begin{proof}
    Since simple Mumford--Tate triples are classified in
 \cite[Table 4.2]{P98}, we can check the irreducibility case by case. For $n\in \bN$, let $\rho_n:\GL_{n}\acts F^{\oplus n}$ be the standard representation. For column 1 and 4 of \cite[Table 4.2]{P98}, there exists some $1\leq s\leq r$ such that the representation $\ad_{G,\mu}$ is isomorphic to the irrep $\rho_s\otimes \rho^{\vee}_{r+1-s}:\GL_{s}\times \GL_{r+1-s}\acts F^{\oplus s}\otimes F^{\oplus r+1-s}$. For column 2 of \cite[Table 4.2]{P98}, $\Lv_{G,\mu}$ is an almost direct product of a central torus $\bG_m$ with $\GL_r$, and    $\ad_{G,\mu}|_{\GL_r}\simeq \mathrm{Sym}^2\rho_r$. The rest of the columns can be checked in a similar manner.
\end{proof}
\subsubsection{The geometric squeeze theorem} Let $$f_0:X_0\hookrightarrow\mathscr{A}_{g_\bbI,\Fpbar_q}\hookrightarrow \mathscr{A}_{g,\Fpbar_q}$$ be a finite field model of $f$ that separate factors as in (\ref{eq:factorize}). Let the notation be as in Notation~\ref{not:H}. To ease notation, write \begin{equation}\label{eq:G=}
G:=G_p(f_\bbI)^\circ=G(\bH_{\bbI,p,X},x)^\circ.
\end{equation}
By Theorem~\ref{Thm:Tatelocal}, the Serre--Tate representation $\ad_{G,\mu_x}$ identifies with $$G(\gr \bH_{\bbI,p,X},x)^\circ\acts T_{f_\bbI,x}=T_{f,x}.$$ 
\begin{theorem}[{The geometric squeeze theorem}, {\cite[Theorem 6.2]{J23}}]\label{thm:squeeze} 
Let the setup be as in Assumption~\ref{ass:AOsep}. Possibly shrinking $X$, there exists a nonzero saturated $\bZ_p$-lisse sheaf $\mathcal{R}_f\subseteq \mathscr{F}_{\bbI,X_0'}$ over a Galois cover $X_0'\rightarrow X_0$, such that for any $x\in X(\Fpbar)$, the following are true:
\begin{enumerate}
\item\label{it:squeeze1} 
${R}_{f,x}:=(\mathcal{R}_{f,x})_{\bQ_p}\subseteq T_{f,x}$ is a subspace invariant under $\ad_{G,\mu_x}$ \footnote{ By $\mathcal{R}_{f,x}$ we mean the fiber of $\mathcal{R}_{f}$ at an $\Fpbar$-point of $X_0'$ that lifts $x$, identified as a sublattice of $X_*(\mathscr{T}_{f,x})$. Different lifts of $x$ may yield different $\mathcal{R}_{f,x}$ as sublattices of $X_*(\mathscr{T}_{f,x})$, but the theorem works for all of them.}.
% \item\label{it:squeeze1.5} The set of indices $\bbJ\subseteq \bbI$ such that $\mathcal{R}_{f,x}$ projects nontrivially to $\mathscr{C}_{i,x}$ is independent of $x$. 
  \item\label{it:squeeze2} $X^{/x}$ is sandwiched between two formal tori: $$\mathcal{R}_{f,x}\otimes \bG_m^\wedge\subseteq X^{/x}\subseteq X_*(\mathscr{T}_{f,x})\otimes \bG_m^\wedge.$$
   \item\label{it:squeeze3}
There is a $Z\in \Xi$, such that for $x\in Z(\Fpbar)$:
$${T}_{f|_Z,x}\subseteq {R}_{f,x}\subseteq T_{f,x},$$ 
and in addition, for each $i\in \bbI$, the projection image of $Z$ in $\mathscr{A}_{g_i,\Fpbar}$ is positive dimensional if and only if the projection image of $\mathcal{R}_{f}$ in $\mathscr{F}_{i,X_0'}$ has positive rank.
\end{enumerate}
\end{theorem}
\subsection{Proof of Theorem~\ref{thm:MTimpliesAO}} We will first treat the case where $\# \bbI=1$, and prove a stronger claim that $X$ is already special (as opposed to quasi-weakly special). After that, we treat the general case. The proof is a generalization of the product GSpin case given in \cite[\S 6]{J23}. 

\subsubsection{The case where $\#\bbI=1$}\label{subsub:I=1} We first consider the case where $\bbI=\{1\}$. Then $G_{\mathrm{B}}(f)^{\ad}$ is a simple $\bQ$-group. We will drop the subscript $\bbI$ or $i$ since it plays no role. Following \S\ref{sub:simplecase}, let $G_{\mathrm{B}}(f)^{\ad}=\Res_{F/\bQ}\mathcal{G}$ and define $C\subseteq\bC$ and $\Sigma$ in exactly the same manner. Let $G$ be as in (\ref{eq:G=}). For $\sigma\in \Sigma$ and $v\in \mathrm{fpl}(C)$, let $\mathscr{E}_{\sigma,v}$ be the $C_u$-coefficient objects as per \S\ref{sub:simplecase}. Let $\Sigma^{\mathrm{ST}}\subseteq \Sigma$ be the subset such that $\{\Lie \mathcal{G}_{\sigma,\bC_p}\}_{\sigma\in \Sigma^{\mathrm{ST}}}$ is the set of Serre--Tate factors of $(G_{\mathrm{B}}(f)_{\bC_p},\rho_{\mathrm{B}})$ with respect to $\mu_x$. The set $\Sigma^{\mathrm{ST}}$ does not depend on $x$ chosen. Let 
\begin{equation}\label{eq:HST}
\mathcal{H}^{\mathrm{ST}}= \prod_{\sigma\in \Sigma^{\mathrm{ST}}}\mathcal{G}_{\sigma,C}.  
\end{equation}
Suppose that $\MTT_p$ holds for $f$, so $G=G_{\mathrm{B}}(f_\bbI)_{\bQ_p}$. By Theorem~\ref{Thm:Tatelocal}, we have
\begin{align*}
    &T_{f,x}= \Lie U_{G,\mu_x^{-1}}=\Lie U_{G_{\mathrm{B}}(f)_{\bQ_p},\mu_x^{-1}},\\
 &T_{f,x,\bC_p}=\Lie U_{ \mathcal{H}^{\mathrm{ST}}_{\bC_p},\mu_x^{-1}}.  
\end{align*}

\begin{theorem}\label{thm:AOvarsimplecase}
Notation as above. Suppose that $\MTT_p$ holds for $f$. Then $X$ is special. 
\end{theorem}
\begin{proof}In the following, we assume that our $X$ is sufficiently shrunk. Let the lisse sheaf $\mathcal{R}_f$ be as in Theorem~\ref{thm:squeeze}. Let $x\in X(\Fpbar)$,  Theorem~\ref{thm:squeeze}(\ref{it:squeeze1}) implies that ${R}_{f,x}\subseteq T_{f,x}$ is invariant under $\ad_{G,\mu_x}$. In the following, we show that $\dim X= \rk T_{f,x}$. Which implies that $X$ is special (conditional on $\mathrm{MT}_p$ for $f$). % It suffices to show the equality after base change to $\bC_p$.

First, note that $\ad_{G_{\bC_p},\mu_x}$ factors surjectively through $\ad_{\mathcal{H}^{\mathrm{ST}}_{\bC_p},\mu_x}$, so ${R}_{f,x,\bC_p}$ is invariant under $\ad_{\mathcal{H}^{\mathrm{ST}}_{\bC_p},\mu_x}$. On the other hand, \begin{equation}\label{eq:factorsimplerep}
\ad_{\mathcal{H}^{\mathrm{ST}}_{\bC_p},\mu_x}= \bigoplus_{\sigma\in \Sigma^{\mathrm{ST}}} \ad_{\mathcal{G}_{\sigma,\bC_p},\mu_x}.
\end{equation}
Since for each $\sigma\in \Sigma^{\mathrm{ST}}$, $\mathcal{G}_{\sigma,\bC_p}$ comes from a simple Mumford--Tate triple of weights $\{0,1\}$, Lemma~\ref{lm:simplecomponentTtrep} implies that each $\ad_{\mathcal{G}_{\sigma,\bC_p},\mu_x}$ in (\ref{eq:factorsimplerep}) is irreducible. As a result, there exists a $\Sigma'\subseteq \Sigma^{\mathrm{ST}}$, independent of $x$, such that 
\begin{equation}\label{eq:factor_domain}
    {R}_{f,x,\bC_p}= \bigoplus_{\sigma\in \Sigma'} \Lie U_{\mathcal{G}_{\sigma,\bC_p},\mu_x^{-1}}.
\end{equation}
Let $Y$ be the smallest closed subvariety of $\mathscr{A}_{g,\Fpbar}$ that contains the formal torus $\mathcal{R}_{f,x}\otimes \mathbb{G}_m^{\wedge}$. Note that an open neighborhood of $Y$ around $x$ is contained in $X$.  \begin{claim} 
$Y^{/x}$ is a union of formal subtori.   
\end{claim}
This is the type of result that eventually follows from Chai's rigidity theorem (\cite{Chai08}). Let's explain the reason. First, the $p$-adic monodromy group $G_p(x)^\circ=G_{\mathrm{B}}(x)\otimes \bQ_p$ preserves $R_{f,x,\bQ_p}$, hence $G_p(x)(\bQ_p)$ contains an open subgroup $\mathcal{O}_{p,x}$ that preserves the integral lattice $\mathcal{R}_{f,x}$. On the other hand, $G_{\mathrm{B}}(x)(\bQ)$ is $p$-adically dense in $G_p(x)^\circ$, so $\mathcal{O}_{x}:=\mathcal{O}_{p,x}\cap G_{\mathrm{B}}(x)(\bQ)$ is $p$-adically dense in $\mathcal{O}_x$. Possibly shrinking $\mathcal{O}_{p,x}$, each element in $\mathcal{O}_{x}$ corresponds to a self-isogeny of $A_x$, inducing a Hecke correspondence of $\mathscr{A}_{g}$ fixing $x$. Such Hecke correspondence needs to fix $\mathcal{R}_{f,x}$, hence also $Y$ (hence $Y^{/x}$ as well). Up to further shrinking, the  $\mathcal{O}_{p,x}$-action on $\mathscr{A}_{g,\Fpbar}^{/x}$  fixes $Y^{/x}$ (more precisely, one can use denseness of $\cO_x$ in $\cO_{p,x}$ to show fixedness for each Artinian truncation of $Y^{/x}$, then take limit). Let $\mu_x:\bG_m\rightarrow G_p(x)$ be the canonical Hodge cocharcter. Then for $n\gg 0$, the element  $\mu_x(1+p^n)\in G_p(x)(\bQ_p)$ lies in $\cO_{p,x}$, hence fixing $ Y^{/x}$. Since  
$\mu_x$ acts on $\mathrm{Lie} U_{G_\mathrm{B}(f)_{\bQ_p},\mu_x^{-1}} $ by pure weight 1, $\mu_x(1+p^n)$ induces multiplication by $1+p^n$ map on $\mathscr{A}^{/x}_{g,\Fpbar}$. Therefore $Y^{/x}$ is fixed by  multiplication by $1+p^n$ map for all $n\gg 0$, so Chai's rigidity theorem (\cite{Chai08}) implies that $Y^{/x}$ is a union of formal subtori, proving the \textit{Claim}.   
 
 Now, the normalization $Y'$ is smooth over the finitely many points above $x$. Up to shrinking, we can assume that $Y'$ is smooth. Let $x'$ be a point 
in $Y'$ above $x$ such that $(Y')^{/x'}$ maps to a component $\mathscr{T}_0\subseteq Y^{/x}$ containing $\mathcal{R}_{f,x}\otimes \mathbb{G}_m^{\wedge}$. We will identify $(Y')^{/x'}$  with $\mathscr{T}_0$. 

%If $\mathrm{logAL}_p$ holds for $f':Y'\rightarrow \mathscr{A}_{g,\Fpbar}$ (with base point $x'$), then we are already done: let $\mathcal{S}\subseteq \mathrm{MS}_{f}(\mathcal{A}_{g})$ and let $\mathcal{S}'\in \mathrm{MS}_{f'}(\mathcal{A}_{g})$ be such that $\mathcal{S}'\subseteq \mathcal{S}$. Then $\mathrm{logAL}_p$ implies that $(\mathscr{S}'_{\Fpbar})^{/x}$ contains $\mathscr{T}_0$ as an irreducible component. Note that $\mathscr{T}_0$ is quite big by (\ref{eq:factor_domain}):  it implies that $\mathcal{S}_{\bC}'$ contains a positive dimensional factor of the symmetric domain of $\mathcal{S}_{\bC}$. Then the simplicity of $\mathcal{S}$ would show that $\mathcal{S}'=\mathcal{S}$. It follows easily that $X$ is special. 
To ease notation, let $U$ be the opposite unipotent of $G_p(f')$ with respect to $\mu_{x'}$. 
Then $U_{\bC_p}$ contains the right hand side of (\ref{eq:factor_domain}). 
 For each $\sigma\in \Sigma'$, $\mathcal{G}^{\ad}_{\sigma,\bC_p}$ is generated by its unipotent and opposite unipotent with respect to $\mu_x$. As a result, 
\begin{equation}\label{eq:in common mom}
    G_p(f')^{\ad}\supseteq \bigoplus_{\sigma\in \Sigma'} \mathcal{G}^{\ad}_{\sigma,\bC_p}. 
\end{equation}
Finally, the same argument as in Lemma~\ref{lm:monodromyEandMT} apply to show that 
 $G_p(f')^{\ad}=(\Res_{F/\bQ}\mathcal{G})_{\bQ_p}$ \footnote{The imput of logAL$_p$ in  Lemma~\ref{lm:monodromyEandMT} is only used to show that $G(\mathscr{E}_p,x)$ and $(\Res_{F/\bQ}\mathcal{G})_{\bQ_p}$ has a simple factor in common, which we already have by (\ref{eq:in common mom}) in our situation}. It follows that $X_*(\mathscr{T}_0)_{\bQ_p}=T_{f,x}$. In particular $\dim Y= \rk T_{f,x}$. Since $\rk T_{f,x}\geq \dim X\geq \dim Y$, we see that  $\rk T_{f,x}= \dim X$ as desired. 
\end{proof}\iffalse
\begin{lemma}\label{lm:Hecke_lemma}
    Let $x\in \mathscr{A}_{g}(\Fpbar)$ be an ordinary point. Then for all $n\gg 0$, the multiplication by $1+p^n$ map on the Serre--Tate torus $\mathscr{A}^{/x}_{g,\Fpbar}$ is a $\bZ_p$-linear combination of endomorphisms of $\mathscr{A}^{/x}_{g,\Fpbar}$  induced by Hecke translations preserving $x$ (i.e., a Hecke correspondence  $\mathscr{H}\subseteq \mathscr{A}_{g}\times \mathscr{A}_{g}$ passing through $(x,x)$).  
\end{lemma}
\begin{proof}
Let $\mu_x: \mathbb{G}_m\rightarrow G_{p}(x)$ be the canonical Hodge cocharacter. Then $\mu_x(1+p^n)$ induces multiplication by $1+p^n$ map on $\mathscr{A}^{/x}_{g,\Fpbar}$, as 
$\mu_x$ acts on $\mathrm{Lie} U_{G_\mathrm{B}(f)_{\bQ_p},\mu_x^{-1}} $ by pure weight 1. Now $\mu_x(1+p^n)$ acts on the rational Dieudonné module of $A_x[p^\infty]$, hence corresponding to an element in $\End(A_x)\otimes \bQ_p$ by the crystalline Tate conjecture. For $n\gg 0$, 
 this element lies in $\End(A_x)\otimes \bZ_p$, hence lying in the $\bZ_p$-span of self-isogenies of $A_x$.  
\end{proof}
\fi
\subsubsection{The general case}\label{subsub:pfofAOgeneral} Now we assume that $G_{\mathrm{B}}(f)^{\ad}=\prod_{i\in \bbI}\Res_{F_i/\bQ}\mathcal{G}_i$, where each $\Res_{F_i/\bQ}\mathcal{G}_i$ is simple over $\bQ$ and $\mathcal{G}_i$ is absolutely simple. Let $C,\Sigma_i$ be as in \S\ref{subsub:MTgen}. For each $i\in \bbI$, define $\Sigma_i^{\mathrm{ST}}$ and $\mathcal{H}^{\mathrm{ST}}_i$ as in \S\ref{subsub:I=1}. Let $$\mathcal{H}^{\mathrm{ST}}=\prod_{i\in \bbI}\mathcal{H}^{\mathrm{ST}}_i.$$ 
Let $G=G(\bH_{\bbI,p,X},x)^\circ$. If $\MTT_p$ holds for $f$, we then have $G=G_{\mathrm{B}}(f_\bbI)_{\bQ_p}$. It is easy to see that $$T_{f,x,\bC_p}=\Lie U_{ \mathcal{H}^{\mathrm{ST}}_{\bC_p},\mu_x^{-1}}.$$

\begin{proof}[Proof of Theorem~\ref{thm:MTimpliesAO}]
Let the notation be as in Notation~\ref{not:H}. 
 It suffices to prove Conjecture~\ref{conj:varAO} under Assumption~\ref{ass:AOsep}. In the following, we assume that our $X$ is sufficiently shrunk. Let $\mathcal{R}_{f}$ be as in Theorem~\ref{thm:squeeze}. Let $\bbJ\subseteq \bbI$ be the set of indices consisting of those $i$ such that the projection of $\mathcal{R}_{f}$ to $\mathscr{F}_{i,X_0'}$ is not 0. Then $\#\bbJ > 0$. Project $X$ to $\prod_{i\in \bbJ}\mathscr{X}_{f_{i},\Fpbar}^\ord$ and let $\cC=\prod_{i\in \bbJ}\cC_i$ be a component of the later space containing the image. Project $X$ to $\prod_{i\in \bbI\setminus \bbJ}\mathscr{X}_{f_{i},\Fpbar}^\ord$ and let $Y$ be the Zariski closure of the image in the later space.  We have\begin{equation}\label{eq:proddd}
\overline{X}\subseteq \cC\times Y.
\end{equation}
It suffices to show that (\ref{eq:proddd}) is an equality. Let $x\in X(\Fpbar)$. Since $\MTT_p$ holds for $f$, we have \begin{equation*}
    T_{f,x}=\bigoplus_{i\in \bbI} T_{f_i,x}.
\end{equation*}
On the other hand, $${R}_{f,x}\subseteq   T_{f_\bbJ,x}=\bigoplus_{i\in \bbJ}T_{f_{i},x}.$$
Theorem~\ref{thm:squeeze}(\ref{it:squeeze1}) implies that ${R}_{f,x}$ is invariant under $\ad_{G,\mu_x}$. Note that $\ad_{G,\mu_x}$ factors surjectively through $\ad_{\mathcal{H}^{\mathrm{ST}}_{\bC_p},\mu_x}$, so  $R_{f,x,\bC_p}$ is also a subrep of $\ad_{\mathcal{H}^{\mathrm{ST}}_{\bC_p},\mu_x}$.  We again obtain a decomposition similar to (\ref{eq:factorsimplerep}):
 \begin{equation*}%\label{eq:factorsimplerep2}
\ad_{\mathcal{H}^{\mathrm{ST}}_{\bC_p},\mu_x}= \bigoplus_{i\in \bbI}\bigoplus_{\sigma_i\in \Sigma^{\mathrm{ST}}_i} \ad_{\mathcal{G}_{i,\sigma_i,\bC_p},\mu_x}.
\end{equation*}
Lemma~\ref{lm:simplecomponentTtrep} again implies that each $\ad_{\mathcal{G}_{i,\sigma_i,\bC_p},\mu_x}$ in (\ref{eq:factorsimplerep}) is irreducible. Let $R_{f_i,x}$ be the projection of ${R}_{f_i,x}$ to $T_{f_i,x}$, it then follows that $${R}_{f,x}=\bigoplus_{i\in \bbJ} R_{f_i,x},$$ equivariant with the decomposition of $T_{f,x}$.

Let $\mathcal{R}_{f_i,x}\subseteq X_*(\mathscr{T}_{f_i,x})$ be the saturated sublattice which admits ${R}_{f_i,x}$ as rationalization. Let $x_i$ be the projection of $x$ to $\mathscr{X}^{\ord}_{f_i,\Fpbar}$, so that we can write $x=(x_i)_{i\in \bbI}$. Let ${x}'=(x_i)_{i\in \bbI\setminus\bbJ}$ be the projection of $x$ to $Y$. For $i\in \bbJ$, write $\mathcal{R}_{f_i,x_i}, {T}_{f_i,x_i}$ for  $\mathcal{R}_{f_i,x}, {T}_{f_i,x}$ but viewed as having base point $x_i$, and let $Y_i$ be the smallest closed subvariety of $\mathscr{X}^{\ord}_{f_i,\Fpbar}$ that contains the formal torus $\mathcal{R}_{f_i,x_i}\otimes \bG_m^\wedge$. Arguing as in the proof of Theorem~\ref{thm:AOvarsimplecase}, we find that $Y_i^{/x_i}$ is a union of formal tori with a component whose rational cocharacter lattice is $T_{f_i,x_i}$, hence $Y_i=\mathcal{C}_i$. This implies that $\mathcal{C}\times \{x'\}$ is the smallest closed subvareity of $\prod_{i\in \bbJ}\mathscr{X}^{\ord}_{f_i,\Fpbar}$ containing $\mathcal{R}_{f,x}\otimes \bG^\wedge_m$. In particular 
$\overline{X}\supseteq \cC\times \{x'\}$.

%Now, Since $\mathcal{R}_{f,x}\otimes \bG_m^\wedge\subseteq X^{/x}\subseteq X_*(\mathscr{T}_{f,x})\otimes \bG_m^\wedge$ by Theorem~\ref{thm:squeeze}(\ref{it:squeeze2}), and 
%$$X^{/x}\supseteq \mathcal{R}_{f,x}\otimes \bG_m^\wedge\supseteq (\cC\times \{x'\})^{/x}.$$ 
%So $\overline{X}\supseteq \cC\times \{x'\}$.
Now let $x$ vary over $X(\Fpbar)$. Since the projection of $X$ to $Y$ is Zariski dense, we find that $\overline{X}\supseteq \cC\times Y$. This implies that (\ref{eq:proddd}) is an equality.   

\end{proof}

\section{Mod $p$ variants of classical methods}\label{Sec:classical} 
In this chapter, we adapt classical methods to establish $\MTT_p$ in various cases. The main results are the characteristic $p$ analogues of results of Commelin, Pink, Tankeev, and Moonen--Zarhin; cf. \cite{JMC,P98,MZ95,Tank83}. The techniques we use for establishing $\MTT_p$ are completely expected, but the truly interesting applications come when we combine these results with the previously established implications 
$$\MTT_p\Leftrightarrow \mathrm{logAL}_p\Rightarrow \mathrm{geoAO}_p.$$
This enables us to obtain nice consequences on the  unlikely intersections side of the story, see Corollary~\ref{cor:Commelinconsequence} and Corollary~\ref{cor:geoAOdim5}. 

In the following, let $k/\Fpbar_q$ be a function field. 

\begin{theorem}[Analogue of {\cite{JMC}}]\label{modpCommelin}
    Let $A_1,A_2$ be ordinary abelian varieties over $k$. Suppose that $\MTT_p$ holds for $A_1$ and $A_2$, then it holds for $A_1\times A_2$. 
\end{theorem}
\begin{theorem}[Analogue of {\cite{P98}}]\label{modpPink}
    Let $A$ be an ordinary abelian variety over $k$. Suppose that $A$ is of Pink type; cf. Definition~\ref{def:pinktype}. Then $\MTT_p$ holds for $A$. %In addition, if $A$ is geometrically simple of type \ref{it:P1}, then $G_\mathrm{B}(A)=\GSp_{2\dim A,\bQ}$.   
\end{theorem}
\begin{theorem}[Analogue of {\cite{Tank83}}]\label{modpTank}
    Let $A$ be an ordinary abelian variety of prime dimension over $k$. Suppose that $A$ is geometrically simple. Then $\MTT_p$ holds for $A$.  
\end{theorem}
\begin{theorem}[Analogue of {\cite{MZ95}}]\label{modpMoonen-Zarhin}
    Let $A$ be an ordinary abelian fourfold over a function field $k$. Suppose that $A$ is not of $p$-adic Mumford type; cf. Definition~\ref{def:cohomologicalMType}. Then $\MTT_p$ holds for $A$.  
\end{theorem}

\begin{corollary}\label{cor:Commelinconsequence}
    Suppose $f:X\rightarrow \mathscr{A}_{g_1,\Fpbar}\times \mathscr{A}_{g_2,\Fpbar}\times ...\times \mathscr{A}_{g_n,\Fpbar}$ be a morphism whose image lies generically in the ordinary locus. If $\mathrm{logAL}_p$ holds for each factor\footnote{For example, $\mathrm{logAL}_p$ trivially holds for $f_i$ if $f_i$ is dominant. More generally, $\mathrm{logAL}_p$ holds for $f_i$ if $f_i$ dominates a special subvariety.} $f_i:X\rightarrow \mathscr{A}_{g_i,\Fpbar}$, then $\MTT_p$ and $\mathrm{logAL}_p$ hold for $f$. If $f$ is further more a locally closed immersion, then $\mathrm{geoAO}_p$ holds for $X$.  
\end{corollary}
\begin{corollary}\label{cor:geoAOdim5}
Suppose $f:X\hookrightarrow \mathscr{A}_{5,\Fpbar}$ is a locally closed immersion whose image lies generically in the ordinary locus. Then $\mathrm{geoAO}_p$ holds for $X$.
\end{corollary}
\begin{corollary}\label{cor:Ttconj}
    The Tate conjecture holds for abelian fourfolds over char 0 finitely generated fields or ordinary abelian fourfolds over char $p$ finitely generated fields.  
\end{corollary}

\subsection{Mumford--Tate conjecture for products}In this section we prove Theorem~\ref{modpCommelin}. Our proof is a modification of that in \cite{JMC}: most group theory treatments carry through, while some geometric constructions need to be modified with care. 

For reader's convenience, we give an overview of Commelin's original proof over number fields; see also \cite[\S 1.5]{JMC}. Instead of working with abelian varieties, one works with the ``adjoint motive'' $M_i$ associated to $A_i$ (over Betti realization, this amounts to consider the adjoint representation of $G_{\mathrm{B}}(A_1)^{\ad}$). It suffices to prove the Mumford--Tate conjecture for $M:=M_1\oplus M_2$. For this, one can assume that  $M_i$ is irreducible. If the Mumford--Tate conjecture holds for $A_1$ and $A_2$, then it holds for $M_1$ and $M_2$. So $G_\mathrm{B}(M_1)_{\bQ_l}=G_l(M_1)^\circ$. 
Assume $G_l(M)^\circ\subsetneq G_l(M_1)^\circ\times G_l(M_2)^\circ$, it suffices to show that $G_\mathrm{B}(M)\subsetneq G_\mathrm{B}(M_1)\times G_\mathrm{B}(M_2)$. The key point is to use Deligne's construction (\S\ref{subsub:deligneconstruction}) to attach to $M_i$ an (explicit) abelian variety $\widetilde{A}_i$ with the property that $G_\mathrm{B}(\widetilde{A}_i)^{\ad}= G_\mathrm{B}(M_i)$. One then uses Faltings' theorem to construct an isogeny between $\widetilde{A}_1$ and $\widetilde{A}_2$ to conclude. 

The overall structure of our proof in characteristic $p$ is not so much different from Commelin's. However, there is some difference that need to be pointed out:
\begin{enumerate}
\item\label{it:reason1} We will work with a geometric version of the Mumford--Tate conjecture as we did in the previous chapters: suppose that $X$ is a smooth connected $\Fpbar$-variety with a map $f:X\rightarrow \mathscr{A}_{g_1,\Fpbar}^{\ord}\times \mathscr{A}_{g_2,\Fpbar}^{\ord}$. Let $f_1,f_2$ be the projection of $f$ to each factor. We will show that if $\MTT_p$ holds for $f_1$ and $f_2$, then it holds for $f$. 
\item Instead of working with motives, we work with collections of coefficient objects (Betti, étale, and crystalline); cf. \S\ref{subsub:translations1}. We do this simply to avoid the technicalities for motives in characteristic $p$ -- in particular when it is tricky to makes sense of ``Betti realizations''. Nevertheless, to cope with the notation in \cite{JMC}, we will still use ``$M$'' to denote a collection of coefficient objects.  
\item As opposed to Commelin, we work explicitly with Shimura varieties, especially when performing Deligne's construction. In fact, we need to construct abelian schemes $\widetilde{A}_i$ in characteristic $p$. This involves a study of the mod $p$ reduction of the Hodge type Shimura variety that arises from Deligne's construction.
\end{enumerate}
\subsubsection{Families of adjoint coefficient objects}\label{subsub:translations1} %Commelin is using the language of abelian motives, while we are using the language of Shimura varieties. Here we will translate his language of hyperadjoint abelian motives into compatible systems of coefficient objects. 

Let $f:(X,x)\rightarrow \mathscr{A}_{g,\Fpbar}^{\ord}$, and let $f_0:X_0\rightarrow \mathscr{A}_{g,\Fpbar_q}^{\ord}$  be a finite field model. Let $A_{X_0}$ be the pullback abelian scheme over $X_0$. We define $M$ to be a collection of cohomology sheaves associated to $A_{X_0}$:  $$M:=\{H_{\bullet}(M),\bullet=\mathrm{B} \text{ and } u\in\mathrm{fpl}(\bQ)\}$$ that plays essentially the same role as the ``hyperadjoint abelian motive associated to $A_{X_0}$'' in the sense of \cite{JMC}: $H_{\mathrm{B}}(M)$ is the variation of rational Hodge structures on $\mathcal{X}_{f}\in \mathrm{MS}_f(\mathcal{A}_{g})$ associated to the adjoint representation $$G_\mathrm{B}(f)\acts \Lie G_\mathrm{B}(f)^{\ad}.$$ For a prime $u\in \mathrm{fpl}(\bQ)$, $H_u(M)$ is the $\bQ_u$-coefficient object (on an étale cover of $X_0$) corresponding to the adjoint representation \begin{equation}\label{eq:coeffobj}
    G_u(f)^\circ\hookrightarrow G_\mathrm{B}(f)_{\bQ_u}\acts \Lie G_\mathrm{B}(f)^{\ad}_{\bQ_u}.
\end{equation}
The symbol ``$M$'' is chosen to mean ``motive'' (and to cope with the notation in \cite{JMC}), though we are not using this language. In our paper, $M$ is simply understood as a collection of coefficient objects. 

Write $G_{\bullet}(M)$ for the monodromy group of $H_\bullet(M)$. We have $G_{\mathrm{B}}(M) = G_\mathrm{B}(f)^{\ad}$, and $G_{u}(M) = G_u(f)^{\ad}$.  Suppose that $E\subseteq \End(H_\mathrm{B}(M))$ is a number field. Following \cite[\S 4.1]{JMC2},  
\begin{enumerate}
    \item denote by $\Lambda_u$ the set of places of $E$ lying above $u$. Then $E_u:=E\otimes_{\bQ} \bQ_u =\prod_{\lambda\in \Lambda_l} E_\lambda$. Let $H_\lambda(M)=H_u(M)\otimes_{E_u} E_\lambda$. We have a splitting $H_u(M)=\bigoplus_{\lambda\in \Lambda_u} H_\lambda(M)$,
    \item denote by $G_\lambda(M)$ the monodromy group of $H_\lambda(M)$, which is a $E_\lambda$-algebra group.
    \item we also write $H_\Lambda(M)$ for the collection of coefficient objects $\{H_{\lambda}|\lambda\in \mathrm{fpl}(E)\}$. 
\end{enumerate}
\begin{lemma}\label{lm:Fcompatible}
   Notation as above, $H_\Lambda(M)$ is a weakly $E$-compatible family of coefficient objects in the sense of \cite[\S 3.4]{J23}.  
\end{lemma}
\begin{proof}
   The proof is similar to Lemma~\ref{lm:compatibleinsinglecase}. The key point is the existence of canonical lift for ordinary abelian varieties. We leave the detailed proof to the reader.
\end{proof}
This lemma is a characteristic $p$ analogue of \cite[Theorem 5.13]{JMC}. The original version is more subtle (since there is in general no canonical lift for a non-ordinary abelian variety over finite field), and requires Kisin's result (\cite[Theorem 2.2.3]{Kisinmodp}) as a main input.

\iffalse
When $G_{\mathrm{B}}(M)$ is simple, then $\End(H_\mathrm{B}(M))$ is a totally real number field. Moreover, $G_{\mathrm{B}}(M)$ is the Weil restriction of an absolutely simple reductive group $\mathcal{G}$ over $F$. We have 
$$G_\mathrm{B}(M)_{\bQ_u}= \prod_{\lambda\in \Lambda_u} \Res_{F_\lambda/\bQ_u}\mathcal{G}_{F_\lambda}.$$

Note that this induces a splitting of the Lie algebra on the right hand side of (\ref{eq:coeffobj}). For $\lambda\in \Lambda_{u}$, we write $H_\lambda(M)$ for the direct summand of $H_u(M) \otimes E_\lambda$ corresponding to $\Lie\mathcal{G}_{E_\lambda}$.
\fi

\subsubsection{Deligne's construction}\label{subsub:deligneconstruction}  %We will follow Deligne's oringal paper, but note that Commelin also succeeded in giving a Shimura-free reinterpretation. 
%\textcolor{red}{Make notation changes!!}
Let $(G,\cD)$ be a simple adjoint Shimura datum. As explained in \cite[Section 2.3.4]{Deligne}, $G=\Res_{E/\bQ}\mathcal{G}$ for some totally real field $E$ and a absolute simple adjoint group over $E$. Let $\Sigma_E$ be the set of real embeddings of $E$. We can decompose $\Sigma_E=\Sigma_{E,c}\sqcup \Sigma_{E,nc}$, where $\Sigma_{E,c}$ (\textit{resp}. $\Sigma_{E,nc}$) consists of embeddings $\sigma$ such that $\mathcal{G}_{\sigma,\bR}$ is compact (\textit{resp}. non-compact).

The Dynkin diagram $\Delta$ of $G$ is a disjoint union of (connected) Dynkin diagrams $\Delta_\tau$ of $G_{\tau,\bC}$, $\tau\in \Sigma_E$. Let $\mu$ be the Hodge cocharacter of a point in $\cD$. Then $\mu$ gives rise to a collection of special nodes in $ \Delta$. Following the terminology of \cite{JMC}, we will call the data $(\Delta,\mu)$ a \textbf{Deligne--Dynkin diagram}. Note that the absolute Galois group $\Gal(\bQ)$ acts on $\Delta$. See \cite[Table 1.3.9]{Deligne} for a complete classification of connected Deligne--Dynkin diagrams, and see also \cite[\S 6.9]{JMC} for a detailed explanation.

Now fix an imaginary quadratic extension $F/E$. Let $\Sigma_F$ be the set of complex embeddings of $F$. Let $\Phi\subseteq \Sigma_F$ be a partial CM type that contains exactly one complex embedding lying above each $\sigma\in \Sigma_{E,c}$, and contains no complex embedding lying above $\Sigma_{E,nc}$.  Define a  Hodge structure $h_{\Phi}: \bS\rightarrow F^*\otimes \bR$ (here $F^*=\Res_{F/\bQ}\bG_m$) on $F\otimes \bC\simeq \bC^{\Sigma_F}$ by placing $\tau\in \Sigma_F$ in the bi-degree $(1,0)$, if $\tau\in\Phi$; $(0,1)$, if  $\overline{\tau}\in \Phi$; and $(0,0)$ otherwise. 
\begin{proposition}[{\cite[Proposition 2.3.10]{Deligne}}]\label{Deligneconstruction} Let $(G,\cD)$ be of type A,B,C,$\text{D}^{\mathbb{H}}$ or $\text{D}^{\mathbb{R}}$. Let $F$ be any quadratic extension of $E$ with partial CM type $\Phi$ as above. There exists morphisms of Shimura data (as constructed in the following paragraphs)
$$(G,\cD)\leftarrow (G_1,\cD_1)\hookrightarrow (\GSp(V'),\mathfrak{H}^{\pm})$$
    such that $\mathbf{E}(G_1,\cD_1)=\mathbf{E}(G,\cD)\mathbf{E}(F^*,h_\Phi)$ and $G_1^{\ad}=G$.\footnote{To avoid confusion with the totally real field $E$, we use bold font $\mathbf{E}$ for the reflex field.} 
\end{proposition}
A consequence of this proposition is that a simple adjoint Shimura datum $(G,\cD)$ is of abelian type if and only if it is of type A,B,C,$\text{D}^{\mathbb{H}}$ or $\text{D}^{\mathbb{R}}$.

Since it will be used later, we briefly sketch the construction. First, $(\Delta,\mu)$ admits a subset $S$ of $\mu$-symplectic nodes (see \cite[Definition 6.7, Theorem 6.12]{JMC}, or \cite[2.3.7]{Deligne}). Let $\widetilde{G}$ be the simply connected cover of $G$. For each $s\in S$, there is a representation $V(s)$ of $\widetilde{G}_{\bC}$ whose highest weight corresponds to $s$. These can be descend to a representation $V$ of $\widetilde{G}$ over $\bQ$, such that $V_{\bC}=\bigoplus V(s)^n$ for some $n$. Let $\widetilde{G}'$ be the image of $\widetilde{G}$ in $\GL(V)$. 

Lift $h\in \cD$ to a fractional homomorphism (\cite[1.3.4]{Deligne}) $\mathbb{S}\rightarrow \widetilde{G}_\bR$. This gives a fractional Hodge structure over $V$, which can be explicitly described as follows: if $s\in \Delta_\tau$ for a $\tau\in \Sigma_{E,c}$, then $V(s)$ is of type $\{(0,0)\}$; if $s$ lies in a $\tau\in \Sigma_{E,nc}$, then $V(s)$ is of type $\{(r,-r), (r-1,1-r )\}$, where $r $ is as in \cite[1.3.9]{Deligne}. 

Let $F_S$ be the étale $E$-algebra such that $\Hom(F_S,\overline{\bQ})\simeq S$ as $\Gal(\bQ)$-sets. We put a fractional Hodge structure on $F_S$ as follows: the component of $\bC^{[s]}$ in $F_S\otimes_\bQ \bC \simeq \bC^S$ is of type $\{(0,0)\}$, if $s\in \Delta_\tau$ for a $\tau\in \Sigma_{E,c}$; otherwise it is of type $(1-r,r)$, where $r$ is as in the previous paragraph.  Note that $V$ is canonically an $F_S$-module, and $F_S\otimes_{F_S} V$ is the same $\bQ$-space $V$ but carrying a different Hodge structure $h_2$: if $s\in \Delta_\tau$ for a $\tau\in \Sigma_{R,c}$, then $V(s)$ is of type $\{(0,0)\}$; otherwise $V(s)$ is of type $\{(1,0), (0,1)\}$. Let $G_2=\widetilde{G}'\cdot F_S^*\subseteq \GL(V)$. %, and $\cD_2$ is the conjugacy class in $G_2(\bR)$ of the cocharacter $h_2:\bS\rightarrow G_{2,\bR}$. Then there is a morphism $(G_2,\cD_2)\rightarrow (G,\cD)$ and  $E(G_2,\cD_2)=E(G,\cD)$.

Recall that $F$ is equipped with a Hodge structure $h_{\Phi}$ (in the paragraph preceding Proposition~\ref{Deligneconstruction}). We take $V'=F\otimes_E V$, and equip it with the product Hodge structure 
\begin{equation}\label{eq:prdhodge}
    h_3=h_\Phi\otimes h_2.
\end{equation}
Then $V'$ is of type $\{(0,1),(1,0)\}$. Let $(G_3,\cD_3)$ be the Shimura datum defined as follows: $G_3={G}_2\cdot F^*\subseteq \GL(V')$, and $\cD_3$ is the conjugacy class in $G_3(\bR)$ of the cocharacter $h_3:\bS\rightarrow G_{3,\bR}$. We have $(G_3,\cD_3)\rightarrow (G,\cD)$ and $\mathbf{E}(G_3,\cD_3)=\mathbf{E}(G,\cD)\mathbf{E}(F^*,h_{\Phi})$. Finally apply \cite[Corollary 2.3.3]{Deligne} to obtain the desired $(G_1,\cD_1)$. We remind the readers that $G_1\subseteq G_3$ is a subgroup that has the same derived subgroup as $G_1$, and the cocharacters in $\cD_3$ factor through $G_{1,\bR}$ (so $\cD_1=\cD_3$).

%Following \cite[\S 8.5]{JMC}, we denote by $W_K$ the CM fractional Hodge structure $K\otimes_{F} K_S$. 

\subsubsection{Ordinary reduction of Deligne's construction} We say a number field $K\subseteq \bC$ is \textbf{of height 1 at $p$}, if the inclusion $K\subseteq \bC\simeq \bC_p$ factors through $\bQ_p$.
\begin{lemma}\label{lm:lmlm}
   Suppose that $\cD$ contains a CM point $h$ with reflex field $\mathbf{E}_h$ of height 1 at $p$. We can choose $F$ and $\Phi$ such that the following are true: \begin{enumerate}
       \item\label{lmlm:111}  $\mathbf{E}(F^*,h_\Phi)$ is of height 1 at $p$.
       \item\label{lmlm:222} Let $h_3$ be the cocharacter that arises from $h$ as in (\ref{eq:prdhodge}), which furthermore lies in $\cD_1$. Then $h_3$ is a CM point with reflex field of height 1 at $p$.   
   \end{enumerate} 
\end{lemma}
\begin{proof}
 %Let $F'\subseteq \bR$ be the Galois closure of $F$. 
 Consider a quadratic imaginary extension $\bQ(\sqrt{-d})/\bQ$ where $p$ splits completely. Let $F=E(\sqrt{-d})$. Let $\Gal(E_c)\subseteq \Gal(\bQ)$ be the subgroup that fixes $\Sigma_{E,c}$ (i.e., the fixed field of $\Sigma_{E,c}$ is $E_c$).  Let $F_c=E_c(\sqrt{-d})$. Consider the orbits $O_1,...,O_n$ of the group action $\Gal(E_c)\acts \Sigma_{E,c}$. For each $1\leq i\leq n$, pick a representative $\sigma_i\in O_i$ and pick a $\tau_i\in \Sigma_{F}$ lying above $\sigma_i$. Define $$\Phi= \bigcup_{i=1}^n \Gal(F_c)\cdot\tau_i.$$
Then $\Phi$ contains at least one complex embedding lying above each $\sigma \in \Sigma_{E,c}$. On the other hand, if $\Phi$ contains both $\tau$ and $\overline{\tau}$, then there exists $g\in \Gal(F_c)$ such that $g\tau=\overline{\tau}$. This implies that $g(\sqrt{-d})=-\sqrt{-d}$, so $g\notin \Gal(F_c)$. The contradiction shows that $\Phi$ contains exactly one element above each $\sigma \in \Sigma_{E,c}$. 

Note that $F_c$ fixes $\Phi$, so we have $\mathbf{E}(F^*,h_{\Phi})\subseteq F_c$. Since $h$ factors non-centrally through a simple factor of $G_{\bR}$ iff this factor is noncompact, we have $E_c\subseteq \mathbf{E}_h$. Since $\mathbf{E}_h$ is of height 1, so is $E_c$. Since  $p$ splits completely in $\bQ(\sqrt{-d})/\bQ$, $F_c$ is also of height 1. This implies that $\mathbf{E}(F^*,h_{\Phi})$ is of height 1 as well. This proves (\ref{lmlm:111}). 

Now let $h_3$ be as in (\ref{eq:prdhodge}). It is clear that $h_3$ is CM. Its reflex field is of height 1 at $p$ since it is contained in $\mathbf{E}(F^*,h_{\Phi})\mathbf{E}_h$. 
\end{proof}

\subsubsection{Deligne abelian scheme}\label{subsub:furthermodi} Suppose that $X$ is a smooth connected $\Fpbar$-variety with a map $f:X\rightarrow \mathscr{A}_{g,\Fpbar}^{\ord}$, such that $G_\mathrm{B}(f)^{\ad}$ is a simple $\bQ$-group. Via pullback, we get an ordinary abelian scheme $A$ over $X$. In this section, we explain how to associate to (a quasi-finite dominant cover of) $X$ a new ordinary abelian scheme $\widetilde{A}$ using Deligne's construction. This is similar to the process of associating a Kuga--Satake abelian scheme to a K3 surface .  

%Suppose that $G_{\mathrm{B}}(f)^{\ad}$ decomposes into simple adjoint $\bQ$-factors, indexed by a finite set $\bbI$. Run \cite[Construction 4.20]{J23}. We can find a \textit{modification} $f':X'\rightarrow \mathscr{A}_{g',\Fpbar}^{\ord}$ of $f$ that separates factors over $\mathscr{A}_{g_\bbI,\Fpbar}\subseteq \mathscr{A}_{g',\Fpbar}$ \footnote{By the definition of modification (\cite[Definition 4.16]{J23}) and the quasi-finiteness result \cite[Lemma 2.7]{J23}, we can even require that $X'$ is quasi-finite and dominant over $X$.}. By Proposition 4.8 of \textit{loc.cit}, $\MTT_p$ for $f$ is equivalent to $\MTT_p$ of $f'$. So we can assume at the very beginning that $f$ separates factors over $\mathscr{A}_{g_\bbI,\Fpbar}$.  Let $f_i$ be the projection of $f$ to each $\mathscr{A}_{g_i,\Fpbar}$. Then $G_\mathrm{B}(f_i)^{\ad}$ is simple and adjoint over $\bQ$. By working with one $i$ at a time, we can assume at the very beginning that $G_\mathrm{B}(f)^{\ad}$ is already simple. 

Fix an $\mathcal{X}_{f}\in \mathrm{MS}_{f}(\mathcal{A}_{g})$ into $\mathcal{A}_{g'}$.  Let $G=G_\mathrm{B}(f)^{\ad}$. We have a simple adjoint Shimura datum $(G,\cD)$. Since $\cD$ contains a CM point of ordinary reduction, we can choose $F$ and $\Phi$ in Deligne's construction \S\ref{subsub:deligneconstruction} so that the conditions of Lemma~\ref{lm:lmlm} are met. As a result, we obtain explicit morphisms of Shimura data $$(G,\cD)\leftarrow (G_1,\cD_1)\hookrightarrow (\GSp(V'),\mathfrak{H}^{\pm}),$$
where $G_1^{\ad}=G$. Moreover, $\cD_1$ also contains a CM point of ordinary reduction. 
Let $\mathcal{A}_{g'}$ be a Siegel modular variety that arises from $(\GSp(V),\mathfrak{H}^{\pm})$, with sufficiently small level structure. Let $\mathscr{A}_{g'}$ be the naïve integral model of $\mathcal{A}_{g'}$ in the integral canonical model $\mathscr{A}_{g''}$ of a larger Siegel modular variety. Up to a finite cover and an isogeny, the universal abelian scheme on $\mathcal{A}_{g}$ extends to $\mathscr{A}_{g}$. Let's call this abelian scheme ${A}^\mathrm{D}$. On the other hand, there exists a special subvariety $\mathcal{X}_{f}'\subseteq\mathcal{A}_{g'}$ which \textit{specially corresponds} to $ \mathcal{X}_f$ in the sense of \cite[\S 2.4.1]{J23}: this means that there exists a special subvariety $$\mathcal{W}\subseteq\mathcal{X}_{f}\times \mathcal{X}_{f}' \subseteq \mathcal{A}_g\times \mathcal{A}_{g'}$$ whose projections to $\mathcal{X}_f$ and $\mathcal{X}_f'$ are finite and surjective. Let $\mathscr{X}'_{f}$ be the Zariski closure of $\mathcal{X}_{f}$ in $\mathscr{A}_{g'}$. By construction, ${\mathscr{X}'_{f}}$ admits non-empty ordinary locus. Let $\mathscr{X}_f$ (resp. $\mathscr{X}_f'$) be the Zariski closure of $\mathcal{X}_f$ (resp. $\mathcal{X}_f'$) in $\mathscr{A}_g$ (resp. $\mathscr{A}_{g'}$). Let $\mathscr{W}$ be the Zariski closure of $\mathcal{W}$ in $\mathscr{A}_{g}\times \mathscr{A}_{g'}$. Note that $\mathscr{X}_f$ admits non-empty ordinary locus. By Lemma~\ref{lm:lmlm}, $\mathscr{X}_f'$ also admits non-empty ordinary locus. Therefore \cite[Lemma 2.7]{J23} implies that the maps $\varpi:\mathscr{W}^{\ord}_{\Fpbar}\rightarrow \mathscr{X}_{f,\Fpbar}^{\ord}$ and   $\varpi':\mathscr{W}^{\ord}_{\Fpbar}\rightarrow \mathscr{X}_{f,\Fpbar}'^{\ord}$ are both surjective and quasi-finite. Take $X'\subseteq \mathscr{W}^{\ord}_{\Fpbar}$ be a smooth subvariety, such that $\varpi|_{X'}$ is Zariski dense in $X$. Consider $f':=\varpi'|_{X'}: X'\rightarrow\mathscr{X}_{f,\Fpbar}'^{\ord}$. Possibly replacing $X'$ by a further finite cover, we can restrict $A^\mathrm{D}$ to $X'$. It is an ordinary abelian scheme over $X'$, and we will call this abelian scheme a \textbf{Deligne abelian scheme} associated to $A$, denoted by $\widetilde{A}$. Note that $f'$ is a modification of $f$ in the sense of 
\cite[Definition 4.16]{J23}. 

We will denote by $H_\mathrm{B}(\widetilde{A})$ the variation of rational Hodge structures over $\mathcal{X}_f'$ that arise from the universal abelian scheme over $\mathcal{A}_{g'}$. It corresponds to a faithful representation $G_\mathrm{B}(f')\acts V'$. We have 
$G_\mathrm{B}(f')^{\ad}=G_\mathrm{B}(f)^{\mathrm{ad}}$. %Furthermore, by \cite[Proposition 4.8]{J23}, $\MTT_p$ for $f$ is equivalent to $\MTT_p$ for $f'$. 

\iffalse
abelian scheme 
$A^{\mathrm{D}}_i$ over $\mathcal{X}_{f_i}$, such that $G_{\mathrm{B}}(A^{\mathrm{D}}_i)^{\ad}=G_{\mathrm{B}}(f_i)^{\ad}$. In fact, up to a finite cover, $\mathcal{X}_{f_i}$

Applying Proposition~\ref{Deligneconstruction} to get the following 
$$(\MTT(f_i)^{\ad},\mathcal{D}_i^{\ad})\leftarrow (G_{1,i},X_{1,i})\hookrightarrow (\GSp(V_i),\mathfrak{H}_i^\pm).$$

we find a special correspondence of $\mathcal{Y}_i$ with a special subvariety $\mathcal{Z}_i$ of a Siegel modular variety. Since $\mathcal{Y}_i$ contains a CM point of ordinary reduction, by Lemma~\ref{lm:lmlm}, we can require that $\mathcal{Z}_i$ to have a CM point of ordinary reduction. 

Therefore, 

This extra modification is important, which makes Commelin's argument work in our setting ---- we will come back to this in the next section.

$\MTT(f_i)^{\ad}$

We can use Deligne's construction to do further modifications. 

Split $\MTT(\mathcal{X})^{\ad}$ into simples. For each simple factor, we get a special subvariety $\mathcal{Y}_i\subseteq \mathcal{X}$, which realizes $\mathcal{X}$ as an almost product of $\mathcal{Y}_i$'s. 

We can reduce 
\fi

\subsubsection{The proof}  
We begin by making some simplifications. Suppose that we have two maps $f_1:X\rightarrow \mathscr{A}_{g_1,\Fpbar}^{\ord}$ and $f_2:X\rightarrow  \mathscr{A}_{g_2,\Fpbar}^{\ord}$, corresponding to two abelian schemes $A_1,A_2$ over $X$. Suppose that $\MTT_p$ holds for both $f_1$ and $f_2$. Our goal is to show that $\MTT_p$ holds for $f=f_1\times f_2$.

First, we can reduce to the case where $f_1$ and $f_2$ separate factors. Let $\alpha\in\{1,2\}$. Suppose that the $\bQ$-simple factors of $G_{\mathrm{B}}(f_\alpha)^{\ad}$ are indexed by $\bbI_\alpha$. Run \cite[Construction 4.20]{J23}. We can find a \textit{modification} $f'_\alpha:X'_\alpha\rightarrow \mathscr{A}_{g'_\alpha,\Fpbar}^{\ord}$ of $f$ that separates factors over $\mathscr{A}_{g_{\bbI_\alpha},\Fpbar}\subseteq \mathscr{A}_{g'_\alpha,\Fpbar}$. By \cite[Definition 4.16]{J23} and the quasi-finiteness result \cite[Lemma 2.7]{J23}, we can find a dominant quasi-finite cover $X''_\alpha\rightarrow X$ which moreover dominates $X'_\alpha$. There exists a smooth variety $X''$ that is quasi-finite over and dominates both $X''_1$ and $X''_2$. Let $f''_\alpha$ be the induced map $X''\rightarrow \mathscr{A}_{g'_\alpha,\Fpbar}$. Note that $f''_\alpha$ is a modification of $f_\alpha$, and $f'':=f''_1\times f_2''$ is a modification of $f=f_1\times f_2$. So by \cite[Proposition 4.8]{J23}, $\MTT_p$ for $f$ (resp. $f_1,f_2$) is equivalent to $\MTT_p$ for $f''$ (resp, $f''_1,f''_2$). Replacing $(X,f)$ by $(X'',f'')$, we can assume at the very beginning that $f_1$ and $f_2$ separate factors. 

Second, we can reduce to the case where 
 $G_{\mathrm{B}}(f_1)^{\ad}$ and $G_{\mathrm{B}}(f_2)^{\ad}$ are simple. Suppose that $f_1,f_2$ already separate factors, so for $\alpha\in \{1,2\}$, we have maps $f_\alpha:X\rightarrow \mathscr{A}_{g_{\bbI_\alpha},\Fpbar}^{\ord}$. Let $i_1\in \bbI_1$ and $i_2\in \bbI_2$. We obtain two maps $f_{i_1}: X\rightarrow \mathscr{A}_{g_{i_1},\Fpbar}^{\ord}$ and $f_{i_2}: X\rightarrow \mathscr{A}_{g_{i_2},\Fpbar}^{\ord}$. Then $G_\mathrm{B}(f_{i_1})^{\ad}$ and $G_\mathrm{B}(f_{i_2})^{\ad}$ are simple, and $\MTT_p$ holds for both $f_{i_1}$ and $f_{i_2}$. Suppose we can show that $\MTT_p$ holds for $f_{i_1}\times f_{i_2}$ for every pair $(i_1,i_2)\in \bbI_1\times \bbI_2$, then $\MTT_p$ holds for $f_1\times f_2$ by \cite[Lemma 10.1]{JMC}.

\begin{proof}[Proof of Theorem~\ref{modpCommelin}]
By the discussion made above, we will assume that  $G_{\mathrm{B}}(f_1)^{\ad}$ and $G_{\mathrm{B}}(f_2)^{\ad}$ are simple. Let $X_0$ be a sufficiently large finite field model of $X$. Possibly replacing $X_0$ be an étale cover, we may assume that the $u$-adic group $G_u(f)$ is connected for all $u\in \mathrm{\bQ}$. Let $l\neq p$ be a prime. Let $M=\{H_\bullet(M)\}$ (\textit{resp}. $M_1=\{H_\bullet(M_1)\}$ \textit{resp}. $M_2=\{H_\bullet(M_2)\}$) be the collections of coefficient objects associated to the pullback abelian schemes corresponding to $f$ (\textit{resp}. $f_1$ \textit{resp}. $f_2$), respectively. It suffices to show that if $G_l(M)\subsetneq G_l(M_1)\times G_l(M_2)$, then $G_\mathrm{B}(M)\subsetneq G_\mathrm{B}(M_1)\times G_\mathrm{B}(M_2)$.  

The proof is a characteristic $p$ variant of the proof in \cite[\S 9]{JMC}. Whenever Commelin uses \cite[Theorem 5.13]{JMC} in his original proof, we use Lemma~\ref{lm:Fcompatible} instead; whenever Commelin uses the group ``$\Gamma_K$'', we use $\pi_1^{\et}(X_0)$ instead. The proof is decomposed into several steps, which roughly follows the order of \cite[\S 9]{JMC}.
\begin{enumerate}
    \item\label{it:Commelinproof1} (\textbf{Showing that $H_\Lambda(M_1)\simeq H_\Lambda(M_2)$})\\   
    Following \cite[\S 9.3]{JMC}, we can identify $E_1=\End(H_\mathrm{B}(M_1))$ with $E_2=\End(H_\mathrm{B}(M_2))$, and there is a place $\lambda$ in $E:=E_1=E_2$ lying above $l$ such that $H_\lambda(M_1)\simeq H_{\lambda}(M_2)$ as $E_\lambda$-coefficient objects. Using Lemma~\ref{lm:Fcompatible}, we see that $H_\Lambda(M_1)\simeq H_\Lambda(M_2)$ as collection of coefficient objects. 
\item (\textbf{Showing that $(\Delta_1,\mu_1)\simeq (\Delta_2,\mu_2)$})\\
To ease notation, let $G_i:=G_\mathrm{B}(M_i)$. Let $(G_i,\cD_i)$ be the adjoint Shimura datum arising from $G_i$, and let  $(\Delta_i,\mu_i)$ be its Deligne--Dynkin diagram. In the following we run \cite[\S 9.4]{JMC}: Write $f: \pi_0(\Delta_1) \rightarrow \pi_0(\Delta_2)$ be the canonical identifications $\pi_0(\Delta_i)\simeq \Hom(E,\overline{\bQ})$ as $\Gal(\bQ)$-sets. From (\ref{it:Commelinproof1}), we have local isomorphisms \begin{equation}\label{eq:isompsi}
\psi_l\in \mathrm{Isom}_f((\Delta_{1},\mu_1)_{\bQ_l},(\Delta_{2},\mu_2)_{\bQ_l})^{\Gal \bQ_l}
\end{equation} for $l\in \mathrm{fpl}(\bQ)$ \footnote{We haven't defined $(\Delta_{1},\mu_1)_{\bQ_l}$ and $\mathrm{Isom}_f$. However, their meanings are pretty self-explanatory.  The readers can see \cite[\S 6.16, \S 6.17]{JMC} for the precise definitions.}. The only nontrivial part of establishing (\ref{eq:isompsi}) is to show that $\psi_l(\mu_1)=\mu_2$ for $l\in \mathrm{fpl}(\bQ)$. For this, one need to apply the Hodge--Tate decomposition (see \cite[\S7]{JMC}). More precisely, one need to show that after identifying $H_l(M_i)$ with $H_\mathrm{B}(M_i)\otimes {\bQ_l}$, an isomorphism $H_l(M_1)\simeq H_l(M_2)$ from (\ref{it:Commelinproof1}) carries the conjugacy class of $\mu_1$
to that of $\mu_2$. In our setting, this is again true: take $x\in X_0(\Fpbar_q)$ and take canonical lift $\Tilde{x}_{\bC}\in \mathcal{A}_{g_1,\bC}\times \mathcal{A}_{g_2,\bC}$. We can take $\mu$ to be the Hodge cocharacter arising from the CM abelian varieties over $\Tilde{x}_{\bC}$, and take $\mu_1$,$\mu_2$ to be its components. We then apply Hodge--Tate decomposition to this CM abelian variety. %When $l=p$, we use de Rham--crystalline comparison instead.

Once we have (\ref{eq:isompsi}), we can apply \cite[Proposition 6.23]{JMC} to obtain a $\Gal(\bQ)$-invariant isomorphism $\phi:(\Delta_1,\mu_1)\simeq (\Delta_2,\mu_2)$, such that $\phi_\lambda=\psi_\lambda$ for some $\lambda\in \mathrm{fpl}(E)$ (here the $\lambda$ may be different from the $\lambda$ from (\ref{it:Commelinproof1}), but we abuse the notation). We can further require that $\lambda\nmid p$: In fact, in the original proof of \cite[Proposition 6.23]{JMC}, one finds such $\lambda$ by Chebotarev's density theorem; however, Chebotarev's density theorem guarantees the existence of infinitely many such $\lambda$. So we can just pick a $\lambda$ which does not divide $p$). Let $l\in \mathrm{fpl}(\bQ)$ be the prime that $\lambda$ lies over. We will fix $\lambda$ and $l$ in the rest of the proof. 
\item (\textbf{Applying Deligne's construction}) \\
Following \cite[\S 9.5]{JMC}, we apply Deligne's construction to the adjoint Shimura data $(G_i,\cD_i)$. We make the following choices: 
\begin{enumerate}
    \item  since each $\cD_i$ contains a CM point of ordinary reduction, we can choose a totally imaginary quadratic extension $F/E$ and a partial CM type $\Phi\subseteq \Sigma_F$, such that the conditions in Lemma~\ref{lm:lmlm} are met.
    \item recall that $\widetilde{G_i}$ is the simply connected cover of $G_i=G_\mathrm{B}(M_i)$, and we have constructed a representation $\widetilde{G_i}\acts V_i$ from the set of $\mu$-symplectic nodes of $(\Delta_i,\mu_i)$. We have $\dim V_1=\dim V_2$. 
\end{enumerate}
We then get explicit morphisms of Shimura data $$(G_i,\cD_i)\leftarrow (G_{i1},\cD_{i1})\hookrightarrow (\GSp(V'_i),\mathfrak{H}^{\pm}_i).$$
As explained in \S\ref{subsub:furthermodi}, possibly replacing $X_0$ by a quasi-finite dominant cover, we obtain two ordinary Deligne abelian schemes $\widetilde{A}_1$ and $\widetilde{A}_2$ over $X_0$, such that $G_\mathrm{B}(\widetilde{A}_i) \acts V'_i$ and $G_\mathrm{B}(\widetilde{A}_i)^{\ad}= G_{\mathrm{B}}(M_i)$. 

Write $W_{F,i}:=F\otimes_E F_{S_i}$, which is a weight 1 fractional Hodge structure of CM type. Then  $V'_i= W_{F,i}\otimes_{F_{S_i}}V_i$. Let $H_{\mathrm{B}}(\widetilde{A}_i)$ be the variation of Hodge structures as explained at the end of \S\ref{subsub:furthermodi}. Let $H_l(\widetilde{A}_i)$ be the first $l$-adic cohomology of $\widetilde{A}_i$. Note that $E\subseteq \End(H_{\mathrm{B}}(\widetilde{A}_i))$ (\cite[\S 8.4,\S8.5]{JMC}). Let $H_\lambda(\widetilde{A}_i)=H_l(\widetilde{A}_i)\otimes_{E_l} E_\lambda$, which splits 
$H_l(\widetilde{A}_i)$ as $\bigoplus_{\lambda\in \Lambda_l} H_\lambda(\widetilde{A}_i)$. Let $G_\lambda(\widetilde{A}_i)$ (resp. $G_\lambda(\widetilde{A}_i)$) be the monodromy group of $H_\lambda(\widetilde{A}_i)$ (resp. $H_\lambda(\widetilde{A}_i)$).

\item (\textbf{Making identifications})\\
Following  \cite[\S 9.6]{JMC}, we can make the following canonical identifications:  
\begin{itemize}
    \item Identify $H_\lambda(M_1)$ and $ H_\lambda(M_2)$, as well as the monodromy representations $\rho_{i,\lambda}: \pi_1^{\et}(X_0)\rightarrow G_\lambda(M_i)(E_\lambda)$. Identify $G_\lambda(M_1)$ with $G_\lambda(M_2)$ and their simply connected covers $\widetilde{G_\lambda(M_1)}$ with $\widetilde{G_\lambda(M_2)}$. The resulting objects will be called $H_\lambda,\rho_\lambda,  G_\lambda$ and $\widetilde{G_\lambda}$, respectively. 
    \item Identify $(\Delta_1,\mu_1)$ with $(\Delta_2,\mu_2)$, and identify $E$-algebras $W_{F,1}$ and $W_{F,2}$. Call the resulting algebra $W_F$ and write $W_{F,\lambda}$ for $W_F \otimes_E E_\lambda$.
    \item     
    Identify $H_\lambda(\widetilde{A}_1)$ with  $H_\lambda(\widetilde{A}_2)$ as representations of $\widetilde{G_\lambda}$ with an action of $W_{F,\lambda}$. Call this representation $H'_\lambda$. Identify $G_\lambda(\widetilde{A}_1)$ with $G_\lambda(\widetilde{A}_2)$. Call this group $G_\lambda'$.
\end{itemize}
Finally, for each $i\in \{1,2\}$, there is a representation $$\widetilde{\rho}_{i,\lambda}: \pi_1^{\et}(X_0)\rightarrow G_\lambda'(E_\lambda)$$ that lifts $\rho_{\lambda}$.
\item (\textbf{Showing that ${\widetilde{\rho}_{1,\lambda}=\widetilde{\rho}_{2,\lambda}}$})\\ For this, we adapt the argument in \cite[\S 9.7,\S 9.8]{JMC}: for $i=1,2$ the natural map $p_i:G_\mathrm{B}(W_F) \rightarrow G_\mathrm{B}(\widetilde{A}_i)^{\mathrm{ab}}$ is an isogeny, and there exists a torus $T$ together with isogenies $q_i: G_\mathrm{B}(\widetilde{A}_i)^{\mathrm{ab}}\rightarrow T$ such that $q_1\comp p_1 = q_2\comp p_2$. Let $T\rightarrow \GL(N)$ be a faithful representation of $T$. By the same reason as in \cite[\S 9.8]{JMC}, $N$ can be viewed as a complex CM motive. Define two Galois representations in the same vein as \textit{loc.cit}:
\begin{equation}\label{eq:imgaN}
    \alpha_i:\pi_1^{\et}(X_0)\xrightarrow{\widetilde{\rho}_{i,l}} G_l(\widetilde{A}_i)(\bQ_l)=G_\mathrm{B}(\widetilde{A}_i)(\bQ_l)\rightarrow G_\mathrm{B}(\widetilde{A}_i)^{\mathrm{ab}}(\bQ_l)\xrightarrow{q_{i,l}} T(\bQ_l) \rightarrow  \GL(N)(\bQ_l).
\end{equation}
We claim $\alpha_1=\alpha_2$: let $x_0\in |X_0|$ be any closed point, then the $l$-adic monodromy $G_l(\widetilde{A}_{i,x_0})(\bQ_l)$ surjects onto $G_\mathrm{B}(\widetilde{A}_i)^{\mathrm{ab}}(\bQ_l)$. Let $\alpha_i'$ be the restriction of $\alpha_i$ to $\pi_1^{\et}(x_0)$. It suffices to show that $\alpha_1'=\alpha_2'$. Take canonical lifts of both  $\widetilde{A}_{1,x_0}$ and $\widetilde{A}_{1,x_0}$ and descend to a sufficiently large number field $K$ where the CM motive $N$ is defined and the $l$-adic monodromy of the canonical lifts are connected. Then the Galois reps 
$$\Gal(\overline{K}/K)\rightarrow G_l(\widetilde{A}_{i,x_0})(\bQ_l)\rightarrow G_\mathrm{B}(\widetilde{A}_i)^{\mathrm{ab}}(\bQ_l)\xrightarrow{q_{i,l}} T(\bQ_l) \rightarrow  \GL(N)(\bQ_l)$$
are both identified with the $l$-adic $\Gal(\overline{K}/K)$-rep associated to the motive $N$, hence independent of $i$. From this, the fact that $\alpha_1'=\alpha_2'$ then follows easily.  
 
 %We have two $\Gal(\overline{K}/K)$-reps: $$\Gal(\overline{K}/K)\rightarrow G_l(\widetilde{A}_{i,x_0})(\bQ_l)\subseteq G_l(\widetilde{A}_i)(\bQ_l)=G_\mathrm{B}(\widetilde{A}_i)(\bQ_l)\rightarrow G_\mathrm{B}(\widetilde{A}_i)^{\mathrm{ab}}(\bQ_l)\xrightarrow{q_{i,l}} T(\bQ_l) \rightarrow  \GL(N)(\bQ_l)$$

 Let $G_l(N)$ be the monodromy group of (\ref{eq:imgaN}), which is nothing but $T_{\bQ_l}$. Following the last paragraph of \cite[\S 9.8]{JMC}, one can use $G_l(N)$ to construct a finite quotient $(\Res_{E_\lambda/\bQ_l}G_{\lambda}')^{\mathrm{ab}}\twoheadrightarrow T'$ such that $$  \pi_1^{\et}(X_0)\xrightarrow{\widetilde{\rho}_{i,l}} (\Res_{E_\lambda/\bQ_l}G_{\lambda}')(\bQ_l)\rightarrow (\Res_{E_\lambda/\bQ_l}G_{\lambda}')^{\mathrm{ab}}(\bQ_l)\rightarrow T'(\bQ_l)$$
is independent of $i=1,2$. The existence of such $T'$ implies that $\widetilde{\rho}_{1,\lambda}=\widetilde{\rho}_{2,\lambda}$, the proof is purely group theoretic, and works in our setting, see \cite[\S 9.9]{JMC} for more details. 
\item (\textbf{Applying the Tate conjecture}) \\Finally, the fact that $\widetilde{\rho}_{1,\lambda}=\widetilde{\rho}_{2,\lambda}$ and the compatibility of $\{H_\lambda(\widetilde{A}_1\oplus \widetilde{A}_2)\}_{\lambda\in E}$ (proved exactly as Lemma~\ref{lm:Fcompatible}) imply that $H_l(\widetilde{A}_1)\simeq H_l(\widetilde{A}_2)$ as Galois representations, and we conclude by Tate isogeny theorem that $G_\mathrm{B}(M)\subsetneq G_\mathrm{B}(M_1)\times G_\mathrm{B}(M_2)$. 
\end{enumerate}
   \end{proof}
\begin{proof}[Proof of Corollary~\ref{cor:Commelinconsequence}]
    By Theorem~\ref{thm:logALimpliesMT}, $\MTT_p$ holds for $f_i$. So $\MTT_p$ holds for $f$ by Theorem~\ref{modpCommelin}. By Theorem~\ref{thm:logALimpliesMT} again, $\mathrm{logAL}_p$ holds for $f$. If $f$ is an locally closed immersion, the we conclude by Theorem~\ref{thm:MTimpliesAO} that $\mathrm{geoAO}_p$ holds for $X$. 
\end{proof}

\subsection{Abelian varieties of Pink type}\label{sub:pinksresult} We refer to Definition~\ref{def:pinktype} for the definition of abelian varieties of Pink type. In this section, we will establish 
Theorem~\ref{modpPink} by modifying Pink's method. Again, we will often work with the geometric version: 
$f: X\rightarrow \mathscr{A}^{\ord}_{g,\Fpbar}$, and let $(X_0,f_0)$ be a suitable finite field model, with an abelian scheme $A_{X_0}$ over $X_0$. Let $\eta_0=\Spec\, k$ be the generic fiber of $X_0$, and write $A:=A_{\eta_0}$.
\begin{ass}\label{ass:pink}~\begin{enumerate}
      \item  By Lefschetz, we can assume that $k$ is a global function field, i.e., $X$ is a curve. 
    \item  Possibly replacing $X$ by an étale cover, we can assume $G_u(f)$ is connected for all $u\in \mathrm{fpl}(\bQ)$. Note that $\rho_u$ is absolutely irreducible. 
\end{enumerate}
\end{ass}

\subsubsection{Formal characters}\label{subsub:formalchar} Consider a faithful representation $(G, \rho)$ over a characteristic $0$ field $F = \overline{F}$ with $\rho$ irreducible. Let $T\subseteq G$
be a maximal torus. The formal character is defined as $$\ch_\rho := \sum_{\chi\in X^*(T)}
\mathrm{mult}_{\rho|_T}(\chi) \cdot \chi \in  \bZ[X^*
(T)].$$ Pink and Larsen investigate in \cite{PLformalchar} the following

\begin{question}
    How uniquely is the pair $(G, \rho)$ determined by $\ch_\rho$ ?
\end{question}

This boils down to the question of how uniquely the root system $\Phi$ of $G$ is determined. Let $\Phi^\circ$ be the subset of roots which are short in their respective simple factors of $\Phi$. It is a root system of the same rank as $\Phi$. It can be shown that when $\Phi$ is simple, all simple factors of $\Phi^\circ$ are of the same type, i.e., $\Phi^\circ$ is \textbf{isotypic}. In \cite{PLformalchar}, Pink and Larsen have shown that $\Phi^\circ$ is uniquely determined by $\ch_\rho$, and $(G,\rho)$ is uniquely determined by $\ch_\rho$ with only a few exceptions.

This general story can be incorporated with the study of Mumford--Tate pairs. Suppose that $(G, \rho)$ is an irreducible weak Mumford--Tate pair (that comes from an abelian variety), $T$ is a conjugate of a Frobenius torus. By restricting to certain tensor factors of $(G,\rho)$, one can assume that $(\Phi^\circ,\ch_\rho)$ is isotypic. After comparing with \cite[Table 4.2]{P98}, Pink classifies the possible triple $(G,\rho,\Phi^\circ)$ (with $(\Phi^\circ,\ch_\rho)$ isotypic) into the following three cases (cf. \cite[\S4]{P98}): 
\begin{itemize}
    \item(The non-ambiguous case) $\Phi=\Phi^\circ$, and all simple factors of $(G,\rho)$ are of one of the types 1, 4 and 6 in \cite[Table 4.2]{P98}. 
    
    \item(The first ambiguous case)  The simple factors of $\Phi^\circ$ and $\Phi$ correspond bijectively to each other. There is an $r\geq 3$, such that each simple factor of $\Phi^\circ$ is of type $D_r$, and 
    each simple factor of $(G,\rho)$ is one of the types 2 and 3 in \cite[Table 4.2]{P98} with this $r$. 

    \item(The second ambiguous case)
   Each simple factor of $\Phi^\circ$ is of form $A_1$. Each simple factor of $(G,\rho)$ is of type 5 in \cite[Table 4.2]{P98} with possibly varying $r$ (the sum of these $r$'s is determined by $\ch_\rho$, though).
\end{itemize}

\subsubsection{Interpolation in characteristic $p$} In \cite[\S6]{P98}, Pink's shows the existence of a good $\bQ$-model for the $l$-adic monodromy of a simple abelian variety over a number field. He calls the method ``interpolation''. The method rests upon a series of deep works by Pink and Larsen \cite{PLformalchar,LP,LP95}. These results can be adapted to the characteristic $p$ function field setting. 
The main goal of this section is to establish a characteristic $p$ analogue of \cite[Theorem 5.3]{P98} (see Theorem ~\ref{thm:Pinkinterpolation}). The original proof of \cite[Theorem 5.3]{P98} relies on a set of finite places $V_{\mathrm{good}}$ that is defined in \cite[(6.9)]{P98}.  In our setting, we will simply define $$V_{\mathrm{good}}:=V_{\max}(X_0).$$
See (\ref{eq:V_max}). By Proposition~\ref{vmax_positivedensity}, $V_{\mathrm{good}}$ contains a subset of positive density.  

When the map $f$ is clear from the context, we will just write $G_u$ for $G_u(f)$. The following is a list of characteristic $p$ constructions that are analogues to the ones in \cite[\S 6]{P98}:
\begin{enumerate}
    \item For each $z\in V_{\mathrm{good}}$, let $T_z:=G_{\mathrm{B}}(z)$ be the Frobenius torus; cf. \S\ref{sub:Frobtorus}. We will fix once and for all a split torus $T_0\subseteq \GL_{2g,\bQ}$ that conjugates to $T_z$ over $\overline{\bQ}$ for all $z\in V_{\mathrm{good}}$.
    \item For each $z\in V_{\mathrm{good}}$, fix an $h_z\in \GL_{2g}(\overline{\bQ})$, such that $T_{0,\overline{\bQ}}= h_z^{-1}T_{z,\overline{\bQ}}h_z$.
    \item For each $u\in \mathrm{fpl}(\bQ)$, fix an $f_u\in \GL_{2g}(\overline{\bQ}_u)$, such that $T_{0,\overline{\bQ}_u}\subseteq f_u^{-1}G_{u,\overline{\bQ}_u}f_u$, and that $f_uT_{0,\overline{\bQ}_u}f_u^{-1}$ is defined over $\bQ_u$. 
    \item Let $\mu_z: \bG_{m,\overline{\bQ}}\rightarrow T_{z,\overline{\bQ}}$ be the Hodge cocharacter of the canonical lift of $z$. Let $\nu_z:= \mu_z^{-1}$ be the Newton cocharacter. Let $\nu_{0,z}:\bG_{m,\overline{\bQ}}\rightarrow T_{0,\overline{\bQ}}$ be $h_z^{-1}\nu_zh_z$.
    \item  Let $\Gamma \subseteq \Aut(T_{0,\overline{\bQ}})$ be the stabilizer of the formal character
of the tautological representation. The form of $T_z$ over $\bQ$ corresponds to a homomorphism
\begin{equation}\label{eq:varphiz}
    \varphi_z : \Gal(\bQ)\rightarrow\Gamma.
\end{equation} 
    \item Let $\Phi_u \subseteq X_*(T_0)$ be the root system of $f_u^{-1}G_{u,\overline{\bQ}_u}f_u$. Let $\Phi_u^\circ$ be the subset of roots which are short in their respective simple factors of $\Phi_u$. Since $\rho_u$ is irreducible of dimension $2g$, $\Phi_u$ and $\Phi^\circ_u$ are non-empty. As we have seen in \S\ref{subsub:formalchar}, $\Phi_u^\circ$ is uniquely determined by the formal cocharacter, and $\rk \Phi_u= \rk \Phi_u^\circ$. 
\item Let $W_u\subseteq \Gamma$ be the Weyl group of $\Phi_u$. Since $(G_{u,\overline{\bQ}_u},\rho_u)$ is independent of $u$ (after identifying the coefficient fields); cf. \cite[Lemma 3.4]{J23}, $\Phi_u$ and $\Phi_u^\circ$ are independent of $u$ chosen, so we can use $\Phi$ and $\Phi^\circ$ to denote them.
\item The form of $f_uT_{0,\overline{\bQ}_u}f_u^{-1}$ over $\bQ_u$ corresponds to a homomorphism $\Gal(\bQ_u)\rightarrow \mathrm{Norm}_{\Gamma}(W_u)$. Composing it with the projection $\pi_u:\mathrm{Norm}_{\Gamma}(W_u)\rightarrow \mathrm{Norm}_{\Gamma}(W_u)/W_u$ gives a map $$\overline{\varphi}_u:  \Gal(\bQ_u)\rightarrow \mathrm{Norm}_{\Gamma}(W_u)/W_u.$$
This map characterizes the form of $G_u$ up to inner twist. 
\end{enumerate}

\begin{proposition}[Analogue of {\cite[Proposition 6.10]{P98}}]\label{prop:Vgood}
For any $z \in V_{\mathrm{good}}$, there exists an $\alpha \in  \Phi^\circ$ such that
$\la \alpha,\nu_{0,z}\ra = -1$.
\end{proposition} 
\begin{proof} Note that $\nu_{0,z}$ has weights $\{0,1\}$ and both 0 and 1 are attained. This means that  $\la \chi,\nu_{0,z}\ra \in \{0,1\}$, for every $\chi \in X^*
(T_0)$ that occurs in the tautological representation, and both 0 and 1 can appear. Now nay root $\alpha\in \Phi^\circ$ is a quotient of two such weights. So we have $\la \alpha,\nu_{0,z}\ra \in \{-1,0,1\}$ for $\alpha\in \Phi^\circ$. Since $\rk\Phi^\circ=\rk\Phi$, we have $\Span_{\bR}\{\Phi\}=\Span_{\bR}\{\Phi^\circ\}$. If  $\la\alpha, \nu_{0,z}\ra=0$ for all $\alpha\in \Phi^\circ$, we must have $\la\alpha, \nu_{0,z}\ra=0$ for all $\alpha\in \Phi$. It follows that $\mu_{z,\overline{\bQ}_p}:\bG_{m,\overline{\bQ}_p}\rightarrow G_{p,\overline{\bQ}_p}$ is central. In particular, $X$ maps to a point in $\mathscr{A}_{g,\Fpbar}$. In other words, the abelian variety $A$ is CM, contradictory to our assumption~\ref{ass:pink}. As a result, $\la\alpha, \nu_{0,z}\ra$ can not be 0 for all $\alpha\in \Phi^\circ$. Therefore there exists an $\alpha \in  \Phi^\circ$ such that $\la \alpha,\nu_{0,v}\ra = -1$.
\end{proof}

\begin{proposition}[Analogue of {\cite[Proposition 6.12]{P98}}]\label{prop:pinktranstive}
    For any $z \in V_{\mathrm{good}}$, the action of $\Gal(\bQ)$ via $\varphi_z$ permutes the
simple factors of $\Phi$ transitively.
\end{proposition}
\begin{proof}
    The proof of \cite[Proposition 6.12]{P98} works with minor modification: the place where Pink uses \cite[Proposition 6.10]{P98} should be replaced by Proposition~\ref{prop:Vgood}. Note that the results in \cite[\S 4]{PLformalchar} are  used in the original proof. These results are purely group theoretic and work in our setting. 
\end{proof}
\begin{corollary}[Analogue of {\cite[Corollary 6.13]{P98}}]\label{cor:P613}
The formal character of $G^{\der}_{u,\overline{\bQ}_u}\subseteq \GL_{2g,\overline{\bQ}_u}$ is $\otimes$-isotypic.  
\end{corollary}
\begin{proof}
    Easy from the proposition. 
\end{proof}
\begin{theorem}[Analogue of {\cite[Theorem 5.13]{P98}}]\label{thm:Pinkinterpolation}
Assumption as in \ref{ass:pink}:
\begin{enumerate}
    \item There exists a connected reductive subgroup $G \subseteq \GL_{2g,\mathbb{Q}}$ such that $G_{u}$ is conjugate to $G_{\bQ_u}$ under $\GL_{2g,\bQ_u}$ for every $u$ in some set $\bL$ of primes of Dirichlet density 1.
    \item The pair consisting of $G$ together with its absolutely irreducible tautological representation is a strong Mumford-Tate pair of weights $\{0, 1\}$ over $\bQ$. 
    \item The derived group $G^{\der}$ is $\bQ$-simple.
\end{enumerate}\end{theorem}
\begin{proof}
By Corollary~\ref{cor:P613} we may now distinguish cases according to the type of the formal character of each simple factor of $\Phi^\circ$; cf. \S\ref{subsub:formalchar}.
\begin{enumerate}
        \item (\textbf{The non-ambiguous case})\\
       The same argument in \cite[p39-40]{P98} works with minor modification. In the original proof, the key inputs are \cite[Proposition 6.12, 6.15]{P98}. For our case, we need to replace \cite[Proposition 6.12]{P98} by its char $p$ analogue, Proposition~\ref{prop:pinktranstive}.
       On the other hand, the obvious analogue of \cite[Proposition 6.15]{P98}
       still holds. Let's briefly explain the reason: First, \cite[Proposition 6.15(a)(c)]{P98} are easy. Second, \cite[Proposition 6.15(b)]{P98} follows from \cite[Theorem 3.2]{LP95}, which is proved for abelian varieties over any global field.
    \item (\textbf{The ambiguous cases})\\
The same argument in \cite[p40-44]{P98} works with minor modification. The key input in the original proof is \cite[Proposition 6.16]{P98}, which follows from \cite[Propositions 8.9]{LP}. The later proposition is proven for all global fields, so works in our setting. As a result, the obvious analogue of \cite[Proposition 6.16]{P98} holds in our setting. This, combined with Proposition~\ref{prop:pinktranstive}, proves the obvious analogues of \cite[Propositions 6.18, 6.19]{P98}. The proof of \cite[Propositions 6.20, 6.23]{P98} is based on \cite[Propositions 6.18, 6.19]{P98} and Lie theory, and is still valid in our case. So Pink's original proof goes through.
    \end{enumerate}
\end{proof}

\begin{proof}[Proof of Theorem~\ref{modpPink}] First use Theorem~\ref{modpCommelin} to reduce to the case of Assumption~\ref{ass:pink}. If $G_\mathrm{B}(A)$ has type $A_{2s-1}$ for $s\geq 1$ or $B_r$ for $r\geq 1$, this is Theorem~\ref{thm:Pinkcase2}. If $2g$ is not an $r$-th power of an integer for $r>1$, nor of the form $\binom{2r}{r}$ for some odd $r>1$, we can apply \cite[Proposition 4.7]{P98} to the strong Mumford--Tate pair $(G_{\mathrm{B}}(A),\rho_{\mathrm{B}})$ (cf. Proposition~\ref{prop:MTpairs}) to obtain $G_{\mathrm{B}}(A)=\GSp_{2g,\bQ}$. Let $G$ be in Theorem~\ref{thm:Pinkinterpolation}. Apply \cite[Proposition 4.7]{P98} to $G$ and its tautological representation. We find that $G=\GSp_{2g,\bQ}$ as well, so $G_u= \GSp_{2g,\bQ_u}$  for some $u\in \bL$, hence for every $u\in \mathrm{fpl}(\bQ)$. Therefore $\MTT_p$ holds for $A$.
\end{proof}
\subsection{Prime and low dimensions} Let $A$ be an abelian varieties over $k/\Fpbar_q$ of dimension $g$. The method in this section involves classifications of the endomorphism algebra  $\End^0(A_{\overline{k}})$. Let ${\epsilon}$ be the generic point of a minimal special subvariety $\cX\subseteq\mathcal{A}_{g,\overline{\bQ}}$ whose mod $p$ reduction contains the moduli point corresponding to $A_{\overline{k}}$. By Proposition~\ref{thm: Endequal}, \begin{equation}\label{eq:equalEnd}
\End^0(A_{\overline{k}})=\End^0(A_{\overline{\epsilon}}).
\end{equation}

\begin{proof}[Proof of Theorem~\ref{modpTank}]
By (\ref{eq:equalEnd}), $A_{\overline{\epsilon}}$ is simple, and we can run the classification from \cite[\S 2.12$\sim$ \S2.16]{Mointro}: when $g$ is an odd prime, the possible structure of $\End^0(A_{\overline{k}})$ are (1) $\bQ$, or (2) a totally real field $F$ of degree $g$ over $\bQ$, or (3) an imaginary quadratic field $E$,
or (4) a CM field of degree $2g$ over $\bQ$. The structure of $G_{\mathrm{B}}(A)^{\der}$ in each cases can also be read off from \textit{loc.cit}.  When $g=2$, the endomorphism algebra can also be (5) an indefinite quaternion over $\bQ$. 

In the first case, $A$ is of Pink type, so $\MTT_p$ holds by Theorem~\ref{modpPink}. In the fourth case, $A$ is CM, so $\MTT_p$ holds. 
In the second case, $G_{\mathrm{B}}(A)$ is a scalar restriction of a form of $\GL_{2,F}$ to $\bQ$. One deduces that $G_p(A)^{\circ\der}_{\bC_p}$ surjects onto each simple factor of $G_{\mathrm{B}}(A)_{\bC_p}^{\der}$, and it is easy to see that $\MTT_p$ must also hold in this case. In the third case, $G_{\mathrm{B}}(A)_{\bC}^{\der}=\SL_{g,\bC}$ and the representation $(G_{\mathrm{B}}(A)_{\bC},\rho_{\mathrm{B}})$ splits into two mutually dual $g$ dimensional irreps $V\oplus V^{\vee}$. Say $V$ corresponds to the standard representation of $\SL_{g,\bC}$. So we obtain a simple Mumford--Tate pair of type $A_{g-1}$ in the first column of \cite[Table 4.2]{P98}. Since $g$ is a prime, this simple Mumford--Tate pair can not admit a smaller irreducible Mumford--Tate pair with the same dimension of representation. So $\MTT_p$ holds in this case as well. 

Finally, in the fifth case $G_{\mathrm{B}}(A)$ is generated by the homothety torus together with $D^{\mathrm{nm}=1}$ where $D$ is the quaternion endomorphism algebra in question. The later is a form of $\SL_2$. This group is too small to hold a smaller $G_l(A)^\circ$. So $\MTT_p$ holds. 
\end{proof}
\begin{lemma}\label{prop:mtreduction1}
Suppose that $A$ is an abelian fourfold and  $\End^0(A_{\overline{k}})=\bQ$.\begin{enumerate}
    \item If $A$ is not of $p$-adic Mumford type, then $\mathrm{MT}_p$ holds for $A$. 
\item 
If $A$ is of $p$-adic Mumford type, then either $\mathrm{MT}_p$ holds for $A$, or $G_\mathrm{B}(A)_{\overline{\bQ}}=\GSp_{8,\overline{\bQ}}$.
\end{enumerate}
\begin{proof}
    Since $(G_\mathrm{B}(A),\rho_{\mathrm{B}})$ is a strong Mumford--Tate pair, \cite[Proposition 4.5]{P98} implies that either $G_\mathrm{B}(A)_{\overline{\bQ}}=\GSp_{8,\overline{\bQ}}$, or $A$ is of \textit{Mumford type}, i.e., if $(G_\mathrm{B}(A)_{\overline{\bQ}}^{\der},\rho_{\mathrm{B}})$ is the tensor product of three copies of the standard representation of $\SL_{2,\overline{\bQ}}$. Let $l$ be a suitable prime. Applying \cite[Proposition 4.5]{P98} to the group $G$ and its tautological representation in Theorem~\ref{thm:Pinkinterpolation}, we see that either $G_l(A)_{\overline{\bQ}_l}^\circ=\GSp_{8,\overline{\bQ}_l}$, or $A$ is of $p$-adic Mumford type. As long as $A$ is not of $p$-adic Mumford type, $\MTT_p$ holds for $A$. 
\end{proof}

\end{lemma}
\begin{proof}[Proof of Theorem~\ref{modpMoonen-Zarhin}] If $A_{\overline{k}}$ is not simple, then (up to isogeny) it splits into a product of simple ones of dimension $1,2$ or $3$. We then apply Theorem~\ref{modpCommelin} and Theorem~\ref{modpTank} to conclude that $\MTT_p$ holds for $A$. Therefore from now on we can assume that $A_{\overline{k}}$ is a simple abelian fourfold. Moonen and Zarhin obtain a thorough classification of (the endomorphism algebra, the Hodge group, and the $l$-adic monodromy group of) simple abelian fourfolds over a number field; cf. \cite[Table 1]{MZ95}. Of course, it is possible to modify the $l$-adic part of their argument to our setting. However, as long as $\MTT_p$ is concerned, we can run a shorter argument based on what we have proved so far.   

Let $D=\End^0(A_{\overline{k}})$.
First, suppose that $D=\bQ$. Then we conclude by Lemma~\ref{prop:mtreduction1} that $\MTT_p$ holds for $A$. Second, suppose that $D$ is of degree 8 over $\bQ$. Then $A$ is either CM or $D$ is an indefinite quaternion over a real quadratic field $F$. The CM case is easy. In the indefinite quaternion case,  $G_{\mathrm{B}}(A)$ is described in \cite[\S 5.1]{MZ95}: geometrically, its Lie algebra is generated by the one dimensional homothety and  $\mathfrak{sl}_2\oplus \mathfrak{sl}_2$, and the monodromy rep is  two copies of the std rep of the first $\mathfrak{sl}_2$ direct sum two copies of the std rep of the second $\mathfrak{sl}_2$. It is easy to see that the $l$-adic monodromy can not be smaller for otherwise the endomorphism algebra becomes bigger. 

Now we are reduced to treat the case where $1<[\End^0(A_{\overline{k}}):\bQ]<8$. The representation $(G_{\mathrm{B}}(A)_{\overline{\bQ}},\rho_{\mathrm{B}})$ splits into a direct sum of Galois isotypical irreducible sub-representations $V_1,...,V_{8/d}$ of the same dimension $d\in \{2,4\}$. If each $V_i$ is a tensor product of irreps of dimension $\leq 3$, then we can read from \cite[Table 4.2]{P98} that $G_p(A)^{\circ\der}_{\bC_p}$ surjects onto each simple factor of $G_{\mathrm{B}}(A)_{\bC_p}^{\der}$. 
Since $\End_{\bC_p}(\rho_\mathrm{B})=\End_{\bC_p}(\rho_p)$, each simple factor of $G_p(A)^{\circ\der}_{\bC_p}$ must sit inside a unique simple factor of $G_{\mathrm{B}}(A)_{\bC_p}^{\der}$. It follows that $\MTT_p$ holds.

This proves $\MTT_p$ when $d=2$, as well as cases of $d=4$ where $V_1$ has non-trivial tensor decomposition. So we are reduced to the case where $d=4$, and the Mumford--Tate pair associated to $V_1$ is simple. Read \cite[Table 4.2]{P98} again, this happens only when the simple Mumford--Tate pair is a (1) std rep of $C_2$ or (2) std rep (or its dual) of $A_3$ (note that the Spin$^+$ rep of $D_3$ is the same as std rep of $A_3$, while the Spin rep of $B_2$ is the same as std rep of $C_2$). Now $C_2$ does not admit a smaller Mumford--Tate pair with an dimension $4$ irrep, so $\MTT_p$ holds in case (1). On the other hand, case (2) can potentially admit smaller Mumford--Tate pairs with dimension $4$ irrep: this happens when the cocharacter multiplicity is $(2,2)$, and the smaller Mumford--Tate pairs are either the std rep of $C_2$, or $A_1\times A_1$ with  tensor product rep (of the std rep of $A_1$). We remind the reader that case (2) with cocharacter multiplicity $(2,2)$ happens 
 exactly when $A_{\overline{\epsilon}}$ is of the case \cite[Type IV(1,1)(ii)]{MZ95}: in this case $\End^0(A_{\overline{k}})$ is an imaginary quadratic field and $G_{\mathrm{B}}(A)_{\overline{\bQ}}$ is generated by the rank 1 torus of homothety in $\GSp_8$ together with a special unitary group (which is isomorphic to $\mathrm{SL}(V_1)$), and the rep $\rho_{\mathrm{B}}$, when restricted to $\mathrm{SL}(V_1)$, identifies with the direct sum of its std rep and the dual of its std rep. We can either adapt the proof there to conclude $\MTT_p$ in this case, or we can use the following argument\footnote{This argument is pointed out to the author by Yunqing Tang, based on her earlier discussion with Victoria Cantoral Farfan.}: if $G_p(A)^\circ_{\bC_p}$ is strictly smaller than $G_{B}(A)_{\bC_p}$, then, as discussed above 
$G_p(A)^\circ_{\bC_p}$ is generated by the rank 1 torus of homothety together with either $\mathrm{SO}_4$ or $\mathrm{Sp}_4$ (corresponding to $A_1\times A_1$ or $C_2$). In both cases the rep of $G_p(A)^\circ_{\bC_p}$ on $V_1\otimes \bC_p$ is self-dual. As a result, $\rho_{p}$ has larger endomorphism algebra, giving you non-existing endomorphisms of $A$. Contradiction.    
\end{proof}
\subsubsection{Abelian fourfolds of $p$-adic Mumford type} Let $f:X\rightarrow \mathscr{A}_{g,\Fpbar}$ be a map whose image lies generically in the ordinary locus, and let $(X_0,f_0)$ be a finite field model. Write $A$ for the abelian variety over the generic point of $X_0$. 
\begin{lemma}\label{lm:multislopes}
    Suppose that $g=4$ and $A$ is of $p$-adic Mumford type. Then $C:=\overline{\im f}\subseteq \mathscr{A}_{4,\Fpbar}$ is a proper Tate-linear curve. 
\end{lemma}
\begin{proof}
  By definition,  $(G_p(A)_{\overline{\bQ}_p}^{\der},\rho_p)$ is the tensor product of three copies of tautological representation of $\SL_{2,\overline{\bQ}_p}$. To distinguish, we denote the three copies by $\theta_{p,i}:\SL_{2,\overline{\bQ}_p}\acts V_i$, where $i\in \{1,2,3\}$. Let $\nu_x:\bG_m\rightarrow G_p(A)$ be the Newton cocharacter of an ordinary point $x\in X(\Fpbar)$. Factorize $\nu_x$ via each $\theta_{p,i}$ to a cocharacter $\nu_{x,i}:\bG_m\rightarrow \SL_{2,\overline{\bQ}_p}$. If $\nu_{x,i}$ is non-trivial, then it induces a two-step slope filtration on $V_i$.  If there are more than one $i$'s such that $\nu_{x,i}$ is not trivial, then the slope filtration on $\bH_{\cris,X}$ has more than two slopes, violating the fact that $A$ is ordinary. As a result,
    there is exactly one $i$ for which $\nu_{x,i}$ is non-trivial. This implies that the unipotent subgroup of $G_p(A)$ corresponding to $\nu_x$ has dimension 1. We conclude by Theorem~\ref{Thm:Tatelocal} that $ C=\overline{\im f}\subseteq \mathscr{A}_{4,\Fpbar}$ is a Tate-linear curve. 
    
   To show that $C$ is compact, we apply \cite[Proposition 2.1]{Noot00}. We need to find a place $l\in \mathrm{fpl}(\bQ)$ such that (1) $(G_l(A)_{\overline{\bQ}_l},\rho_l)$ is of Mumford type, and (2) $\Lie (G_l(A))^{\mathrm{ss}}$ is $\bQ_l$-simple. The first condition is easy. For (2), we argue as the second half of \cite[Lemma 1.3]{Noot00}. Noot's original proof uses \cite[Theorem 5.13]{P98}, while we use its char $p$ analogue  Theorem~\ref{thm:Pinkinterpolation} instead. 
\end{proof}
\begin{proof}[Proof of Corollary~\ref{cor:geoAOdim5}]
 If none of the geometrically simple factors of $A$ is of dimension 4, then $\MTT_p$ holds for $f$ by Theorem~\ref{modpTank} and \ref{modpCommelin}, so $\mathrm{geoAO}_p$ holds for $X$. Otherwise, $A$ splits as the product of a simple abelian fourfold $A'$ and an elliptic curve, up to isogeny. Accordingly, $f$ factors through $\mathscr{A}_{4,\Fpbar}\times \mathscr{A}_{1,\Fpbar}$. If $A'$ is not of $p$-adic Mumford type, then $\MTT_p$ again holds for $A$ by Theorem~\ref{modpMoonen-Zarhin}, and $\mathrm{geoAO}_p$ holds for $X$. Suppose that $A'$ is of $p$-adic Mumford type. Let $C\subseteq \mathscr{A}_{4,\Fpbar}$ be as in Lemma~\ref{lm:multislopes}, then $X$ is contained in the quasi-weakly special subvariety $T=C\times \mathscr{A}_{1,\Fpbar}$. If $\dim X=1$, then $\mathrm{geoAO}_p$ trivially holds for $X$. If $\dim X=2$, then $X$ is dense in $T$, $\mathrm{geoAO}_p$ again holds for $X$.  
\end{proof}

\begin{proof}[Proof of Corollary~\ref{cor:Ttconj}]
  If $A$ is not of $p$-adic Mumford Type, then the Mumford--Tate conjecture is known ( \cite[\S 4.1]{MZ95} in char 0 and Theorem~\ref{modpMoonen-Zarhin} in char $p$), so combining Markman's result (\cite{Markman}) we are done. If $A$ is of $p$-adic Mumford type, then the validity of the Tate conjecture can be checked by hand   by the argument given in \cite[\S 4.1]{MZ95} (which is purely group theoretic and works in both char 0 and char $p$). We remind the reader that the proof works for $p=2$ as well.  
\end{proof}
\section{Tate-linearity and algebraization}\label{sec:Newmethod}
 We prove the Tate-linear conjecture for a special case by studying the geometry of the Tate-linear loci (Construction \ref{const:ttloci}) via integral $p$-adic Hodge theory. Consider $f:X\rightarrow \mathscr{A}_{g,\mathbb{F}}$, where $X$ is a generically ordinary irreducible $\mathbb{F}$-variety. Let $X^\circ\subseteq X$ be the smooth ordinary locus. We will always  assume that $p>2$. %By \textbf{$p$-adic monodromy of $X$}, we mean the neutral component of the monodromy group $G(\bH_{\cris, X^\circ})$. 
We say that $X$ has \textbf{unramified} $p$-adic monodromy, if $G(\bH_{\cris, X^\circ},x)$ (with fiber functor $\omega^{\mathrm{DM}}_x$) is unramified. 
\iffalse
one of the following two equivalent conditions holds 
\begin{enumerate}
    \item $G(\bH_{\cris, X^\circ},x)$ (with fiber functor $\omega^{\mathrm{iso}}_x$) admits a reductive group model over $W$.
    \item $G(\bH_{\cris, X^\circ},x)$ (with fiber functor $\omega^{\mathrm{DM}}_x$) has the property that $G(\bH_{\cris, X^\circ},x)\otimes K$ admits a reductive group model over $W$.\footnote{This is weaker than the usual notion of unramifiedness since we don't require quasi-splitness.} 
\end{enumerate}
\fi
 % The property of being unramified, however, is independent of the choice. 
 The following is the main result of this section: 
\begin{theorem}\label{thm:TTLpintegraldcs}
  Suppose that $T\subseteq \mathscr{A}_{g,\Fpbar}$ is a proper Tate-linear subvariety with unramified $p$-adic monodromy. Then $\mathrm{Tl}_{p}$ holds for $T$, i.e., $T$ is a component of the mod $p$ reduction of a special subvariety.  
\end{theorem}
The theorem has the following application to the missing piece of $\MTT_p$ for abelian fourfolds. 
Let $A$ be an ordinary abelian fourfold defined over a function field $k$ over $\Fpbar_q$. We say that $A$ has \textbf{everywhere potentially good reduction}, if its moduli point in $\mathscr{A}_{g,\Fpbar_q}$ has proper Zariski closure.
\begin{corollary}\label{cor:coomologyMT}
  Suppose that $A$ is of $p$-adic Mumford type and has unramified $p$-adic monodromy, then $\MTT_p$ holds for $A$. 
\end{corollary}
These will be proved at the end of \S\ref{sub:PTLC}.  
\subsection{Preliminaries} In this section, we review some technical setups. 
\subsubsection{$p$-integrality} Consider $f:X\rightarrow \mathscr{A}_{g,\mathbb{F}}$, where $X$ is an irreducible connected $\mathbb{F}$-variety, not assumed to be generically ordinary. Let $X^{\mathrm{sm}}$ be the smooth locus of $X$. 
%whose image lies generically in the ordinary locus.Let $\mathbf{M}_0$ be the pullback Dieudonné crystal over $X_0$. 

\begin{definition}\label{def:cohomologyunramified}
   Let $x\in X^\mathrm{sm}(\Fpbar)$. We say that $X$ is \textbf{$p$-integral at $x$}, if there exist tensors 
   $\{s_\alpha\}\subseteq \omega_x^{\mathrm{iso}}(\mathbb{H}_{\cris})^{\otimes}$, such that their stabilizer $\mathcal{G}\subseteq \GL(\omega_x^{\mathrm{iso}}(\mathbf{H}_{\cris}))$\footnote{Here $\omega_x^{\mathrm{iso}}(\mathbf{H}_{\cris})$ is the $W$-lattice $\mathbf{H}_{\cris,x}$, sitting inside the $K$-space $\omega_x^{\mathrm{iso}}(\mathbb{H}_{\cris})$. } is a reductive subgroup scheme over $W$, and $\mathcal{G}_{K}=G(\mathbb{H}_{\cris,X^\mathrm{sm}},x)^{\circ}$ (here $G(\mathbb{H}_{\cris,X^\mathrm{sm}},x)$ is taken with respect to the fiber functor $\omega_x^{\mathrm{iso}}$. In particular, it is a reductive group over $K$). %We say that $f$ is \textbf{$p$-integral} if it is $p$-integral at every geometric point. 
\end{definition}
Note that even when $X$ is generically ordinary, being $p$-integral at an ordinary point is stronger than having  unramified $p$-adic monodromy. Indeed, when speaking of $p$-integrality, we are working with a fixed integral structure: the $W$-lattice $\omega_x^{\mathrm{iso}}(\mathbf{H}_{\cris})\subseteq \omega_x^{\mathrm{iso}}(\mathbb{H}_{\cris})$. On the other hand, in order to guarantee unramified $p$-adic monodromy, we only need the existence of a reductive integral model $\mathcal{G}$ for a certain $W$-latice $H\subseteq \omega_x^{\mathrm{iso}}(\mathbb{H}_{\cris})$.

In general, we don't expect that $p$-integrality is a global property, i.e., if $X$ is $p$-integral at one point, it is in general not true that $X$ is $p$-integral at another point. For such an example, one can take $X$ to map to the supersingular loci. However, when $X$ is generically ordinary, we do expect that $p$-integrality is a global property, see \S\ref{subsub:HTofTcan} for a partial result.  

\subsubsection{Chai's Tate-linear canonical lifting}
Suppose that $T\subseteq\mathscr{A}^{\Sigma}_{g,\Fpbar}$ is a closed, generically ordinary Tate-linear subvariety. Let $T^{\circ}\subseteq T$ be the smooth ordinary locus. Chai proved that one can canonically lift $T^{\circ}$ to a locally closed $p$-adic formal subscheme of $(\mathscr{A}^{\Sigma}_{g,W})^{\wedge p}$:
\begin{proposition}[Chai]\label{prop:chaitatelinearlifting}
Let $U=\mathscr{A}^{\ord}_{g,\Fpbar}-(T\setminus T^{\circ})$ be the open subset obtained by throwing out non-smooth points of $T$, and let  $\mathscr{A}_{g,W}^{/U}\subseteq(\mathscr{A}_{g,W}^\Sigma)^{\wedge p}$ be the corresponding open formal subscheme. There is a unique closed formal subscheme    $\widetilde{T^{\circ}}\subseteq \mathscr{A}_{g,W}^{/U}$, formally smooth over $W$, such that $$\widetilde{T^{\circ}}\times_{\Spf W} \Spec\Fpbar = T^{\circ}$$ and the formal completion of $\widetilde{T^{\circ}}$ at every closed point $z\in T^{\circ}$ is a formal subtorus of the Serre-Tate formal torus $\mathscr{A}_{g,W}^{/z}$. 
\end{proposition}
\begin{proof}
    The proof is almost identical to \cite[Proposition 5.5]{Ch03}, where the result is stated in a slightly less general form. We briefly give a sketch here. It suffices to show that, for every positive integer $n$, there exists a unique lift $\widetilde{T^\circ_n}$ of $T^\circ$ over $W_n=W/p^n$, such that the formal completion of $\widetilde{T^{\circ}_n}$ at every closed point $z\in T^{\circ}$ is a formal subtorus of the truncated Serre--Tate formal torus $\mathscr{A}_{g,W_n}^{/z}$. For this, one can work locally. It suffices to show that if $Z\subseteq T^{\circ}$ is an affine open subset, then we can uniquely lift $Z$ to a scheme $Z_n$ formally linear over $W_n$. This can be done inductively on $n$ using global Serre--Tate coordinates, see the third paragraph of Chai's proof. One then uses the uniqueness of the lifting to glue local liftings together, which yields $\widetilde{T^\circ_n}$.
\end{proof}
\begin{remark}
 $\widetilde{T^\circ}$ is smooth connected of dimension $\dim T$ in the sense of \cite{Conrad}.
\end{remark}
\iffalse
The integral $p$-adic Hodge theory of $\widetilde{T^\circ}$ is more or less well known, and can be approached by more classical (as compared to developments after 2010) methods. In the following, let $\bbH_{\cris, T^{\circ}}^\otimes$ be the tensor space $\bigoplus $. A global section of $\bbH_{\cris, T^{\circ}}^\otimes$ is a global section of a $\bbH_{\cris}^{m,n}$. \textcolor{red}{to do}
\fi
\begin{remark}\label{rmk:etalecoverofChaicanlift} In practice, we will often need certain étale covers of $\widetilde{T^\circ}$ which are constructed as follows. Let $f:X^{\circ}\rightarrow T^\circ$ be a connected finite étale cover. Since taking formal thickening has no effect on the étale site, there is a unique formal scheme étale over ${\mathscr{A}}_{g,W}^{/T^\circ}$, with special fiber $X^{\circ}$. This formal scheme will be denoted by ${\mathscr{A}}_{g,W}^{/{X^{\circ}}}$. We will also denote by $\widetilde{{X^{\circ}}}$ the pullback of $\widetilde{T^\circ}$ to ${\mathscr{A}}_{g,W}^{/{X^{\circ}}}$. Note that $\widetilde{{X^{\circ}}}$ is again smooth connected of dimension $\dim T$, with special fiber ${X^{\circ}}$ (the connectedness of $\widetilde{{X^{\circ}}}$ follows from the fact that any of its connected component has special fiber ${X^{\circ}}$). 

%In practice, we will often take ${X^{\circ}}$ to be a connected étale cover with the property that $G(\bH_{\cris,{X^{\circ}}})=G(\bH_{\cris,T^\circ})^{\circ}$.
\end{remark} 
Chai's Tate-linear canonical lifting has nice integral $p$-adic Hodge theoretic properties. Fix an $\widetilde{{X^{\circ}}}$ as in Remark~\ref{rmk:etalecoverofChaicanlift}, and let $]{X^{\circ}}[$ be the rigid tube of ${X^{\circ}}$ in $\mathscr{A}_{g,W}^{/{X^{\circ}}}$. As we explained in Construction~\ref{const:ttloci}, the work of Berthlot and Ogus enables us to identify 
$\bH_{\cris, {X^{\circ}}}$ with the de Rham bundle $\bH_{\dR,]{X^{\circ}}[}$ together with a Frobenius structure. An integral tensor $s\in\bbH_{\cris, {X^{\circ}}}^\otimes$ gives rise to a horizontal section $s^{\mathrm{an}}\in\bH_{\dR,]{X^{\circ}}[}^\otimes$. We can then restrict $s^{\mathrm{an}}$ to a horizontal section 
$$s^{\mathrm{an}}|_{(\widetilde{{X^{\circ}}})^{\mathrm{rig}}}\in \bH_{\dR,(\widetilde{{X^{\circ}}})^{\mathrm{rig}}}^\otimes.$$
\begin{lemma}\label{lm:Tatelineartensorconstruction}  Assumption as above. The rational section $s^{\mathrm{an}}|_{(\widetilde{{X^{\circ}}})^{\mathrm{rig}}}$ extends to an integral horizontal section of $\bbH_{\dR,\widetilde{{X^{\circ}}}}^\otimes$ that furthermore lies in $\Fil^0\bbH_{\dR,\widetilde{{X^{\circ}}}}^\otimes$. 
\end{lemma}
\begin{proof}
Without loss of generality, we can take ${X^{\circ}}=T^\circ$. Since the statement is local, we can moreover assume that $\widetilde{T^\circ}=\Spf A$ is affine, where $A$ is an admissible $W$-algebra.

If $T^\circ$ is a single point, then $\widetilde{T^\circ}$ is its canonical lift, and the statement is well known. In the following we reduce the general case to the point case. 

Back to the setting where  $\widetilde{T^\circ}=\Spf A$. Suppose that $s \in \bbH_{\cris, T^{\circ}}^{m,n}$. Then $\bbH_{\dR,\widetilde{T^\circ}}^{m,n}$ is a bundle over $\widetilde{T^\circ}$, with sub-bundle $\Fil^0\bbH_{\dR,\widetilde{T^\circ}}^{m,n}$. Further shrinking $A$, we can assume that the two bundles are trivial. There is some $r\geq 0$ such that the section $s$ lies in $p^{-r} \bbH_{\dR,\widetilde{T^\circ}}^{m,n}$. Therefore it makes sense to talk about the loci $\mathfrak{V}\subseteq \widetilde{T^\circ}$ where $s \in \Fil^0\bbH_{\dR,\widetilde{T^\circ}}^{m,n}$. This is a Zariski closed formal subscheme of $\widetilde{T^\circ}$. It suffices to show that 
$\mathfrak{V}=\widetilde{T^\circ}$. 

For $x\in T^{\circ}(\Fpbar)$, let $\Tilde{x}$ be its canonical lift, which lies in $\widetilde{T^\circ}(W)$  by the last statement of Proposition~\ref{prop:chaitatelinearlifting}.  It follows that $\mathfrak{V}$ contains all $\Tilde{x}$ (i.e., the ideal defining $\mathfrak{V}$ is contained in the ideal defining $\Tilde{x}$). Let $I$ be the ideal of $\mathfrak{V}$ in $A$.  
Since $A$ is Noetherian, there are finitely many irreducible components of $\mathfrak{V}$, defined by minimal primes lying above $I$. By what we have seen so far, there is an irreducible component $\mathfrak{V}_0=\Spf (A/I_0)$ not contained in the special fiber, whose mod $p$ reduction is $T^\circ$. Then $I_0\subseteq (p)$. This implies that $I_0=(0)$. It follows that $I=(0)$ and $\mathfrak{V}=\widetilde{T^\circ}$. 
\end{proof}

\subsection{The Tate-linear loci}\label{sub:Toy} Let $T\subseteq\mathscr{A}_{g,\Fpbar}$ be a smooth Tate-linear subvariety, not assumed to be proper. From Construction~\ref{const:ttloci}, we have the so called Tate-linear loci $\mathfrak{T}\subseteq ]T[$ cut out by the crystalline cycles $\{s_\alpha\}$. We will show that $\mathfrak{T}$ is an extension of Chai's Tate-linear canonical lift $\widetilde{T^\circ}$. 

We will also fix a connected étale cover $X\rightarrow T$ with the property that $G(\bH_{\cris,X})=G(\bH_{\cris,T})^{\circ}$. Let $X^\circ$ be the inverse image of $T^\circ$ in $X$. From the process in Remark~\ref{rmk:etalecoverofChaicanlift}, we get a smooth $p$-adic formal scheme $\widetilde{X^\circ}$ étale over $\widetilde{T^\circ}$. 
%For our purpose, we need an overconvergent upgrade of construction~\ref{const:ttloci}: Let $T^{\mathrm{cl}}\subseteq \mathscr{A}_{g,\Fpbar}^\Sigma$ be the Zariski closure of $T$, and let $]T^{\mathrm{cl}}[$ be the tubular neighborhood of $T^{\mathrm{cl}}$ in $(\mathscr{A}_{g,W}^\Sigma)^{\wedge p}$. Since $\bH_{p,T}$ is overconvergent and the sections $\{s_\alpha\}$ used in Construction~\ref{const:ttloci} are fixed by the overconvergent monodromy group, the locus $\mathfrak{T}$ defined in Construction~\ref{const:ttloci} extends to a strict neighborhood of $]T[$ in $]T^{\mathrm{cl}}[$. If not otherwise specified, we will think of $\mathfrak{T}$ as defined over a fixed strict neighborhood $\mathfrak{U}_{0}$. 

%By \cite[]{Conrad}, the dimension of an irreducible component of an \textcolor{red}{admissible} formal scheme/rigid variety is well-defined. 

\begin{lemma}\label{lm:TinT}
 We have $(\widetilde{T^\circ})^{\mathrm{rig}}\subseteq \mathfrak{T}$.
\end{lemma}\begin{proof}
Let $X^\circ$ be as above. Replacing $\{s_\alpha\}$ by sufficiently large $p$-power multiples, we can assume that they are all integral tensors of $\bbH_{\cris, X^\circ}^\otimes$. Apply Lemma~\ref{lm:Tatelineartensorconstruction} to conclude. 
\end{proof}
The main input of the following theorem is Crew's parabolicity conjecture, which is encapsulated in Theorem~\ref{Thm:Tatelocal}.
\begin{theorem}\label{thm:localstructureofT}
   Let $x\in T^{\circ}(\Fpbar)$ and let $]x[$ be its tubular neighborhood in $\mathscr{A}_{g,{W}}^{\wedge p}$. Each irreducible component of $\mathfrak{T} \cap]x[$ is the generic fiber of the translate of $(\widetilde{T^\circ})^{/x}$ by a torsion point. In particular, $(\widetilde{T^\circ})^{\mathrm{rig}}\cap ]x[$ is the unique irreducible component of $\mathfrak{T} \cap]x[$ that contains $\Tilde{x}^{\mathrm{rig}}$.
\end{theorem}
  \begin{proof}
  Let $\mathfrak{C}$ be an irreducible component of $\mathfrak{T} \cap]x[$. Let $\mathfrak{z}$ be a $\overline{K}$-point of $\mathfrak{C}$. Then $ \mathfrak{z}$ extends to a $\overline{W}$-point $\overline{\mathfrak{z}}$ of $]x[$. By \cite[Theorem 2.8]{N96}, there is a maximal formal subscheme $\mathcal{N}\subseteq \mathscr{A}_{g,\overline{W}}^{/x}$ passing through $\overline{\mathfrak{z}}$, such that the tensors $\{s_\alpha|_{\mathfrak{z}}\}$ extend to $\Fil^0\bbH_{\dR,\mathcal{N}}^\otimes$ up to some $p$-power multiple. Moreover, the formal subscheme $\mathcal{N}$ is a translate of a formal subtorus by a torsion point, whose dimension equals the dimension of the opposite unipotent of $\mu_x$ in $G(\bH_{\cris,T})^{\circ}$. Combining Theorem~\ref{Thm:Tatelocal}(3), we find that $\dim \mathcal{N}=\dim T$. It is clear that $\mathcal{N}^{\mathrm{rig}} \subseteq \mathfrak{T}$. By the maximality of $\mathcal{N}$ we also have $\mathfrak{C}\subseteq \mathcal{N}^{\mathrm{rig}}$. This is enough to show that $\mathfrak{C}=\mathcal{N}^{\mathrm{rig}}$:
  first, $\mathfrak{T}\cap ]x[$ must have a component that contains $\mathcal{N}^{\mathrm{rig}}$ (one can show this by first restricting to an admissible open). Now this component contains $\mathfrak{C}$, so must be identical to $\mathfrak{C}$. As a result, $\mathfrak{C}=\mathcal{N}^{\mathrm{rig}}$. 

  Take $\overline{\mathfrak{z}}= \Tilde{x}$ for this moment. Then we have $(\widetilde{T^\circ})^{/x}\subseteq \mathcal{N}$: by Lemma~\ref{lm:TinT}, the tensors $\{s_{\alpha,\Tilde{x}^{\mathrm{rig}}}\}$ extend to $\Fil^0\bbH^\otimes_{\dR,(\widetilde{T^\circ})^{/x}}$ up to some $p$-power multiple. Counting dimensions, we see that $(\widetilde{T^\circ})^{/x}=\mathcal{N}$, hence $$\mathfrak{C}=\mathcal{N}^{\mathrm{rig}}=(\widetilde{T^\circ})^{\mathrm{rig}}\cap ]x[.$$ Now back to the original setting that ${\mathfrak{z}}$ is an arbitrary point. We want to show that $\mathcal{N}$ is a translate of $(\widetilde{T^\circ})^{/x}$ by a torsion point. First note that $\mathcal{N}$ contains a quasi-canonical lift $\Tilde{x}'$ of $x$. The torsion translate $\Tilde{x}'+(\widetilde{T^\circ})^{/x}$ passes through $\Tilde{x}'$, and satisfies the property that the tensors $\{s_{{\alpha,\mathfrak{z}}}\}$ extend to $$\Fil^0\bbH^\otimes_{\dR,\Tilde{x}'+(\widetilde{T^\circ})^{/x}}$$ up to some $p$-power multiple. By maximality of $\mathcal{N}$, as well as dimension counting, we must have $\mathcal{N}=\Tilde{x}'+(\widetilde{T^\circ})^{/x}$. This concludes the theorem.
  \end{proof}

\begin{corollary}\label{cor:containTcan}
    There is a unique irreducible Tate-linear locus $\mathfrak{T}_{\mathrm{can}}$ associated to $T$ such that $]T^\circ[\cap \mathfrak{T}_{\mathrm{can}}$ contains $(\widetilde{T^\circ})^{\mathrm{rig}}$ as an irreducible component. The dimension of $\mathfrak{T}_{\mathrm{can}}$ is equal to $\dim T$.  
\end{corollary}  \begin{proof}
By Theorem~\ref{thm:localstructureofT}, we can take an irreducible component $\mathfrak{T}_{\mathrm{can}}$ of $\mathfrak{T}$ whose restriction to $]x[$ contains $(\widetilde{T^\circ})^{\mathrm{rig}}\cap ]x[$ as an irreducible component. Then $$\dim \mathfrak{T}_{\mathrm{can}}=\dim (\widetilde{T^\circ})^{\mathrm{rig}}\cap ]x[= \dim T.$$ The uniqueness is clear: if $\mathfrak{T}_{\mathrm{can}}'$ is any other irreducible component with the same property, then $$\mathfrak{T}_{\mathrm{can}}'\cap\mathfrak{T}_{\mathrm{can}}\supseteq(\widetilde{T^\circ})^{\mathrm{rig}}\cap ]x[.$$ This happens only when $\mathfrak{T}_{\mathrm{can}}'=\mathfrak{T}_{\mathrm{can}}$. The proof that $]T^\circ[\cap \mathfrak{T}_{\mathrm{can}}$ contains $(\widetilde{T^\circ})^{\mathrm{rig}}$ as an irreducible component can be carried out in a similar manner. 
\end{proof}

We will call $\mathfrak{T}_{\mathrm{can}}$ the \textbf{canonical Tate-linear locus}.

\begin{lemma}[Moonen]\label{lm:moonenlemma}
    If $\mathfrak{T}_{\mathrm{can}}$ is algebraizable, i.e., there is an algeberaic subvariety $Z\subseteq\mathcal{A}_{g,\overline{K}} $ such that $Z^{\mathrm{an}}\cap ]T[$ admits $\mathfrak{T}_{\mathrm{can}}$ as an irreducible component. Then $Z$ is special. In particular, $\mathrm{Tl}_p$ holds for $T$. 
\end{lemma}
\begin{proof}
    We can assume that $Z$ is irreducible. From \cite[\S 3.6]{M98}, we know that $(\widetilde{T}^\circ)^{\mathrm{rig}}$ contains a Zariski dense collection of CM points. This implies that $Z$ contains a Zariski dense collection of CM points. We conclude by André--Oort conjecture that $Z$ is special (see also \cite[Theorem 4.5]{M98} if one wants to avoid the André--Oort conjecture). 
\end{proof}
  
\subsection{Integral $p$-adic Hodge theory on $\mathfrak{T}_{\mathrm{can}}$}\label{subsub:HTofTcan} In this section we make the further assumption that there is an $x\in T^\circ(\Fpbar)$ at which $T$ is $p$-integral; cf, Definition~\ref{def:cohomologyunramified}. Let $\{s_\alpha\}$ be a collection of crystalline tensors cutting out $G(\bH_{\cris,T},x)^\circ$, such that their pointwise stablizer in $\GL(\omega_x(\bbH_{\cris}))$ is a reductive $W$-subgroup scheme. We will identify $\bH_{\cris,T}$ with $\bH_{\dR,]T[}$ together with a Frobenius structure. From  Construction~\ref{const:ttloci} we know that each $s_\alpha$ gives rise to a ``multivalued analytic horizontal section'' $s_\alpha^{\mathrm
an}$ of the bundle $\bH^\otimes_{\dR,]T[}$, and that $\{s_\alpha^{\mathrm{an}}\}$ cuts out the Tate-linear loci $\mathfrak{T}\subseteq ]T[$.  

Let $X$ and $X^\circ$ be as in the beginning of \S\ref{sub:Toy}. As we observed in (the footnote of)  Construction~\ref{const:ttloci}, the sections $\{s_\alpha^{\mathrm{an}}\}$ are horizontal sections of the bundle $\bH_{\dR,]X[}^\otimes$. Let $\pi:]X[\rightarrow ]T[$ be the étale map of tubes, where we recall that $]X[$ is the rigid tube of $X$ in $\mathscr{A}_{g,W}^{/X}$. Let $\mathfrak{X}_{\mathrm{can}}$ be the unique irreducible component of $\pi^{-1}({\mathfrak{T}}_{\mathrm{can}})$ that contains $(\widetilde{X^\circ})^{\mathrm{rig}}$ (it can also be constructed as the unique component of the loci in $]X[$ cut out by the condition that $\{s_\alpha^{\mathrm{an}}\}\subseteq\Fil^0\bH^\otimes_{\dR,]X[}$) 

Let $\widetilde{\mathfrak{X}}_{\mathrm{can}}$ be the normalization of $\mathfrak{X}_{\mathrm{can}}$, and let $\widetilde{\mathfrak{X}}_{\mathrm{can}}^{\mathrm{sm}}\subseteq \widetilde{\mathfrak{X}}_{\mathrm{can}}$ be the Zariski open subspace obtained by throwing out the singular loci (which is of codimension at least 2). Then $\widetilde{\mathfrak{X}}_{\mathrm{can}}^{\mathrm{sm}}$ is smooth and connected. 

Let $\mathfrak{z}\rightarrow \widetilde{\mathfrak{X}}_{\mathrm{can}}$ be a $\overline{K}$-point (by a $\overline{K}$-point we  mean a point defined over a sufficiently large finite extension of $\bZ_p$ which is subject to be replaced by larger extensions whenever needed). % should be viewed as, a $W(k)(p^{\frac{1}{e}})[p^{-1}]$-point of $(\mathscr{A}_{g,\bZ_p}^{\wedge p})^{\mathrm{rig}}$ for some sufficiently large finite field $k$ and sufficiently large ramification index $e$; the field $k$ and the index $e$ are not considered as fixed, and are subject to be replaced by larger ones whenever needed).

The restriction of $\{s_{\alpha,\mathfrak{z}}^{\mathrm{an}}\}$ to $\mathfrak{z}$ lie in $\Fil^0\mathbb{H}^{\otimes}_{\dR,\mathfrak{z}}$, giving rise to Frobenius invariant de Rham tensors $\{s_{\alpha,\dR, \mathfrak{z}}\}$. By Fontaine's theory, they correspond to a collection of $p$-adic étale tensors
\begin{equation}\label{eq:petaletensors}
    \{s_{\alpha,\et,\mathfrak{z}}\}\subseteq \mathbb{H}^{\otimes}_{p,\et,\mathfrak{z}}.
\end{equation}
 Under $p$-adic comparisons, the étale tensors $\{s_{\alpha,\et,\mathfrak{z}}\}$ are carried to de Rham and crystalline tensors $\{s_{\alpha,z}\}$ and $\{s_{\alpha,\dR, \mathfrak{z}}\}$, where $z$ is the reduction of $\mathfrak{z}$.  We say that $\mathfrak{z}$ is \textbf{$p$-étale integral}, if $\{s_{\alpha,\et,\mathfrak{z}}\}$ cuts out a reductive $\bZ_p$-group subscheme in $\GL(\mathbf{H}_{p,\et,\mathfrak{z}})$.

\begin{proposition}\label{thm:petaleintegral}
 Assumption and notation as above. Every $\overline{K}$-point $\mathfrak{z}\rightarrow \widetilde{\mathfrak{X}}_{\mathrm{can}}^\mathrm{sm}$ is $p$-étale integral.    
\end{proposition}\begin{proof}
If $\mathfrak{z}_0=\Tilde{x}^{\mathrm{rig}}$ is the canonical lift of $x$, then any pre-image of $\mathfrak{z}_0$ in $\widetilde{\mathfrak{X}}_{\mathrm{can}}$ lies in $\widetilde{\mathfrak{X}}_{\mathrm{can}}^{\mathrm{sm}}$, and is $p$-étale integral. The strategy isto compare the étale tensors $\{s_{\alpha,\et,\mathfrak{z}}\}$ over a point $\mathfrak{z}$ to that of the canonical lift $\mathfrak{z}_0$. To achieve this, we need to construct $p$-adic étale tensors in families, using a relative version of Fontaine's theory (cf. \cite[Proposition 3.46(3)]{du2024logprismaticfcrystalspurity}). First, we can make some simplifications. Since $\widetilde{\mathfrak{X}}_{\mathrm{can}}^{\mathrm{sm}}$ is connected, we can find finitely many connected affinoids that connect $\mathfrak{z}_0$ and $\mathfrak{z}$ (this is by definition of connectedness, see \cite{Conrad}). Therefore it suffices to show the following: if $\mathfrak{B}\subseteq \widetilde{\mathfrak{X}}_{\mathrm{can}}^{\mathrm{sm}}$ is a connected affinoid that contains a $p$-étale integral point $\mathfrak{z}_1$, then so is every point in $\mathfrak{B}$. 

By Raynaud's theory, we can find an admissible formal model $\mathscr{B}$ of $\mathfrak{B}$ with a map $\mathscr{B}\rightarrow \mathscr{A}_{g,W}^{/T}$. %(\mathscr{A}_{g,K}^\Sigma)^{/T^{\mathrm{cl}}}.
 By \cite[Theorem 3.3.1]{TM17}, possibly base change to a finite extension $K'/K$, there is a formal scheme $\mathscr{B}'\rightarrow \mathscr{B}_{W'}$ which is an étale morphism over the generic fiber, with the property that $\mathscr{B}'$ is strictly semi-stable over $W'$ (here $W'$ is the integral ring of bounded elements in $K'$). Let $\mathfrak{B}'$ be the rigid generic fiber of $\mathscr{B}'$. 

Make $\mathscr{B}'$ into a log formal scheme $(\mathscr{B}', M)$ where $M$ comes from the semi-stable structure; cf. \cite[Notation and conventions]{du2024logprismaticfcrystalspurity}. Over the log convergent site of $(\mathscr{B}'_{\Fpbar}, M)$, we have a log convergent $F$-isocrystal $\mathbb{H}_{\cris,(\mathscr{B}'_{\Fpbar}, M)}$ that is pulled back from $\mathbb{H}_{\cris,T}$. The pro-étale $\bbZ_p$-local system $\mathbf{H}_{p,\et,\mathfrak{B}'}$ is canonically associated with $\mathbb{H}_{\cris,(\mathscr{B}'_{\Fpbar}, M)}$ in the sense of \cite[\S 3.6]{du2024logprismaticfcrystalspurity}: indeed, it is just a matter of pulling back the canonical association over $\mathscr{A}_{g,W}^{\wedge p}$ to $(\mathscr{B}', M)$. On the other hand, the tensors $\{s_{\alpha}\}$ also pull back to tensors $$\{s^{\log}_{\alpha,(\mathscr{B}'_{\Fpbar}, M)}\}\subseteq\mathbb{H}_{\cris,(\mathscr{B}'_{\Fpbar}, M)}^\otimes.$$  Moreover, $\{s_{\alpha,\mathfrak{B}'}^{\mathrm{an}}\}$ lie in $\Fil^0\bH_{\dR,\mathfrak{B}'}^\otimes$ by construction, giving rise to de Rham tensors $\{s_{\alpha,\dR,\mathfrak{B}'}\}$. Since a tensor of a local system can be understood as a morphism from the trivial local system to an object in the tensor category generated by the local system, we can appeal to \cite[Proposition 3.46(3)]{du2024logprismaticfcrystalspurity} and obtain $p$-adic étale tensors $\{s_{\alpha,\et,\mathfrak{B}'}\}$ of the local system  $\mathbb{H}_{p,\et,\mathfrak{B}'}$ from de Rham and log crystalline ones  $\{s_{\alpha,\dR,\mathfrak{B}'}\}$ and $\{s^{\log}_{\alpha,(\mathscr{B}'_{\Fpbar}, M)}\}$.  The following diagram summarizes the relation between spaces, local systems, and tensors: 

\begin{center}
    % https://tikzcd.yichuanshen.de/#N4Igdg9gJgpgziAXAbVABwnAlgFyxMJZARgBoAmAXVJADcBDAGwFcYkQAKAHS4Ft6cACzgBjAE7AAQgF8A5KQAEAWQCUIaaXSZc+QigAMFanSat23PgOHipcgPrAeAMTQAjemI3K1GrdjwERORGNAwsbIggPPxCAGZi9ADWtrLqmiAY-rpEhsTGYWaR0VaursAAEtIOPOJYcKQWMdYSMrLVXC7unoqq0gB6PBB4vPBpfjqBBqT6+aYRUVzAcH2OXIwQAOZVq0xogvQNxUKiLfarnR5evTzSYxnaAXok07Ph7DxL7bv7pDwwOL9LHEEslWtIbndMhMnmQ8qE5u8gYJShV+oNhvAHGhAf9AU14kkUrdfPcspNkABmF7wt5FRZwL6MPYHHhQABKeKsBNBcnBXGJ6ShjyIVLhJlpCycWEYfX0PFc5Xa7M5wMJYIGXCGWBGCBJQuyKDIFNehQWTROtgcAHVZALxsKUMFjTTTUdBNzbJCHgbKaRneLXVwAO5YWB4RiwVb4kHAAAa0m2brEvGAInoYATKyTKbgvATXrJTwALH6TfNKAAVZAF6FEEtigrzD7LKNWZPAdMJxnMwHsvl20m1lAlqguptIlGVLOajFwRNcZWV5AD-XkqmjgPjqUyuUTspT1ZanXz5U8ENh6WRt0e+PzprttMZ-qtoTt3MJlfe8nBItlxHmmwAEFtg2UgrWfAB6CtiWMGAoA2UYUFAeIIF4JAyBAHAICQQxMPoaV2EECAIESO4ULQxBgkw7DEFwnB8MYQjiMSJgcDIsRUKQEtqKQKlN0ROBmFcOB-hgABHBQQBoRgsDAeY4AgGSoCkkBBBgehlMiSA5JU+iCK0gg2Gk+hXBgRgAAUvz0EBZOwWB2M4xA+Kw9CSXIpAAFYaBcxAADZjNMiyrPYWzQyM-i6UE4TRIk3SGPYbTwpknTIgUpSVLUjSEsMhyKP8njKLcjiKIADm8mjyCKxyAHZyqQMqQEYEyzMswsQrAOzwsbAShJEnBxJU5L5MU0MMvUzTwBy7z4oMnSqoouiaIATnmpAqJ84h9FWxBNrqxAlum-TVOY1jcvQxb0IwvTGMiIiSLOnarpo4hKvSdydu4jbXuQ4r0Oc57uKawLWqHGyOrClTusi3qYrGrLZrYbbiH+9CvMa2ThvSmhMomxK4qOvGApa4LIlC+yxx66L+rEh7iDRjbiCRlHEFqkBTLATSAFoKVwob2Cgeg4DU5TtoOgriF8pH1ue-LrqYkjTukShpCAA
\begin{tikzcd}
{\{s^{\log}_{\alpha,(\mathscr{B}'_{\Fpbar}, M)}\}} \arrow[d, "\subseteq " description, no head] & {\{s_{\alpha,\et,\mathfrak{B}'}\}} \arrow[d, "\subseteq " description, no head] &                                               & {\{s_{\alpha,\dR,\mathfrak{B}'}\}} \arrow[d, "\subseteq" description, no head]                     &                                                                                                                                                            \\
{\mathbb{H}_{\cris,(\mathscr{B}'_{\Fpbar}, M)}^\otimes} \arrow[d]                               & {\mathbb{H}^\otimes_{p,\et,\mathfrak{B}'}} \arrow[rd]                           &                                               & {\Fil^0\bH_{\dR,\mathfrak{B}'}^\otimes} \arrow[ld]                                                 & {\{s^{\mathrm{an}}_{\alpha,\dR}\}} \arrow[d, "\subseteq" description, no head] \arrow[ld, "\subseteq" description, no head] \arrow[lu, dashed, bend right] \\
{(\mathscr{B}'_{\Fpbar}, M)} \arrow[r, hook]                                                    & {(\mathscr{B}', M)} \arrow[d]                                                   & \mathfrak{B}' \arrow[l, hook'] \arrow[d]      & {\Fil^0\mathbb{H}^{\otimes}_{\dR,\widetilde{\mathfrak{X}}_{\mathrm{can}}^{\mathrm{sm}}}} \arrow[d] & {\mathbb{H}^{\otimes}_{\dR,]T[}} \arrow[d]                                                                                                                 \\
                                                                                                & \mathscr{B}_{W'} \arrow[rd]                                                     & \mathfrak{B} \arrow[l, hook'] \arrow[r, hook] & \widetilde{\mathfrak{X}}_{\mathrm{can}}^{\mathrm{sm}} \arrow[r]                                    & {]T[} \arrow[lld, hook']                                                                                                                                   \\
                                                                                                &                                                                                 & {\mathscr{A}_{g,W}^{/T}}                      &                                                                                                    &                                                                                                                                                           
\end{tikzcd}
\end{center}
If $\mathfrak{z}\rightarrow\mathfrak{B}'$ is a $\overline{K}$-point, the pullback of $s_{\alpha,\et,\mathfrak{B}'}$ to $\mathfrak{z}$ is just the tensor $s_{\alpha,\et,\mathfrak{z}}$ in (\ref{eq:petaletensors}). Suppose that $\mathfrak{z}_1$ is a $p$-étale integral $\overline{K}$-point of $\mathfrak{B}$. Let $\mathfrak{B}'_+\subseteq \mathfrak{B}'$ be a connected component where $\mathfrak{z}_1$ lifts to a $p$-étale integral point $\mathfrak{z}_{1+}$. Since $\mathfrak{B}'_+$ is connected, and since $\{s_{\alpha,\et,\mathfrak{B}'_+}\}$ is a collection of global tensors of the $p$-adic étale local system $\bH_{p,\et,\mathfrak{B}'_+}$, every $\overline{K}$-point of the image of $\mathfrak{B}'_+\rightarrow \mathfrak{B}$ is $p$-étale integral. Since the connected components of $\mathfrak{B}'$ jointly cover $\mathfrak{B}$, we can translate $p$-étale integrality over the overlaps to see that every point of $\mathfrak{B}$ is $p$-étale integral. 
\end{proof}
%\subsubsection{Geometry of $\mathfrak{T}_{\mathrm{can}}$} Assumption being the same as \S\ref{subsub:HTofTcan}. 

\begin{proposition}\label{thm:smooththeorem}
   Suppose that $T$ is proper and is $p$-integral at a point in $T^\circ(\Fpbar)$. Then for every point $z\in T(\Fpbar)$, $T$ is $p$-integral at $z$, and there exists  a dimension $\dim T$ formal $W$-subscheme $\mathscr{T}_z\subseteq \mathscr{A}^{/z}_{g,W}$
   linear in the Lie theoretic coordinates 
   such that $\mathscr{T}_z^{\mathrm{rig}}=\mathfrak{T}_{\mathrm{can}}\cap ]z[$ and $\mathscr{T}_z\otimes{\Fpbar}=T^{/z}$. In particular, $\mathscr{T}_z=(\widetilde{T^\circ})^{/z}$ when $z\in T^\circ(\Fpbar)$ and  $\mathfrak{T}_{\mathrm{can}}\cap ]T^\circ[ = (\widetilde{T^\circ})^{\mathrm{rig}}$.
\end{proposition}
\begin{proof}
First, every $z\in T(\mathbb{F})$ lies in the specialization of $\mathfrak{T}_{\mathrm{can}}$. This follows from 
Lemma~\ref{lm:specialization_lemma} by taking $\mathscr{A},Z,\mathfrak{Z},\mathfrak{V}$ to be $\mathscr{A}_{g}, T, \mathfrak{T}_{\mathrm{can}},(\widetilde{T^\circ})^{\mathrm{rig}}$. We will use  $\mathfrak{B}$ in the third paragraph.

Let $z\in T(\mathbb{F})$. Write $M_0=\bD(A_z[p^{\infty}])(W)$ as in \S\ref{subsub:lethcoor}, the tensors $\{s_\alpha\}$ restrict to tensors $\{s_{\alpha,z}\}\subseteq M_0^\otimes$. Let $G\subseteq \GL(M_0)$ be the group scheme cut out by $\{s_{\alpha,z}\}$. We have $\mathfrak{T}_{\can}\cap ]z[\neq \varnothing$, so by dimension reasons there exist points in $\widetilde{\mathfrak{X}}_{\can}^{\mathrm{sm}}$ specializing to $z$. Let $\mathfrak{z}$ be any point of   $\widetilde{\mathfrak{X}}_{\can}^{\mathrm{sm}}$ specializing to $z$. Let $\mathrm{cl}({\mathfrak{z}})$ be the closure of $\mathfrak{z}$ in $\Spf R:=\mathscr{A}_{g,\overline{W}}^{/z}$. Note that $\{s_{\alpha,z}\}$ is exactly the image of $\{s_{\alpha,\et,\mathfrak{z}}\}$ under the $p$-adic comparison isomorphism. It follows from  Theorem~\ref{thm:BKcriterion} and the $p$-étale integrality of $\{s_{\alpha,\et,\mathfrak{z}}\}$ from Proposition~\ref{thm:petaleintegral}, that $G$ is reductive, and the image of $\mathrm{cl}({\mathfrak{z}})$ lies in $\Spf R_G$ (using the notation from that section). This implies that $z$ is $p$-integral. 

Let $\mathscr{T}_z:= \Spf R_G$ as above, which is linear in the Lie theoretic coordinates of dimension $\dim T$. Every point in $\widetilde{\mathfrak{X}}_{\can}^{\mathrm{sm}}$ that maps to $]z[$ factors through $\mathscr{T}_z^{\mathrm{rig}}$. By dimension reasons ($\widetilde{\mathfrak{X}}_{\mathrm{can}}^{\mathrm{sm}}\subseteq \widetilde{\mathfrak{X}}_{\mathrm{can}}$ is a Zariski open subspace obtained by throwing out the codimension $\geq$2 singular loci), we have ${\mathfrak{T}}_{\can}\cap ]z[=\mathscr{T}_z^{\mathrm{rig}}$. To show that $\mathscr{T}_z\otimes \Fpbar=T^{/z}$, let $\mathscr{B}$ be an admissible formal model of $\mathfrak{B}$ given by Lemma~\ref{lm:specialization_lemma}(2), which maps to $\mathscr{A}_{g,W}^{/T}$ and contains a irreducible proper subvariety $Y$ which maps surjectively onto $T$. Let $y\in Y(\Fpbar)$ be a point that maps to $z$. Then $Y^{/y}\rightarrow T^{/z}$ is surjective. On the other hand, the map $\mathscr{B}^{/y}\rightarrow \Spf R$ factors through $\Spf R_G$ since it does on the generic fiber. Therefore $T^{/z}$ is contained in $\Spf R_G\otimes \Fpbar$, which must be equal to $\Spf R_G\otimes \Fpbar$ by dimension reasons. 
% So $\mathfrak{T}_{\mathrm{can}}\cap ]z[$ is linear in the Lie theoretic coordinates. 

%The fact that   $\mathfrak{T}_{\mathrm{can}}\cap ]T^\circ[ = (\widetilde{T^\circ})^{\mathrm{rig}}$ follows easily from Theorem~\ref{thm:localstructureofT}. 
\end{proof}
\begin{lemma}\label{lm:specialization_lemma}
 Let $R$ be a complete $p$-adic DVR with fraction field $K$ and residue field $\mathbb{F}$. Let $\mathscr{A}$ be a smooth admissible formal $R$-scheme and let $Z\subseteq \mathscr{A}_\mathbb{F}$ be a closed subscheme which is proper over $\mathbb{F}$. Suppose that $\mathfrak{Z}\subseteq ]Z[_{\mathscr{A}}$ is a closed rigid analytic subspace and there is a quasi-compact admissible open $\mathfrak{V}\subseteq \mathfrak{Z}$ such that the specialization of its $\overline{K}$-points are Zariski dense in $Z$. Then there exists a quasi-compact admissible open $\mathfrak{B}\subseteq \mathfrak{Z}$ containing $\mathfrak{V}$ such that 
 \begin{enumerate}
     \item every point $z\in Z(\Fpbar)$ lies in the specialization of $\overline{K}$-points of $\mathfrak{B}$. 
     \item $\mathfrak{B}$ admits a formal model $\mathscr{B}\rightarrow \mathscr{A}^{/Z}$ such that $\mathscr{B}_{\Fpbar}$ contains an irreducible proper subvariety $Y$ such that $Y\rightarrow Z$ is surjective. 
     \end{enumerate}
\end{lemma}
\begin{proof}
By \cite[Lemma 4.19]{Nearby_Cycle},  $]Z[_{\mathscr{A}}$ is partially proper over $K$. So $\mathfrak{Z}$, as a closed rigid subspace, 
is also partially proper over $K$. As a result, there is a quasi-compact admissible open $\mathfrak{B}$ containing $\mathfrak{V}$ and contained in $\mathfrak{Z}$ such that $ \mathfrak{V}\Subset\mathfrak{B}$ (i.e., $\mathfrak{V}$ is relatively compact in $\mathfrak{B}$). Now by Raynaud theory, one can find admissible formal model $\mathscr{V}\subset\mathscr{B}$ of this inclusion. Up to admissible blow up, we can further assume that there is a map $\mathscr{B}\rightarrow \mathscr{A}^{/Z}$ extending $\mathfrak{B}\rightarrow ]Z[_{\mathscr{A}}$. Let $$Y=\mathrm{cl}_{\mathscr{B}_{\mathbb{F}}}({\mathscr{V}_{\mathbb{F}}}),$$ 
i.e., the Zariski closure of ${\mathscr{V}_{\mathbb{F}}}$ in $\mathscr{B}_{\mathbb{F}}$. Then relative compactness implies that $Y$ is proper over $\Fpbar$. The natural  map $Y\rightarrow Z$ need to be surjective since its image contains the dense subset of $Z$ that comes from the specialization of $\overline{K}$-points of $ \mathfrak{Z}$. It follows that $\mathscr{B
}_{\mathbb{F}}\rightarrow Z$ is surjective. Therefore every $\mathbb{F}$-point of $Z$ lies in the specialization of $\mathfrak{B}$. The second assertion is also clear: we just take an irreducible component of $Y$ which dominates $Z$. The dominant map needs to be surjective by properness.  
\end{proof}
%\begin{remark}\label{rmk:exist_B}
 %   From the proof of Lemma~\ref{lm:specialization_lemma} one can see that  
%\end{remark}

\subsection{Algebraization I: smooth case}\label{sub:PTLC}
In the following two sections we establish the Tate-linear conjecture for proper Tate-linear subvarieties when $T$ is $p$-integral at a point. We begin by the case where the Tate-linear subvariety is smooth to familiarize the reader with our argument: 
\begin{theorem}\label{thm:TTlinearsmooth}
    Suppose that $T\subseteq \mathscr{A}_{g,\Fpbar}$ is a smooth proper Tate-linear subvariety. If $T$ is $p$-integral at an ordinary point $x$, then there is a special subvariety $Z\subseteq \mathcal{A}_{g,\overline{K}}$, such that $Z^{\mathrm{an}}= \mathfrak{T}_{\mathrm{can}}$. In particular, $\mathrm{Tl}_p$ holds for $T$. 
\end{theorem}
\begin{proof}
  Let $V\subseteq T$ be an affine open, so that $V$ is cut out from $\mathscr{A}_{g,W}^{/V}$ via an ideal $(p,f_{1},f_{2},...,f_{m})$. \begin{claim}
      There exists an $\epsilon< 1$, such that the $p$-adic norm $|f_{i}(\mathfrak{z})|$ is bounded above by $\epsilon$ for all $\overline{K}$-points $\mathfrak{z}$ in $\mathfrak{T}_{\mathrm{can}}\cap ]V[$ and all $i\in [1,m]$ (Intuitively, this means that $\mathfrak{T}_{\mathrm{can}}\cap ]V[$ is contained in a smaller tube inside $]V[$). 
  \end{claim}
  Assuming the claim, we can deduce the theorem as follows: the claim implies that we can find an admissible cover $]V[= ]V[_{>\epsilon}\cup ]V[_{<(1+\epsilon)/2}$, where $$]V[_{>\epsilon}=\{\mathfrak{z}:\exists i, |f_i(\mathfrak{z})|>\epsilon \},$$ and $$]V[_{<(1+\epsilon)/2}=\{\mathfrak{z}:\forall i, |f_i(\mathfrak{z})|<(1+\epsilon)/2 \},$$ such that $$\mathfrak{T}_{\mathrm{can}}\cap ]V[ \subseteq ]V[_{< (1+\epsilon)/2},\qquad\mathfrak{T}_{\mathrm{can}}\cap ]V[_{>\epsilon}= \varnothing.$$ Now cover $T$ by finitely many affine subschemes, and use the claim and the implication described in the last a few sentences. Since $T$ is proper, we find that there exists an admissible cover $$(\mathcal{A}_{g,\overline{K}}^\Sigma)^{\an}=\mathfrak{V}_1\cup \mathfrak{V}_2,\qquad\mathfrak{V}_1\subseteq ]T[$$ such that $\mathfrak{T}_{\mathrm{can}}$ is a Zariski closed subset of $\mathfrak{V}_1$, and $\mathfrak{T}_{\mathrm{can}}\cap \mathfrak{V}_2=\varnothing$. This implies that $\mathfrak{T}_{\mathrm{can}} \subseteq (\mathcal{A}_{g,\overline{K}}^\Sigma)^{\an}$ is Zariski closed. Therefore we conclude by rigid GAGA that $\mathfrak{T}_{\mathrm{can}}=Z^{\mathrm{an}}$ for some algebraic subvariety $Z\subseteq\mathcal{A}_{g,\overline{K}}^\Sigma$, which in fact lies in $\mathcal{A}_{g,\overline{K}}$. We conclude by Lemma~\ref{lm:moonenlemma}.

It remains to prove the \textit{Claim}. Let $V$ be an affine open as above and put $V^\circ=V\cap T^\circ$ and $\partial V:=V\cap (T\setminus T^\circ)$. By Proposition~\ref{prop:chaitatelinearlifting}, we know that $\widetilde{T}^\circ$ is a locally closed formal subscheme of $(\mathscr{A}_{g,W}^{\Sigma})^{\wedge p}$. Since $]T^\circ[\cap \mathfrak{T}_{\mathrm{can}}= (\widetilde{T}^\circ)^{\mathrm{rig}}$ (Proposition~\ref{thm:smooththeorem}), 
we find that $|f_i(\mathfrak{z})|$ is bounded by some $\epsilon<1$ for all $\mathfrak{z}\in \mathfrak{T}_{\mathrm{can}}\cap ]V^\circ[$. Now let $z\in \partial V(\Fpbar)$. By Proposition~\ref{thm:smooththeorem}, $\mathfrak{T}_{\mathrm{can}}\cap ]z[$ admits a linear formal model $\mathscr{T}_z=\Spf R_{G}$ such that $\Spf R_{G}\otimes \mathbb{F}=T^{/z}$. Let $\mathcal{R}$ (resp. $R$) be the structure ring of $\mathscr{A}^{/V}_{g,W}$ (resp. $\mathscr{A}^{/z}_{g,W}$), then $\forall i$ the image of $f_i$ under the map 
$$\mathcal{R}\rightarrow R\rightarrow R_G $$
lies in $pR_G$, because a further composition with $R_G\rightarrow R_G\otimes \mathbb{F}$ kills it. Therefore for all $\mathfrak{z}\in  \mathfrak{T}_{\mathrm{can}}\cap ]z[$, we have $|f_i(\mathfrak{z})| \leq |p|$. Possibly shrinking $\epsilon$, we find that 
$|f_i(\mathfrak{z})|\leq \epsilon$ for all $\mathfrak{z}\in \mathfrak{T}_{\mathrm{can}}\cap ]V[$, as desired. 
%this would imply the \textit{Claim} when $V\subseteq T^\circ$. Therefore it remains to show that for any $\Fpbar$-point $t\in T\setminus T^{\circ}$ and 

%Since $T$ is a curve, there are only finitely points $t_1,t_2,...,t_n$ in $T\setminus T^\circ$. If $t_i$ does not lie in the specialization of $\mathfrak{T}_{\mathrm{can}}$ then \textit{Claim} is vacuously true for a small Zariski neighborhood of $t_i$. If $t_i$ lies in the specialization of $\mathfrak{T}_{\mathrm{can}}$, then
%by Proposition~\ref{thm:smooththeorem} and its proof, $\mathfrak{T}_{\mathrm{can}}\cap ]t_i[$ is linear in the Lie theoretic coordinates and the mod $p$ reduction of $\Spf R_{G}$  equals $T^{/t_i}$. We see that the \textit{Claim} remains valid for a small Zariski neighborhood of $t_i$. The validity of the claim follows clearly from the above discussion.  

\end{proof}
\subsection{Algebraization II: general case}\label{sub:PTLC2}Now we remove the condition that $T$ is smooth. This brings some technical difficulty, and in particular involves two variants of Construction~\ref{const:ttloci} to accommodate the singular situation. 

\subsubsection{Local geometry of $T$}  We will first show that $p$-integrality at an ordinary point implies that $T$ has nice local geometry; cf. Proposition~\ref{prop:local_geometry}. This relies on the following variant of Construction~\ref{const:ttloci}:
\begin{construction}[Tate-linear loci of an alteration graph]\label{const:graph_ttloci}
Let $\pi:X\rightarrow T$ be an smooth projective alteration. Embed $X$ into a projective space $\mathbb{P}_W$ over $W$. The graph $\Gamma_\pi$ is a smooth closed subvariety of the special fiber of $\mathbb{P}_W\times \mathscr{A}_{g,W}$. Let $]\Gamma_\pi[$ be the tubular neighborhood of $\Gamma_\pi$ in $(\mathbb{P}_W\times \mathscr{A}_{g,W})^{\wedge p}$. 
We equip $\mathbb{P}_W\times \mathscr{A}_{g,W}$ with an abelian scheme and various cohomology sheaves that pulled back from 
those on $\mathscr{A}_{g,W}$ via projection. %We then canonically identify the convergent $F$-isocrystal $\bH_{\cris,\mathbb{P}_{\Fpbar}\times \mathscr{A}_{g,{\Fpbar}}}$ with the de Rham bundle $\bH_{\dR,((\mathbb{P}_{W}\times \mathscr{A}_{g,{W}})^{\wedge p})^{\mathrm{rig}}}$, together with a Frobenius structure.
Similar to Construction~\ref{const:ttloci}, we can canonically identify the convergent $F$-isocrystal $\bH_{\cris,\Gamma_\pi}$ as the de Rham bundle $\bH_{\dR,]\Gamma_\pi[}$, together with a Frobenius structure and a two step Hodge filtration $\mathrm{Fil}^\bullet$. Consider a finite collection of crystalline tensors $\{s_\alpha\}$ cutting out the overconvergent monodromy group $G(\bH_{\cris,T^\circ})^{\circ}$. Possibly replacing $X$ by a finite étale cover, we can assume that $G(\bH_{\cris,X})=G(\bH_{\cris,T^\circ})^{\circ}$. Then each $s_\alpha$ gives rise to a horizontal section $s_\alpha^{\mathrm{an}}$ of an element $\bM_\alpha\in \bH_{\dR,]\Gamma_\pi[}^{\otimes}$. Take the locus 
$\mathfrak{I}_{\pi,\alpha}\subseteq ]\Gamma_\pi[$ over which $s_\alpha^{\mathrm{an}}\in  \Fil^0 \bM_\alpha$, and take $\mathfrak{T}_{\pi}=\bigcap_{\alpha} \mathfrak{I}_{\pi,\alpha}$. This is a locally Noetherian Zariski closed analytic subspace of $]\Gamma_\pi[$. The locus $\mathfrak{T}_\pi$ is a variant of the Tate-linear locus in $]\Gamma_\pi[$. 
\end{construction}
Let $\Gamma_{\pi}^\circ\subseteq \Gamma_{\pi}$ be the open subset that maps to $T^\circ$. Let $]T[$ be the tube of $T$ inside $\mathscr{A}_{g,W}^{\wedge p}$, then we have the following commutative diagram of rigid spaces \begin{center}
% https://tikzcd.yichuanshen.de/#N4Igdg9gJgpgziAXAbVABwnAlgFyxMJZARgBpiBdUkANwEMAbAVxiRAoBUA9YAHV4DGWAE4CAvshBjS6TLnyEUAJnJVajFm06Tps7HgJEADKur1mrRCAAU-AO5ZYeBrGDc+gkeLEBKHvwBbOhwAC2EA4GEsAHMxMSkZEAx9BSIyIzVzTSsKfgBxOgCggH0PNCwxLn4hUR1E5PlDZVIMsw1LdnzCkrKKnTUYKGj4IlAAM2EIAKQVEBwIJBM5uiwGNhCICABrBPHJ6cQAZmp5pAAWE5W1qw3t3ZAJqcWThcQyZdX1zZ3dB-2kY5zV5LLIdfj0YTlfyeUTxX6PA4XIFId6gtjguiQrBSChiIA
\begin{tikzcd}
                                                       & {]\Gamma_{\pi}^\circ[} \arrow[r, hook] \arrow[d, "\varpi^{\circ}"] & {]\Gamma_{\pi}[} \arrow[d, "\varpi"] \\
(\widetilde{T^{\circ}})^{\mathrm{rig}} \arrow[r, hook] & {]T^{\circ}[} \arrow[r, hook]                                      & {]T[}                               
\end{tikzcd}
\end{center}
Most results in \S\ref{sub:Toy} have  variants for $\mathfrak{T}_{\pi}$: they can be obtained by pulling back the original version via $\varpi$, so we omit the proof. 
\begin{manualtheoremc}{\ref{lm:TinT}*}\label{lm:TinT2} 
We have $\varpi^{-1}(\widetilde{T^\circ})^{\mathrm{rig}} \subseteq \mathfrak{T}_{\pi}$. 
\end{manualtheoremc}
\begin{manualtheoremb}{\ref{thm:localstructureofT}*}
    \label{thm:localstructureofT2}
   Let $x\in T^{\circ}(\Fpbar)$, then each irreducible component of $  \mathfrak{T}_{\pi}\cap \varpi^{-1}]x[$ is an irreducible component of the preimage under $\varpi^{-1}$ of the generic fiber of a torsion-point-translate of $(\widetilde{T^\circ})^{/x}$. 
\end{manualtheoremb}
\begin{manualtheorema}{\ref{cor:containTcan}*}\label{cor:containTcan2}
    There is a unique irreducible component  $\mathfrak{T}_{\pi,\mathrm{can}}\subseteq \mathfrak{T}_\pi$ such that $]\Gamma_\pi^\circ[\cap \mathfrak{T}_{\pi,\mathrm{can}}$ contains $\varpi^{-1}(\widetilde{T^\circ})^{\mathrm{rig}}$ as an irreducible component. 
\end{manualtheorema}
 Let $\widetilde{\mathfrak{T}}_{\pi,\mathrm{can}}$ be the normalization of $\mathfrak{T}_{\pi,\mathrm{can}}$ and let $\widetilde{\mathfrak{T}}_{\pi,\mathrm{can}}^{\mathrm{sm}}$ be the Zariski open subspace obtained by throwing out the singular loci.  We have the following variant of Proposition~\ref{thm:petaleintegral},which can be proved in exactly the same manner 
\begin{manualtheoremd}{\ref{thm:petaleintegral}*}\label{thm:petaleintegral22}
Suppose that $\Gamma_\pi$ is $p$-integral at a point. Then every $\overline{K}$-point $\mathfrak{z}\rightarrow \widetilde{\mathfrak{T}}_{\pi,\mathrm{can}}^{\mathrm{sm}}$ is $p$-étale integral.   
\end{manualtheoremd}

Finally, we can use these to prove the following 
\begin{proposition}\label{prop:local_geometry} Suppose that $\Gamma_\pi$ is $p$-integral at a point, then $\Gamma_\pi$ is $p$-integral at every point. For any $z\in T(\Fpbar)$, $T^{/z}$ is a reduced formal scheme and if we write $\mathscr{A}_{g,W}^{/z}=\Spf R$ with Lie theoretic coordinates, then each irreducible component $\mathscr{C}\subseteq T^{/z}$ is the mod $p$ reduction of a linear formal subscheme. In particular, $T$ becomes smooth after normalization and the normalization map is unramified.   
\end{proposition}
\begin{proof}
    This is similar to Proposition~\ref{thm:smooththeorem}. First we show that every $\gamma\in \Gamma_\pi(\Fpbar)$ lies in the specialization of $\mathfrak{T}_{\pi,\can}$. Clearly, $\omega^{-1}(\widetilde{T^\circ})^{\mathrm{rig}}$ contains a quasi-compact admissible open $\mathfrak{U}$ whose specialization is $\Gamma_{\pi}^\circ$. Now invoking
Lemma~\ref{lm:specialization_lemma}, taking $\mathscr{A},Z,\mathfrak{Z},\mathfrak{V}$ to be $\mathbb{P}_W\times\mathscr{A}_{g,W}, \Gamma_\pi, \mathfrak{T}_{\pi,\mathrm{can}},\mathfrak{U}$, we find that there is a quasi-compact open $\mathfrak{B}\subseteq \mathfrak{T}_{\pi,\mathrm\can}$ whose specialization is all of  $\Gamma_\pi$. We will use $\mathfrak{B}$ later.  

Let $z\in T(\mathbb{F})$, then by excellence of $T$, $T^{/z}\subseteq \Spf R$ is a reduced formal subscheme  all of whose irreducible components have dimension $\dim T$. Let $\mathscr{C}$ be an irreducible component. Since $\Gamma_\pi\rightarrow T$
 is surjective, there is a $\gamma\in \Gamma_\pi(\mathbb{F})$ mapping to $z$, such that the map $\Gamma_\pi^{/\gamma}\rightarrow\Spf R$ factors surjectively through $\mathscr{C}$ (surjectivity just means that the corresponding map between structure rings is injective). 
 
Write $M_0=\bD(A_\gamma[p^{\infty}])(W)$ as in \S\ref{subsub:lethcoor}, the tensors $\{s_\alpha\}$ restrict to tensors $\{s_{\alpha,\gamma}\}\subseteq M_0^\otimes$. Let $G\subseteq \GL(M_0)$ be the group scheme cut out by $\{s_{\alpha,\gamma}\}$. Since we have $\mathfrak{T}_{\pi,\can}\cap ]\gamma[\neq \varnothing$, by dimension reasons there exist points in $\widetilde{\mathfrak{T}}_{\pi,\can}^{\mathrm{sm}}$ specializing to $\gamma$. Let $\mathfrak{z}$ be any $\overline{K}$-point of  $\widetilde{\mathfrak{T}}_{\pi,\can}^{\mathrm{sm}}$ that specializes to $\gamma$. Let $\mathrm{cl}({\mathfrak{z}})$ be the closure of $\mathfrak{z}$ in $(\mathbb{P}_{\overline{W}}\times\mathscr{A}_{g,\overline{W}})^{/\gamma}$. Using a same argument as in Proposition~\ref{thm:smooththeorem} together with Proposition~\ref{thm:petaleintegral22}, we find that $G$ is reductive and the map $\mathrm{cl}({\mathfrak{z}})\rightarrow \Spf R$ factors through $\Spf R_G$. This implies that $\gamma$ is $p$-integral and the map ${\mathfrak{T}}_{\pi,\can}\cap ]\gamma[ \rightarrow ]z[=(\Spf R)^{\mathrm{rig}}$ factor through $(\Spf R_G)^{\mathrm{rig}}$.

\begin{claim}
    The map $\Gamma_\pi^{/\gamma}\rightarrow \Spf R$ factors through $\Spf R_G$.
\end{claim}
Assuming the claim, we immediately find that $\mathscr{C}=\Spf (R_G/p)$. This finishes the theorem.

It remains to prove the claim. Let $\mathfrak{B}$ be the quasi-compact admissible open from the first paragraph, and let $\mathscr{B}\rightarrow (\mathbb{P}_W\times \mathscr{A}_{g,W})^{/\Gamma_\pi}$ be an admissible formal model such that $\mathscr{B}_{\Fpbar}$ contains an irreducible proper subvariety $Y$ mapping surjectively onto $\Gamma_\pi$. Let $y\in Y(\Fpbar)$ be a point that maps to $\gamma$, then $Y^{/y}\rightarrow \Gamma_{\pi}^{/\gamma}$ is surjective. Therefore it suffices to show that the map $Y^{/y}\rightarrow \Spf R$ factors through $\Spf R_G$. This reduces to showing that $\mathscr{B}^{/y}\rightarrow \Spf R$ 
factors through $\Spf R_G$, which follows immediately from the fact that $\mathfrak{B}\cap ]\gamma[\rightarrow ( \Spf R)^{\mathrm{rig}}$ factors through $ (\Spf R_G)^{\mathrm{rig}}$.

\end{proof}

% When $T$ only has mild singularity, a similar argument works with minor modification. More precisely, if $T$ is normal or more generally, geometrically unibranch, then  
%Suppose that $T\subseteq \mathscr{A}_{g,\Fpbar}$ is a proper Ta

\subsubsection{Geometry in the tame case}\label{subsub:sinsincase} Assumption as in the last section: let $T\subseteq\mathscr{A}_{g,\Fpbar}$ be a proper Tate-linear subvariety which is $p$-integral at an $x\in T^{\circ}(\Fpbar)$. By Proposition~\ref{prop:local_geometry}, for each $z\in T(\Fpbar)$, $T^{/z}$ is a reduced formal subscheme of $\mathscr{A}^{/z}_{g,W}$ each of whose irreducible component is the mod $p$ reduction of a linear formal subscheme in the Lie theoretic coordinates. Even though $T$ is not smooth, its geometry is very tame. An irreducible variety whose normalization is smooth and the normalization map 
is unramified will be called \textbf{tame}. 

Our treatment relies on yet another variant of Construction~\ref{const:ttloci}: 
\begin{construction}[Tate-linear loci in the tame case]\label{const:graph_ttloci} Suppose that $T$ is tame. Let $T'\rightarrow T$ be the normalization and let $X$ be a finite étale cover of $T'$ such that $G(\bH_{\cris,X})=G(\bH_{\cris,T^\circ})^{\circ}$. Note that the map $\xi:X\rightarrow T$ is still finite and unramified. By \cite[Theorem 1.1]{etale_evelope} and the remarks after that, there exists an algebraic space  $E:=E_{X/\mathscr{A}_{g,W}}$ (called the \textbf{étale envelope}) such that the unramified map $X\rightarrow \mathscr{A}_{g,W}$ factors as 
$$X\xrightarrow{i} E\xrightarrow{\pi} \mathscr{A}_{g,W},$$
where $i$ is a closed embedding, and $\pi$ is an étale morphism. We warn the reader that $\pi$ is in general not separated. Identify $X$ with its image in $E$ and take completion $E^{/X}$; see \cite[tag 0GXT]{stacks-project} for the property of completion in the algebraic space setting. This is a separated formally smooth  formal scheme over $W$; cf., \cite[tag 05ZR]{stacks-project}. Let $]X[$ be the rigid generic fiber of $E^{/X}$.

Similar to Construction~\ref{const:ttloci}, we can canonically identify the convergent $F$-isocrystal $\bH_{\cris,X}$ as the de Rham bundle $\bH_{\dR,]X[}$, together with a Frobenius structure and a two step Hodge filtration $\mathrm{Fil}^\bullet$. Consider a finite collection of crystalline tensors $\{s_\alpha\}$ cutting out the overconvergent monodromy group $G(\bH_{\cris,X})$. Then each $s_\alpha$ gives rise to a horizontal section $s_\alpha^{\mathrm{an}}$ of an element $\bM_\alpha\in \bH_{\dR,]X[}^{\otimes}$. Take the locus 
$\mathfrak{I}_{\alpha}\subseteq ]X[$ over which $s_\alpha^{\mathrm{an}}\in  \Fil^0 \bM_\alpha$, and take $\mathfrak{X}=\bigcap_{\alpha} \mathfrak{I}_{\alpha}$. This is a locally Noetherian Zariski closed analytic subspace of $]X[$. This is a variant of the Tate-linear loci in the tame case.  %Since the map $]X[\rightarrow]T[$ is proper (here $]T[$ is the rigid tube of $T$ in $\mathscr{A}_{g,W}$), the image of $\mathfrak{X}$ in $]T[$ is again a Zariski closed analytic subspace. We call it $\mathfrak{T}$. 
\end{construction}

Mimicking 
\S\ref{subsub:HTofTcan}, 
we define $X^\circ\subseteq X$ as the preimage of $T^\circ$, which is étale over $T^\circ$.  If $x\in X(\Fpbar)$, we will identify $E^{/x}$ with $\mathscr{A}_{g,W}^{/\xi(x)}$, hence identifying $]x[$ with $]\xi(x)[$. Again, most results in \S\ref{sub:Toy} still remain valid:
\begin{manualtheoremc}{\ref{lm:TinT}**}\label{lm:TinT3} 
We have $(\widetilde{X^\circ})^{\mathrm{rig}} \subseteq \mathfrak{X}$. 
\end{manualtheoremc}
\begin{manualtheoremb}{\ref{thm:localstructureofT}**}
    \label{thm:localstructureofT3}
   Let $x\in X^{\circ}(\Fpbar)$, then each irreducible component of $\mathfrak{X}\cap ]x[$ is the translate of the rigid generic fiber of $(\widetilde{X^\circ})^{/x}$ by a torsion point. 
\end{manualtheoremb}
\begin{manualtheorema}{\ref{cor:containTcan}**}\label{cor:containTcan3}
    There is a unique irreducible component  $\mathfrak{X}_{\mathrm{can}}\subseteq \mathfrak{X}$ such that $]X^\circ[\cap \mathfrak{X}_{\mathrm{can}}$ contains $(\widetilde{X^\circ})^{\mathrm{rig}}$ as an irreducible component. 
\end{manualtheorema}

Let $\mathfrak{X}_{\mathrm{can}}\subseteq \mathfrak{X}$ be the unique irreducible component that contains $(\widetilde{X^\circ})^{\mathrm{rig}}$. Let $\widetilde{\mathfrak{X}}_{\mathrm{can}}$ be the normalization of $\mathfrak{X}_{\mathrm{can}}$ and let $\widetilde{\mathfrak{X}}_{\mathrm{can}}^{\mathrm{sm}}$ be the loci obtained by throwing out singularities. We have the following variant of Proposition~\ref{thm:petaleintegral}, whose proof is almost identical to the original version: 

\begin{manualtheoremd}{\ref{thm:petaleintegral}**}\label{thm:petaleintegral2}
Suppose that $X$ is $p$-integral at an $x\in X^\circ(\Fpbar)$. Then every $\overline{K}$-point $\mathfrak{z}\rightarrow \widetilde{\mathfrak{X}}_{\mathrm{can}}^{\mathrm{sm}}$ is $p$-étale integral.   
\end{manualtheoremd}

Using this, we get a variant of Proposition~\ref{thm:smooththeorem} whose proof we again leave to the reader:

\begin{manualtheoremd}{\ref{thm:smooththeorem}**}\label{thm:smooththeorem2}
Suppose that $X$ is $p$-integral at an $x\in X^\circ(\Fpbar)$. Then for every point $z\in X(\Fpbar)$, $X$ is $p$-integral at $z$, and
there exists a dimension $\dim T$ formal $W$-subscheme $\mathscr{T}_z\subseteq \mathscr{A}^{/\xi(z)}_{g,W}$
   linear in the Lie theoretic coordinates 
   such that $\mathscr{T}_z^{\mathrm{rig}}=\mathfrak{X}_{\mathrm{can}}\cap ]z[$ and $\mathscr{T}_z\otimes{\Fpbar}=X^{/z}$. In particular, $\mathscr{T}_z=(\widetilde{X^\circ})^{/z}$ when $z\in X^\circ(\Fpbar)$ and  $\mathfrak{X}_{\mathrm{can}}\cap ]X^\circ[ = (\widetilde{X^\circ})^{\mathrm{rig}}$.

%$\mathfrak{X}_{\mathrm{can}}\cap ]z[$ is the rigid generic fiber of a dimension $\dim T$ formal subscheme of $\mathscr{A}^{/\xi(z)}_{g,W}$ which is linear in the Lie theoretic coordinates. In particular,$\mathfrak{X}_{\mathrm{can}}\cap ]X^\circ[ = (\widetilde{X^\circ})^{\mathrm{rig}}$.
\end{manualtheoremd}
\subsubsection{Tate-linearity in the general case}
Finally, we are able to prove the following
\begin{theorem}\label{thm:TTlinearsmooth2}
    Suppose that $T\subseteq \mathscr{A}_{g,\Fpbar}$ is a proper Tate-linear subvariety. If $T$ is $p$-integral at an ordinary point $x\in T^\circ(\Fpbar)$, then $\mathrm{Tl}_p$ holds for $T$. 
\end{theorem}
\begin{proof}
    The proof is similar to Theorem~\ref{thm:TTlinearsmooth}. Let $V\subseteq T$ be an affine open, so that $V$ is cut out from $E^{/V}$ via an ideal $(p,f_{1},f_{2},...,f_{m})$. \begin{claim}
      There exists an $\epsilon< 1$, such that the $p$-adic norm $|f_{i}(\mathfrak{z})|$ is bounded above by $\epsilon$ for all $\overline{K}$-points $\mathfrak{z}$ in $\mathfrak{X}_{\mathrm{can}}\cap ]V[$ and all $i\in [1,m]$. 
  \end{claim}
The \textit{Claim} is proved exactly in the way as in Theorem~\ref{thm:TTlinearsmooth} with Proposition~\ref{thm:smooththeorem} replaced by Proposition~\ref{thm:smooththeorem2}. 

Now the \textit{Claim} shows that $\mathfrak{X}_{\mathrm{can}}$ is quasi-compact. On the other hand, 
$\mathfrak{X}_{\mathrm{can}}$ is a closed subspace of the partially proper rigid space $]X[=(E^{/X})^{\mathrm{rig}}$ (cf. \cite[Lemma 4.19]{Nearby_Cycle}), hence is partially proper. As a result, $\mathfrak{X}_{\mathrm{can}}$ is proper. As a result, the map $\mathfrak{X}_{\mathrm{can}}\rightarrow (\mathscr{A}_{g,W}^{\Sigma,\wedge p})^{\mathrm{rig}}$ is proper. Let $\mathfrak{T}_{\mathrm{can}}$ be the image, then it is algebraic by rigid GAGA. So $(\widetilde{T^\circ})^{\mathrm{rig}}$ is the restriction of the analytification of an algeberaic subvariety. A similar argument as in Lemma~\ref{lm:moonenlemma} then establishes the theorem.  
%this would imply the \textit{Claim} when $V\subseteq T^\circ$. Therefore it remains to show that for any $\Fpbar$-point $t\in T\setminus T^{\circ}$ and 

%Since $T$ is a curve, there are only finitely points $t_1,t_2,...,t_n$ in $T\setminus T^\circ$. If $t_i$ does not lie in the specialization of $\mathfrak{T}_{\mathrm{can}}$ then \textit{Claim} is vacuously true for a small Zariski neighborhood of $t_i$. If $t_i$ lies in the specialization of $\mathfrak{T}_{\mathrm{can}}$, then
%by Proposition~\ref{thm:smooththeorem} and its proof, $\mathfrak{T}_{\mathrm{can}}\cap ]t_i[$ is linear in the Lie theoretic coordinates and the mod $p$ reduction of $\Spf R_{G}$  equals $T^{/t_i}$. We see that the \textit{Claim} remains valid for a small Zariski neighborhood of $t_i$. The validity of the claim follows clearly from the above discussion.  

\end{proof}
\subsection{Proof of the theorems}
\begin{proof}[Proof of Theorem~\ref{thm:TTLpintegraldcs}] Let $T\subseteq \mathscr{A}_{g,\Fpbar}$ be a proper Tate-linear subvariety with unramified $p$-adic monodromy. The idea is to replace $T$ by a Hecke translate which is $p$-integral at an ordinary point.  

Let $x\in T(\Fpbar)$ be a smooth ordinary point. There exists a $\bZ_p$-lattices $H_0\subseteq \omega_x^{\mathrm{DM}}(\bH_{\cris})$ such that $G(\bH_{\cris,T^\circ},x)$ admits a reductive $\bZ_p$-model  $\cG\subseteq\GL(H_0)$ where $\mu_x$ extends. Let $H=H_0\otimes_{\bZ_p} W$, then $H$ is fixed by the Frobenius and Verschiebung $F,V$.  

%\footnote{Intersect $H_0$ with the filtration induced by $\mu_x$ and apply \cite[Proposition 1.1.4]{KM09} to split the filtration integrally. Conjugating $H_0$ and $\mathcal{G}$ over $\bQ_p$ so that the splitting is $\mu_x$, we can require that $\mu_x$ extends to $ \mathcal{G}$.} 

%We can assume that $\mu_x$ 

If $H$ is furthermore self-dual (with respect to the polarization), then it is a polarized Dieudonné lattice. There is some $n\gg 0$ such that  $p^nH\subseteq \omega_x^{\mathrm{iso}}(\bbH_{\cris})$. Then there exists a $y$ lying in the $p$-power Hecke orbit of $x$, such that $\omega_y^{\mathrm{iso}}(\bbH_{\cris})=p^nH$ as sublattices of $\omega_x^{\mathrm{iso}}(\bH_{\cris})$. Consider a $p$-power Hecke translate $T'$ of $T$ that contains $y$. Then $T'$ is again a proper Tate-linear subvariety. If $y$ is a smooth point of $T'$, then $T'$ is $p$-integral at $y$. Of course, it can happen that $T'$ is not smooth at $y$, but we can anyway take an alteration graph and use Proposition~\ref{prop:local_geometry} to show that $T'$ is $p$-integral at some other smooth ordinary point. Apply Theorem~\ref{thm:TTlinearsmooth2} to conclude that $\mathrm{Tl}_p$ holds for $T'$. It follows that $\mathrm{Tl}_p$ holds for $T$. 

If $H$ is not self-dual, we use a strategy similar to Zarhin's trick to reduce to the self-dual case. More precisely, consider $A_T^{4}$, which is four copies of the abelian schemes over $T$. This makes $T$ a proper Tate-linear subvariety in $\mathscr{A}_{g,\mathbb{F}}^4\subseteq \mathscr{A}_{4g,\mathbb{F}}$. On the monodromy side, the fiber lattice $\omega_x^{\mathrm{DM}}(\bH_{\cris}^4)$ becomes 4 copies of the original one, and the $p$-adic monodromy $G(\bH_{\cris,T^{\circ}}^4,x)$ is equal to $G(\bH_{\cris,T^{\circ}},x)$ sitting diagonally in $\omega_x^{\mathrm{DM}}(\bH_{\cris}^4)$. We will write $\psi$ for the symplectic pairing on $\omega_x^{\mathrm{DM}}(\bH_{\cris}^4)$ induced from polarization. Let $H'_0$ be the dual of $H_0$ under $\psi$. Then $H'_0$ is also $\mathcal{G}$-stable. In other words, $\mathcal{G}\subseteq \GL(H'_0)$. Consider the standard symmetric form $q$ on $\mathbf{U}:=\bZ_p^4$ with Gram matrix $I_4$. Since $p>2$, $\bZ_p^4$ decomposes into two mutually dual totally isotropic direct summands $\mathbf{U}_+$ and $\mathbf{U}_-$. Identify \[
\omega_x^{\mathrm{DM}}(\bH_{\cris}^4)=\omega_x^{\mathrm{DM}}(\bH_{\cris})\otimes \mathbf{U}.\]
We have $\psi^4=\psi\otimes q$. Let $$M_0=(H_0\otimes \mathbf{U}_+) \oplus (H'_0\otimes \mathbf{U}_-), \qquad M=M_0\otimes W.$$
Then $M_0$ is self-dual under $\psi^4$, and $M$ is stable under $F,V$. The reductive group $\mathcal{G}$ sits inside $\GL(M_0)$. Therefore, we have constructed a self-dual Dieudonné lattice $M\subseteq \omega_x^{\mathrm{DM}}(\bH_{\cris}^4)$ such that $\mathcal{G}\subseteq \GL(M)$. The rest of the proof is similar to the case where $H$ is self-dual: taken some sufficiently large $p^nM$ and use it to define a Hecke translate to get $T'$ which is $p$-integral at a smooth ordinary point, then apply Theorem~\ref{thm:TTlinearsmooth2} to conclude.

%\begin{claim}   There is a Dieudonné lattice $M\subseteq \omega_x^{\mathrm{iso}}(\bbH_{\cris}^4)$, self-dual under $\psi^4$, such that there is a reductive $W$-group $\mathcal{G}_1\subseteq\GL(M)$ with $\mathcal{G}_{1,K}\simeq G(\bH_{\cris,T^{\circ}}^4,x)^\circ\otimes K$.\end{claim}
\end{proof}
\begin{proof}[Proof of Corollary~\ref{cor:coomologyMT}] 
Up to a finite extension and isogeny, we can extend $A$ to a map $f:X\rightarrow \mathscr{A}_{4,\Fpbar}$. 
Let $C=\overline{\im f}$ be the proper Tate-linear curve in Lemma~\ref{lm:multislopes}. Theorem~\ref{thm:TTLpintegraldcs} implies that $C$ factors through a special curve. This implies that $G_{\mathrm{B}}(A)_{\overline{\bQ}}\subsetneq \GSp_{8,\overline{\bQ}}$. We conclude by Lemma~\ref{prop:mtreduction1} that $\MTT_p$ holds for $A$. 
\end{proof}

\bibliographystyle{alpha}
\bibliography{ref}

\end{document}